\documentclass[reqno]{amsart}
\usepackage{amsmath,amsthm,amscd,amssymb,amsfonts, amsbsy}
\usepackage{latexsym}
\usepackage{color}
\usepackage{enumerate}
\usepackage{pxfonts}
\usepackage{verbatim}
\usepackage{pgf}

\usepackage[utf8]{inputenc}
\usepackage[english]{babel}

\usepackage{geometry}
\usepackage{marginnote}

\usepackage{marginnote}
\usepackage{todonotes}
\usepackage{mathrsfs}
\usepackage{hyperref}
\usepackage{pdfsync}
\numberwithin{equation}{section}

\newtheorem{theorem}{Theorem}[section]

\newtheorem{lemma}[theorem]{Lemma}
\newtheorem{proposition}[theorem]{Proposition}

\theoremstyle{definition}
\newtheorem{remark}[theorem]{Remark}

\theoremstyle{definition}
\newtheorem{definition}[theorem]{Definition}

\theoremstyle{definition}

\newcommand{\dv}{\operatorname{div}}

\newcommand{\diam}{\operatorname{diam}}

\newcommand{\tr}{\operatorname{tr}}

\newcommand*{\tran}{^{\mkern-1.5mu\mathsf{T}}}

\newcommand{\bN}{\mathbb{N}}

\newcommand{\bR}{\mathbb{R}}

\newcommand{\bH}{\mathbb{H}}
\newcommand{\overbar}[1]{\mkern 1.5mu\overline{\mkern-1.5mu#1\mkern-1.5mu}\mkern 1.5mu}
\newcommand\cD{\mathcal{D}}

\newcommand\cH{\mathcal{H}}

\newcommand\cP{\mathcal{P}}

\newcommand\rC{\mathring{C}}

\providecommand{\set}[1]{\{#1\}}

\providecommand{\abs}[1]{\lvert#1\rvert}
\providecommand{\Abs}[1]{\left\lvert#1\right\rvert}

\providecommand{\norm}[1]{\lVert#1\rVert}

\renewcommand{\vec}[1]{\boldsymbol{#1}}

\makeatletter
\def\dashint{\operatorname%
{\,\,\text{\bf-}\kern-.98em\DOTSI\intop\ilimits@\!\!}}
\makeatother

\DeclareMathOperator*{\esssup}{ess\,sup}
\DeclareMathOperator*{\osc}{osc}

\begin{document}
\title[Estimates for parabolic operators]
{On $C^{1/2,1}$, $C^{1,2}$, and $C^{0,0}$ estimates for linear parabolic operators}

\author[H. Dong]{Hongjie Dong}
\address[H. Dong]{Division of Applied Mathematics, Brown University,
182 George Street, Providence, RI 02912, United States of America}
\email{Hongjie\_Dong@brown.edu}
\thanks{H. Dong was partially supported by the NSF under agreement DMS-1600593.}

\author[L. Escauriaza]{Luis Escauriaza}
\address[L. Escauriaza]{UPV/EHU, Dpto. Matem\'aticas, Barrio Sarriena s/n 48940 Leioa, Spain}
\email{luis.escauriaza@ehu.eus}
\thanks{L. Escauriaza is supported by Basque Government grant IT1247-19 and MICINN grant PGC2018-094522-B-I00.}

\author[S. Kim]{Seick Kim}
\address[S. Kim]{Department of Mathematics, Yonsei University, 50 Yonsei-ro, Seodaemun-gu, Seoul 03722, Republic of Korea}
\email{kimseick@yonsei.ac.kr}
\thanks{S. Kim is  supported by NRF Grant No. NRF-20151009350 and No. NRF-2019R1A2C2002724.}



\begin{abstract}
We show that weak solutions to parabolic equations in divergence form with zero Dirichlet boundary conditions are continuously differentiable up to the boundary when the leading coefficients have Dini mean oscillation and the lower order coefficients verify certain conditions.  Similar results are obtained for non-divergence form parabolic operators and their adjoint operators. Under similar conditions, we also establish a Harnack inequality for nonnegative adjoint solutions, together with upper and lower Gaussian bounds for the global fundamental solution.
\end{abstract}
%\today
\maketitle

\section{Introduction}
This paper is a continuation of our previous work \cite{DK17, DEK18} in which we obtained interior and boundary $C^1$, $C^2$, and $C^0$ estimates for divergence form, non-divergence form, and the corresponding adjoint elliptic operators under the assumption that the coefficients and data have Dini mean oscillation. These work were motivated by a question raised by Yanyan Li \cite{Y.Li2016} about divergence form elliptic equations. In this paper, we consider the corresponding parabolic operators under the condition that the coefficients and data have Dini mean oscillation with respect to either the space variables or all the variables, and establish both interior and boundary estimates.

Let $\Omega_T=(0,T)\times \Omega \subset \bR^{n+1}$ be a cylindrical domain, where $\Omega$ is a bounded domain in $\bR^n$, with $n \ge 1$.
We consider a second order parabolic operator $P$ in divergence form
\begin{equation}\label{master-d}
P u= \partial_t u - \sum_{i,j=1}^n D_i(a^{ij}(t,x) D_j u + b^i(t,x) u) + \sum_{i=1}^n c^i(t,x) D_iu + d(t,x)u
\end{equation}
and also consider parabolic operators $\cP$ in non-divergence form
\begin{equation}\label{master-nd}
\cP u= \partial_t u-\sum_{i,j=1}^n a^{ij}(t,x) D_{ij} u+ \sum_{i=1}^n b^i(t,x) D_i u + c(t,x)u.
\end{equation}
Here, the coefficients $\mathbf{A}=(a^{ij})_{i,j=1}^n$, $\vec b=(b^1,\ldots, b^n)$, $\vec c=(c^1,\ldots, c^n)$, $d$  and $c$ are measurable functions defined on $\overline \Omega_T$.
We assume that the principal coefficients $\mathbf{A}=(a^{ij})$ are defined on $\bR^{n+1}$ and satisfy the uniform parabolicity condition
\begin{multline}\label{parabolicity}
\lambda \abs{\vec \xi}^2 \le \sum_{i,j=1}^n a^{ij}(t,x) \xi^i \xi^j, \quad \Abs{\sum_{i,j=1}^n a^{ij}(t,x) \xi^i \eta^j }\le \lambda^{-1}\abs{\vec \xi} \abs{\vec \eta},\\
\forall  \vec \xi=(\xi^1,\ldots, \xi^n), \;\;\forall \vec \eta =(\eta^1,\ldots, \eta^n) \in \bR^n,\quad\forall (t,x) \in \bR^{n+1}
\end{multline}
for some positive constant $\lambda$.

For the non-divergence form operator, we shall further assume that  the coefficients $\mathbf{A}$ are symmetric, i.e. $a^{ij}=a^{ji}$.

Throughout the paper, we shall use $X = (t, x)$ to denote a point in $\bR^{n+1} = \bR\times \bR^n$; $x = (x^1,\ldots, x^n)$ will always be a point in $\bR^n$.
We also write $Y = (s, y)$, $X_0 = (t_0, x_0)$, $Z = (\tau,z)$, etc.
We define the parabolic distance between the points $X = (t,x)$ and $Y = (s,y)$ in $\bR^{n+1}$ as (by abuse of notation)
\[
\abs{X-Y}=\max(\sqrt{\abs{t-s}}, \abs{x-y}).
\]
We say that a non-negative measurable function $\omega: (0,1] \to \bR$ is a Dini function provided that there are constants $c_1, c_2 >0$ such that
\begin{equation}
                                \label{eq9.36}
c_1 \omega(t) \le \omega(s) \le c_2 \omega(t)
\end{equation}
whenever $\tfrac12 t \le s \le  t$ and $0<t<1$ and
\[
\int_0^1 \frac{\omega(s)}s \,ds <+\infty,\quad \forall t \in (0,1].
\]
For a domain $\Omega$ in $\bR^n$, we shall write
\begin{equation*}%\label{E:notacion1}
\Omega_r(x)=\Omega \cap B_r(x)
\end{equation*}
and
\begin{equation*}	%\label{E:notacion2}
Q_r^-(X)=Q_r^-(t,x)= (t-r^2,t]\times \Omega_r(x).
\end{equation*}
For a locally integrable function $g$ on $\Omega_T=(0,T)\times \Omega$, we shall say that $g$ is uniformly Dini continuous if the function $\varrho_g: \bR_+ \to \bR$ defined by
\[
\varrho_g(r):=\sup\set{\abs{g(Y)-g(Y')}: Y, Y' \in Q_r^-(X) \subset \Omega_T}
\]
is a Dini function, while we shall say that $g$ is of \emph{Dini mean oscillation} over $\Omega_T$ and write $g \in \mathsf{DMO}$ if the function $\omega_g: \bR_+ \to \bR$ defined by
\begin{equation*}
\omega_g(r):=\sup_{Q_r^-(X)\subset \overline \Omega_T} \fint_{Q_r^-(X)} \,\abs{g(Y)- g_{X,r}}, \quad\ \left(\; g_{X,r} :=\fint_{Q_r^{-}(X)} g\;\right)
\end{equation*}
is a Dini function.
We shall say that $g$ is of \emph{Dini mean oscillation in $x$} over $\Omega_T$ and write $g \in \mathsf{DMO_x}$ if the function $\omega_g^{\textsf x}: \bR_+ \to \bR$ defined by
\begin{equation}\label{13.24f}
\omega_g^{\textsf x}(r):=\sup_{Q_r^-(X) \subset \overbar \Omega_T }\fint_{Q_r^-(X)} \,\Abs{g(s,y)- \bar g^{\textsf x}_{x,r}(s)},\quad\left(\;\bar g^{\textsf x}_{x,r}(s) :=\fint_{\Omega_r(x)} g(s,\cdot)\;\right)
\end{equation}
is a Dini function.
In view of the proof on \cite[p. 495]{Y.Li2016}, we know that both $\omega_g$ and $\omega_g^{\textsf x}$ satisfy \eqref{eq9.36}.

The main theorems of this paper are as follows.
We ask reader to refer to Section \ref{sec2} for the notations related  to the sets $\partial_p^-\Omega_T$, $\partial_p ^+\Omega_T$ and the function spaces such as $\cH^1_2$, $\mathring{\mathcal H}^1_{p}$, $W^{1,2}_2$, $\rC^{1/2,1}$, etc.

\begin{theorem}\label{thm-main-d}
Let $q>n+2$, $\Omega$ have $C^{1, Dini}$ boundary, and the coefficients of $P$ in \eqref{master-d} satisfy the following conditions in addition to \eqref{parabolicity}:
$\mathbf{A} \in \mathsf{DMO_x}$ and $\vec b \in \mathsf{DMO_x} \cap  L^\infty$ over $\Omega_T$, $\vec c \in L^{q}(\Omega_T)$, and $d \in L^{q} (\Omega_T)$.
Let $u \in \cH^1_2(\Omega_T)$ be a weak solution of
\[
P u=\dv  \vec g + f\;\mbox{ in } \; \Omega_T,\quad u=0\;\mbox{ on }\;\partial_p ^-\Omega_T,
\]
where $\vec g=(g^1,\ldots, g^n) \in \mathsf{DMO_x} \cap L^\infty(\Omega_T)$  and $f \in L^q(\Omega_T)$.
Then, we have $u\in \rC^{1/2,1}(\overbar{\Omega_T})$.
\end{theorem}

\begin{theorem}					\label{thm-main-nd}
Let $\Omega$ have $C^{2, Dini}$ boundary and the coefficients of $\cP$ in \eqref{master-nd} satisfy the following conditions in addition to \eqref{parabolicity}:
$\mathbf{A}\in \mathsf{DMO_x}$, $\vec b \in \mathsf{DMO_x} \cap L^\infty$, and $c \in \mathsf{DMO_x}\cap L^\infty$ over $\Omega_T$.
Let $u \in W^{1,2}_2(\Omega_T)$ be the strong solution of
\[
\cP u= g\;\mbox{ in } \; \Omega_T,\quad u=0\;\mbox{ on }\;\partial_p^-\Omega_T,
\]
where $g \in \mathsf{DMO_x}\cap L^\infty(\Omega_T)$.
Then, $D^2 u\in C^{0,0}(\overbar{\Omega'_{T}})$ for any $\Omega'\subset\subset \Omega$.
Moreover, if the even extensions of $\mathbf A$, $\vec b$, and $c$ and the $0$ extension of $g$ are in $\mathsf{DMO}$ over $(-T,T)\times\Omega$  , then $u\in C^{1,2}(\overbar {\Omega_T})$.
\end{theorem}

We also deal with the adjoint boundary value problem
\begin{equation}\label{master-adj-prob}
\mathcal P^\ast u=\nabla^2\mathbf g +f\; \text{ in }\; \Omega_T,\quad u= -\frac{\mathbf{g}\nu\cdot \nu}{\mathbf{A}\nu\cdot \nu}+\psi \; \text{ on }\; (0,T) \times \partial \Omega, \quad u(T)=\varphi\; \text{on}\;  \Omega,
\end{equation}
where $\nu$ is the exterior unit normal vector to $\partial\Omega$, $\mathbf g = (g^{kl})_{k,l=1}^n$ is a symmetric matrix
\[
\nabla^2\mathbf g=\sum_{k,l=1}^n D_{kl}g^{kl}
\]
and $\mathcal P$ is the formal adjoint operator of $\mathcal P$, i.e.,
\begin{equation*}		%\label{master-adj}
\mathcal P^\ast u = -\partial_t u- \sum_{i,j=1}^n D_{ij}(a^{ij}(t,x) u)-D_i(b^i(t,x) u)+c(t,x)u.
\end{equation*}
The appearance of the term $\mathbf{g}\nu\cdot \nu/ \mathbf{A}\nu\cdot \nu$ as a part of boundary values helps to make $\mathbf{g}$ to disappear from the boundary integral in the identity \eqref{eq13.53pde}, which formally defines a ``weak'' adjoint solution to \eqref{master-adj-prob}.

\begin{definition}%\label{D: solucionadju}
Let $\Omega \subset \bR^n$ be a bounded $C^{1,1}$ domain.
Assume that  $\mathbf{g} \in L^p(\Omega_T)$, $f \in L^p(\Omega_T)$, $\varphi\in L^p(\Omega)$ and $\psi\in L^p((0,T)\times\partial\Omega)$, where $1<p<\infty$.
We say that $u \in L^p(\Omega_T)$ is an adjoint solution to \eqref{master-adj-prob} if $u$ satisfies
\begin{equation}\label{eq13.53pde}
\int_{\Omega_T} u\,\mathcal P v=\int_{\Omega_T} f v+\tr(\mathbf{g}\,D^2 v)+\int_\Omega \varphi v(T)-\int_{[0,T]\times\partial\Omega}\mathbf A\nabla v\cdot\nu\psi
\end{equation}
for any $v \in W^{1,2}_{p'}(\Omega_T)\cap \mathring{\mathcal H}^1_{p'}(\Omega_T)$, with $\frac{1}{p}+\frac{1}{p'}=1$.
\end{definition}
The existence and uniqueness of the weak adjoint solution to \eqref{master-adj-prob} is simple to derive by transposition from the unique existence of a solution $v \in W^{1,2}_{p'}(\Omega_T)\cap \mathring{\mathcal H}^1_{p'}(\Omega_T)$ to the direct problem
\begin{equation*}%\label{E:3 direct problem}
\begin{cases}
\mathcal Pv = F\ &\text{in}\ \Omega_T,\\
u = 0\ &\text{in}\ \partial_p^-\Omega_T,\\
\end{cases}
\end{equation*}
and the $L^{p'}(\Omega_T)$-estimate
\begin{equation*}
\|D^2v\|_{L^{p'}(\Omega_T)}+\|\partial_tv\|_{L^{p'}(\Omega_T)}+\|v\|_{L^\infty_tL^{p'}_x(\Omega_T)}\lesssim\|F\|_{L^{p'}(\Omega_T)},
\end{equation*}
which holds when the matrix of coefficients $\mathbf A$ is sufficiently regular (the continuity of $\mathbf A$ in $\bR^{n+1}$ or other weaker conditions suffice). See the analog construction for non-divergence form elliptic equations in \cite[Lemma 2]{EM2016}.

\begin{theorem}\label{thm-main-adj}
Let $q>n+2$, $\Omega$ have a $C^{1,1}$ boundary, the coefficients of $\cP$ in \eqref{master-nd} satisfy the following conditions in addition to \eqref{parabolicity}:
$\mathbf{A} \in \mathsf{DMO_x}$ over $\bR^{n+1}$, $\vec b \in L^q(\Omega_T)$, $c \in L^{q/2}(\Omega_T)$.
Let $u \in L^2(\Omega_T)$ be the adjoint solution of the problem \eqref{master-adj-prob},
where $\mathbf{g} \in \mathsf{DMO_x}\cap L^\infty(\Omega_T)$,  $f\in L^{q/2}(\Omega_T)$, $\varphi\in L^q(\Omega)$ and $\psi\in L^q((0,T)\times\partial\Omega)$.
Then, $u\in C^{0,0}(\overbar{\Omega'_{T'}})$ for any $\Omega'\subset\subset \Omega$ and $0<T'<T$.
 Moreover, if $\mathbf{A}$ is $\mathsf{DMO}$ over $\bR^{n+1}$, the zero extension of $\mathbf{g}$ for $t\ge T$ is in  $\mathsf{DMO}$ over $(0,2T)\times\Omega$ and the pair $\varphi$ and $\psi$ defines a continuous function over $\partial_p^+\Omega$, then $u\in C^{0,0}(\overbar{\Omega_T})$.
\end{theorem}

As in \cite{DK17,DEK18}, the proofs of Theorems \ref{thm-main-d}, \ref{thm-main-nd}, and \ref{thm-main-adj} are based on Campanato's approach, which was used previously, for instance, in \cite{Giaq83, Lieb87}. The main step of Campanato's approach is to show the mean oscillations of $Du$ (or $D^2u$, or $u$, respectively) in balls (or cylinders) vanishes to a certain order as the radii of the balls (or cylinders) go to zero.
The main difficulty is that because we only impose the assumption on the $L^1$-mean oscillation of the coefficients and data with respect to either $x$ or $(t,x)$, the usual argument based on the $L^2$ (or $L^p$ for $p>1$) estimates does not work in our case. To this end, we exploit weak type-$(1,1)$ estimates, the proof of which involves a duality argument, as well as the Sobolev estimates for parabolic equations with coefficients measurable in time; see \cite{DK11,DK11b}. We then adapt Campanato's idea in the $L^p$ setting for some $p\in (0,1)$.
In order to derive Theorem \ref{thm-main-adj} we must also establish new results for non-negative adjoint solutions to homogeneous non-divergence form elliptic equations, as Lemmas \ref{lem8.06} and \ref{lem8.0671}. See also the Remark \ref{R:remrk10} for an application of those results to derive upper and lower Gaussian bounds for the global fundamental solution of such operators.

\begin{remark}\label{rmk1.5}
Elliptic and parabolic equations with uniformly Dini continuous coefficients have been well studied; see, for instance, \cite{ME65,Lieb87,MM11, BS17}.
In particular, in \cite{BS17} the authors obtained the continuity of adjoint solutions and the Harnack inequality for nonnegative adjoint solutions under the assumption that the leading coefficients are uniformly Dini continuous. See also \cite{S1975,Mamedov92, BKR2001} for related results.
Elliptic equations with Dini mean oscillation coefficients were recently studied in \cite{DK17}.
By H\"older's inequality, the $L^1$-Dini mean oscillation (or  $\mathsf{DMO}$) condition is weaker than the $L^p$-Dini mean oscillation condition for any $p>1$, i.e., the function
\[
\omega_{(p)}(r):=\sup_{Q_r^-(X)\subset \overline \Omega_T} \left( \fint_{Q_r^-(X)} \,\abs{f(Y)- f_{X,r}}^p\right)^{\frac1 p}, \quad\text{where }\; f_{X,r} :=\fint_{Q_r^{-}(X)} f.
\]
is a Dini function.
For example, the $L^2$-Dini mean oscillation condition was used in \cite{Y.Li2016}\footnote{This paper was written and first submitted in 2008.} and also in \cite{KM2012, Dong2012}.
These conditions are in fact strictly weaker than the uniform Dini continuity condition; see an example in \cite[p. 418]{DK17}.
On the other hand, the situation is different if we instead consider the functions
\[
\widehat\omega_{(p)}(r):= \sup_{0<s\le r} \omega_{(p)}(s).
\]
It follows from the proof of \cite[Proposition~1.13]{Ac92} that if $\widehat\omega_{(1)}$ is a Dini function, then $\widehat \omega_{(p)}$ are also Dini functions for all $p \in [1, \infty)$.
It is clear that if $\widehat \omega_{(1)}$ is a Dini function, then $\omega_{(1)}$ is also a Dini function.
However, as it is mentioned in \cite[Remark~2.2]{DL20} that the converse is not always true as there is a function $f$ for which $\omega_{(1)}$ is a Dini condition, while $\widehat \omega_{(1)}$ is not.
From this perspective, it is not clear to us whether our $\mathsf{DMO}$ condition implies that $\omega_{(p)}$ is a Dini function for $p \in (1,\infty)$.
\end{remark}

The paper is organized as follows.
In Section~\ref{sec2} we introduce some notation, and function spaces, and provide some preliminary lemmas. Section \ref{sec:int} is devoted to interior estimates for three types of equations.
Boundary estimates for three types of equations are established in Section~\ref{sec:bdry}, where we complete the proofs of Theorems~\ref{thm-main-d}, \ref{thm-main-nd}, and \ref{thm-main-adj}.
Section~\ref{sec5} is devoted to the study of adjoint and normalized adjoint solutions.
We present some useful properties of adjoint solutions to equations with $\mathsf{DMO_x}$ coefficients, some of which are used  to derive Theorem \ref{thm-main-adj}.
See also Remark \ref{R:remrk10} for an application to get upper and lower Gaussian bounds for the fundamental solution of parabolic equations in non-divergence form with $\mathsf{DMO_x}$ leading coefficients over $\bR^{n+1}$.

Finally, we mention that most of our main results are readily extended to parabolic systems of second order with the operators  $P$ and $\mathcal P$ replaced respectively by
\begin{equation*}
P \vec u= \vec u_t-D_\alpha(\mathbf{A}^{\alpha\beta}D_\beta \vec u+\hat{\mathbf B}^\alpha \vec u)+\mathbf{B}^\alpha D_\alpha \vec u+\mathbf{C}\vec u
\end{equation*}
and
\begin{equation*}
\cP \vec u = \vec u_t - \mathbf{A}^{\alpha \beta}
D_{\alpha\beta}\vec u + \mathbf{B}^{\alpha} D_\alpha \vec u + \mathbf{C} \vec u.
\end{equation*}
Here the coefficients $\mathbf{A}^{\alpha \beta}$, $\mathbf{B}^\alpha$,
$\hat {\mathbf B}^{\alpha}$, and $\mathbf C$ are $m \times m$ matrix-valued functions given on $\bR^{n+1}$, i.e.,
$\mathbf{A}^{\alpha\beta}=(A^{\alpha\beta}_{ij}(t,x))_{i,j=1}^m$, etc., and the leading coefficients satisfies the Legendre-Hadamard condition
\[
A_{ij}^{\alpha \beta}(t,x) \xi_{\alpha} \xi_{\beta} \vartheta^{i} \vartheta^{j}
\ge \lambda \abs{\xi}^2 \abs{\vartheta}^2
\]
for all $(t,x) \in \bR^{n+1}$, $\xi\in \bR^n$, $\vartheta\in \bR^m$, for some constant $\lambda>0$.
We remark that the above condition is also satisfied by the linear systems of elasticity.
More precisely, Theorems~\ref{thm-main-d} and \ref{thm-main-nd} remain valid for the corresponding parabolic systems of second order, while Theorem~\ref{thm-main-adj} can be extended to the strongly parabolic systems in the case when $\psi$ and $\varphi$ are identically zero.
This is because we only use the scalar structure of the solutions in the proof of Proposition~\ref{thm-main-adj2}, which uses the properties of scalar adjoint solutions developed in Section~\ref{sec5}.

\section{Notation and preliminaries}					\label{sec2}
\subsection{Basic notation}
We denote the standard forward parabolic cylinder as
\begin{equation*}%\label{eq1612sat}
C_r^-(X)=C_r^-(t,x)= (t-r^2, t] \times B_r(x),
\end{equation*}
where $B_r(x)$ denotes the standard Euclidean ball in $\bR^n$.
When we deal with adjoint  equations, we use instead backward cylinders
\begin{equation*}%\label{eq1612sat+}
C_r^+(X)=C_r^+(t,x)= [t, t+r^2) \times B_r(x).
\end{equation*}
We also denote the double centered cylinder as
\begin{equation*}%\label{E: notation3}
C_r(X)=(t-r^2, t+r^2) \times B_r(x).
\end{equation*}
For a cylindrical domain $\mathcal D=(a,b)\times \Omega$, its forward parabolic boundary is defined by
\[
\partial_p^-\mathcal D= (a,b) \times \partial \Omega \cup \set{a}\times \overbar\Omega,
\]
while its backward parabolic boundary as
\[
\partial_p^+\mathcal D= (a,b) \times \partial \Omega \cup \set{b}\times \overbar\Omega.
\]
Finally, in analogy with previous notation, we denote
\begin{equation*}
Q_r^+(X)=Q_r^+(t,x)= [t,t+r^2)\times \Omega_r(x),\quad Q_r(X)=Q_r(t,x)= (t-r^2,t+r^2)\times \Omega_r(x)
\end{equation*}
and $C_r^-$, $C_r^+$, $C_r$, $Q_r^-$, $Q_r^+$ and $Q_r$ denote respectively the same sets but when their center is $(0,0)$.
\subsection{Function spaces}

For any domain $\mathcal D\subset \bR^{n+1}$ and $p\in [1,\infty]$, we shall denote by
$L^p(\mathcal D)$ the standard Lebesgue class, i.e., the set of all functions for which
\[
\norm{f}_{L^p(\mathcal D)}=\left(\int_\mathcal D \abs{f}^p\right)^{1/p}<\infty,\quad \norm{f}_{L^\infty(\mathcal D)}=\esssup_\mathcal D \,\abs{f}<\infty.
\]
We define the function spaces
\begin{align*}
W_{p}^{1,2}(\mathcal D)&=
\set{u:\,u, \, \partial_t u,\, Du,\, D^2u\in L^p(\mathcal D)},\\
\bH^{-1}_{p}(\mathcal D)&=
\set{u:\,u=\dv \vec f+g,\;\text{ for some }\;\vec f,\,g\in L^p(\mathcal D)},\\
\cH_{p}^{1}(\mathcal D)&=
\set{u:\,u, \,Du\in L^p(\mathcal D),\;\partial_t u \in \bH^{-1}_{p}(\mathcal D)},
\end{align*}
which are equipped with norms
\begin{align*}
\|u\|_{W_{p}^{1,2}(\mathcal D)}&=\norm{u}_{L^p(\mathcal D)}+\norm{Du}_{L^p(\mathcal D)}+\norm{D^2u}_{L^p(\mathcal D)}+\norm{\partial_t u}_{L^p(\mathcal D)},\\
\norm{u}_{\bH^{-1}_{p}(\mathcal D)}&=\inf\set{ \norm{\vec f}_{L^p(\mathcal D)}+\norm{g}_{L^p(\mathcal D)}:\,
u=\dv \vec f+g,\; \vec f,\,g \in L^p(\mathcal D)},\\
\norm{u}_{\cH_{p}^{1}(\mathcal D)}&=\norm{u}_{L^p(\mathcal D)}+\norm{Du}_{L^p(\mathcal D)}+\norm{\partial_t u}_{\bH^{-1}_p(\mathcal D)}.
\end{align*}
We denote by $C^{0,0}(\overbar{\mathcal D})$ the space of all continuous functions over $\overbar{\mathcal D}$ and define
\[
C^{1,2}(\overbar{\mathcal D})= \set{u \in C^{0,0}(\overbar{\mathcal D}):  \partial_t u,\, Du, \,D^2u \in C^{0,0}(\overbar{\mathcal D})}.
\]
For a constant $\delta \in (0,1]$, we denote
\[
[u]_{C^{\delta/2, \delta}(\mathcal D)}=\sup_{\substack{X, Y \in \mathcal D\\X \neq Y}} \frac{\abs{u(X)-u(Y)}}{\abs{X-Y}^\delta};\quad \norm{u}_{C^{\delta/2, \delta}(\mathcal D)}=[u]_{C^{\delta/2, \delta}(\mathcal D)}+\sup_{\mathcal D}\,\abs{u}.
\]
By $C^{\delta/2, \delta}(\mathcal D)$ we denote the space of all functions $u$ for which $\norm{u}_{C^{\delta/2, \delta}(\mathcal D)}<\infty$.
$\mathring C^{1/2,1}(\overbar{\mathcal D})$ is the set of all functions $u \in  C^{1/2,1}(\overbar{\mathcal D})$ for which $Du \in C^{0,0}(\overbar{\mathcal D})$ and
\begin{equation}\label{eq1121may9}
\frac{\abs{u(t,x)-u(s,x)}}{\abs{t-s}^{1/2}} \to 0\;\text{ as }\;\abs{t-s}\to 0\;\text{ for }\;(t,x), \,(s,x) \in \mathcal{\overbar{\mathcal D}}.
\end{equation}

Finally, $\mathring{\mathcal H}^1_{p}(\Omega_T)$ and $\mathring{W}^{1,2}_{p}(\Omega_T)$ denote respectively the closure of $ C^{1,2}(\overbar{\Omega}_T)$ functions in $\mathcal H^1_{p}(\Omega_T)$ and $W^{1,2}_p(\Omega_T)$, which vanish over $\partial_p^{-}\Omega_T$.

\subsection{Preliminary lemmas}
First, we state Sobolev-Morrey embedding theorems in the parabolic setting.
The first one is a special case of \cite[Lemma~3.3, \S II.3, p. 80]{LSU}.
For the second lemma, we refer the reader to \cite[Lemma 8.1]{Kr07b}.
\begin{lemma}				\label{psobolev}
Let $\Omega\subset \bR^n$ be a bounded Lipschitz domain.
Assume $u \in W^{1,2}_q(\Omega_T)$.
\begin{enumerate}[(i)]
\item
If $1\le q<\frac{n+2}{2}$, then $u \in L^p(\Omega_T)$, where $\frac{1}{p}=\frac{1}{q}-\frac{2}{n+2}$, and we have the estimate
\[
\norm{u}_{L^p(\Omega_T)} \le C \norm{u}_{W^{1,2}_q(\Omega_T)}.
\]
\item
If $q>\frac{n+2}{2}$, then $u \in C^{\alpha/2,\alpha}(\Omega_T)$, where
\[
\alpha=
\begin{cases}
2-\frac{n+2}{q}, &\text{ if }\; q<n+2\\
1-\epsilon \;\text{ for any }\;\epsilon\in (0,1), &\text{ if }\;q \ge n+2.
\end{cases}
\]
We have in addition the estimate
\[
\norm{u}_{C^{\alpha/2,\alpha}(\Omega_T)} \le C \norm{u}_{W^{1,2}_q(\Omega_T)}.
\]
\item
If $1\le q<n+2$, then $Du \in L^p(\Omega_T)$, where $\frac{1}{p}=\frac{1}{q}-\frac{1}{n+2}$, and we have the estimate
\[
\norm{Du}_{L^p(\Omega_T)} \le C \norm{u}_{W^{1,2}_q(\Omega_T)}.
\]
\item
If $q>n+2$, then $Du \in C^{\alpha/2,\alpha}(\Omega_T)$, where $\alpha=1-\frac{n+2}{q}$, and
we have  the estimate
\[
\norm{Du}_{C^{\alpha/2,\alpha}(\Omega_T)} \le C \norm{u}_{W^{1,2}_q(\Omega_T)}.
\]
\end{enumerate}
Here, $C$ is a constant (varying line from line) depending only on $n$, $q$, $\Omega$, and $T$.
\end{lemma}
\begin{lemma}				\label{psobolev2}
Let $\Omega\subset \bR^d$ be a bounded Lipschitz domain.
Assume $u \in \cH^1_q(\Omega_T)$.
\begin{enumerate}[(i)]
\item
If $1\le q<n+2$, then $u \in L^p(\Omega_T)$, where $\frac{1}{p}=\frac{1}{q}-\frac{1}{n+2}$, and we have the estimate
\[
\norm{u}_{L^p(\Omega_T)} \le C \norm{u}_{\cH^1_q(\Omega_T)}.
\]
\item
If $q>n+2$, then $u \in C^{\alpha/2,\alpha}(\Omega_T)$, where $\alpha=1-\frac{n+2}{q}$, and
we have  the estimate
\[
\norm{u}_{C^{\alpha/2,\alpha}(\Omega_T)} \le C \norm{u}_{\cH^1_q(\Omega_T)}.
\]
\end{enumerate}
Here, $C$ is a constant (varying line from line) depending only on $n$, $q$, $\Omega$, and $T$.
\end{lemma}

Throughout the rest of paper, the usual summation convention over repeated indices are assumed. For non-negative (variable) quantities $A$ and $B$,
we denote $A\lesssim B$ if there exists a generic positive constant C such that $A \le CB$.
We may add subscript letters like $A\lesssim_{a,b} B$ to indicate the dependence of the implicit constant $C$ on the parameters $a$ and $b$.

\section{Interior estimates}					\label{sec:int}
In this section, we consider parabolic equations without lower order coefficients and develop interior estimates.
More specifically, we shall prove the following theorems, where the usual summation convention with repeated indices are assumed.
\begin{theorem}\label{thm:d-wolt}
Let $a^{ij}$ satisfy \eqref{parabolicity}.
Let $u \in \cH^1_2(C_4^-)$ be a weak solution of the divergence equation
\begin{equation*}				%\label{d-wolt}
u_t- D_i(a^{ij}(t,x) D_j u)= D_i g^i \;\text{ in }\; C_4^-.
\end{equation*}
If $a^{ij} \in \mathsf{DMO_x}$ and  $g^i \in \mathsf{DMO_x} \cap L^\infty$, then we have $u\in \mathring{C}^{1/2,1}(\overbar{C_1^-})$.
\end{theorem}
\begin{theorem}\label{thm:nd-wolt}
Let $a^{ij}$ be symmetric and satisfy \eqref{parabolicity}.
Let $u \in W^{1,2}_2(C_4^-)$ be a strong solution of the non-divergence equation
\begin{equation*}				%\label{nd-wolt}
u_t- a^{ij}(t,x) D_{ij}u = g \;\text{ in }\;C_4^-.
\end{equation*}
If $a^{ij} \in \mathsf{DMO_x}$ and  $g \in \mathsf{DMO_x} \cap L^\infty$, then we have $D^2u \in C^{0,0}(\overbar{C_1^-})$ and $\partial_t u \in L^\infty({C_1^-})$.
Moreover, if $a^{ij}$ and  $g$ are continuous, then we have $u \in C^{1,2}(\overbar{C_1^-})$.
\end{theorem}

\begin{theorem}\label{thm:dd-wolt}
Let $a^{ij}$ be symmetric and satisfy \eqref{parabolicity}.
Let $u \in L^2(C_4^+)$ be an adjoint solution of the double divergence form equation
\begin{equation*}				%\label{dd-wolt}
-u_t- D_{ij}(a^{ij}(t,x) u) = D_{ij}g^{ij}\;\text{ in }\; C_4^+.
\end{equation*}
Assume $a^{ij} \in \mathsf{DMO_x}$ and  $g^{ij} \in \mathsf{DMO_x} \cap L^\infty$.
Then, we have $u\in C^{0,0}(\overbar{C_1^+})$.
\end{theorem}
\subsection{Preliminary lemmas}
It should be clear that if $g$ is uniformly Dini continuous, then it is of Dini mean oscillation and $\omega_g(r) \le \varrho_g(r)$.
It is worthwhile to note that if $\Omega$ is such that for any $x\in \overline \Omega$,
\begin{equation}					\label{cond_a}
\abs{\Omega_r(x)} \ge A_0 r^n,\quad \forall r\in (0, \diam \Omega]\quad (\text{$A_0$ is a positive constant})
\end{equation}
and if $g$ is of Dini mean oscillation, then $g$ is uniformly continuous with a modulus of continuity determined by $\omega_g$.
\begin{lemma}				\label{lem00}
Let $\Omega$ satisfy the condition \eqref{cond_a}.
If $f$ is uniformly Dini continuous and $g$ is $\mathsf{DMO_x} \cap L^\infty(\Omega_T)$,  then $fg$ is $\mathsf{DMO_x} \cap L^\infty$  in $\Omega_T$.
\end{lemma}
\begin{proof}
For any $X=(t,x)\in \overbar\Omega_T$, $r>0$ and $Q_r^-(X)\subset\overbar\Omega_T$, we have, for $s \in (t-r^2,t)$ that
\begin{multline*}
\fint_{\Omega_r(x)} \Abs{fg(s,y)-\overbar{\left(fg\right)}^{\textsf x}_{x,r}(s)}dy
\le \fint_{\Omega_r(x)} \Abs{fg-f\,\bar g^{\textsf x}_{x,r}(s)} +\Abs{f \bar g^{\textsf x}_{x,r}(s)-\overbar{\left(fg\right)}^{\textsf x}_{x,r}(s)} dy\\
\le \sup_{Q_r^-(X)} \abs{f} \cdot \fint_{\Omega_r(x)} \abs{g(s,y)-\bar g^{\textsf x}_{x,r}(s)}dy + \varrho_f(r) \cdot \fint_{\Omega_r(x)} \abs{g(s,y)}dy,
\end{multline*}
where we used the fact that for $s \in(t-r^2,t)$, we have
\[
\fint_{\Omega_r(x)} \Abs{f(s,y) g(s,y)-f(s,y)\,\bar g^{\textsf x}_{x,r}(s)} \le \sup_{Q_r(X)} \abs{f} \cdot \fint_{\Omega_r(x)} \Abs{g(s,y)-\bar g^{\textsf x}_{x,r}(s)}dy
\]
and
\[
\Abs{f(s,y) \bar g^{\textsf x}_{x,r}(s)-\overbar{\left(fg\right)}^{\textsf x}_{x,r}(s)} \le  \Abs{\fint_{\Omega_r(x)} (f(s,y)-f(s,z))g(s,z)dz }
\le \varrho_f(r) \cdot \fint_{\Omega_r(x)} \abs{g(s,z)}dz.
\]
Therefore,
\begin{equation}				\label{eq11.39w}
\omega_{fg}^{\textsf x}(r) \le \norm{f}_{L^\infty(\Omega_T)}\, \omega_g^{\textsf x}(r)+\norm{g}_{L^\infty(\Omega_T)}\,\varrho_f(r)
\end{equation}
and thus $\omega_{fg}^{\textsf x}$ is a Dini function.
It is obvious that $fg \in L^\infty$.
\end{proof}
\begin{lemma}			\label{lem01-stein}
Let $T$ be a bounded linear operator from $L^2(C_1^-)$ to $L^2(C_1^-)$.
Suppose there are $c>1$ and $C>0$ such that for any $Y \in C_1^-$ and $0<r<\frac12$ we have
\[
\int_{C_1^- \setminus C_{cr}(Y)} \abs{T b} \le C \int_{C_r^-(Y)\cap C_1^-} \abs{b}
\]
whenever $b \in L^2(C_1^-)$ is supported in $C_r^-(Y)\cap C_1^-$, $\int_{C_1^-} b =0$.
Then, for $f \in L^2(C_1^-)$ and any $\alpha>0$, we have
\[
\Abs{\set{X \in C_1^- : \abs{T f(X)} > \alpha}}  \le \frac{C'}{\alpha} \int_{C_1^-} \abs{f},
\]
where $C'=C'(n,c,C)$ is a constant.
\end{lemma}
\begin{proof}
We refer to Stein \cite[p. 22]{Stein93}, where the proof is based on the Calder\'on-Zygmund decomposition and the domain is assumed to be the whole space.
In our case, we can modify the proof there by using the ``dyadic cubes'' decomposition of $C_1^-$; see Christ \cite[Theorem 11]{Ch90}.
\end{proof}
\begin{lemma}			\label{lem02-weak11}
Let $\bar a^{ij}=\bar a^{ij}(t)$ satisfy \eqref{parabolicity}.
Consider the operator $P_0$ defined by
\begin{equation}			\label{eq1225tu}
P_0 u := \partial _t u- D_i(\bar a^{ij}(t) D_j u).
\end{equation}
For $\vec f=(f^1,\ldots, f^n) \in L^2(C_1^-)$, let $u \in \cH^1_2(C_1^-)$ be the weak solution to
\[
P_0 u= \dv\vec f\;\mbox{in}\; C_1^-;\quad
u=0 \;\mbox{on} \; \partial_p^-C_1^-.
\]
Then for any $\alpha>0$, we have
\[
\Abs{\set{X \in C_1^- : \abs{D u(X)} > \alpha}}  \lesssim_{n, \lambda}\, \frac{1}{\alpha} \int_{C_1^-} \abs{\vec f}.
\]
\end{lemma}
\begin{proof}
The proof is a modification of \cite[Lemma 2.2]{DK17}.
Since the map $T: \vec f \mapsto D u$ is a bounded linear operator on $L^2(C_1^-)$, it suffices to show that $T$ satisfies the hypothesis of Lemma~\ref{lem01-stein}.
We set $c=2$.
Fix $Y \in C_1^-$ and $0<r<\frac1 2$, let $\vec b \in L^2(C_1^-)$ be supported in $C_r^-(Y)\cap C_1^-$ with mean zero.
Let $u \in \cH^{1}_2(C_1^-)$ be the unique weak solution of
\[
P_0 u= \dv \vec b\;\mbox{ in }\; C_1^-;\quad
u=0 \;\mbox{ on } \; \partial_p^-C_1^-.
\]
For any $R\ge 2r$ such that $C_1^- \setminus C_R(Y) \neq \emptyset$ and $\vec g \in C^\infty_c((C_{2R}(Y)\setminus C_R(Y))\cap C_1^-)$,
let  $v \in \cH^{1}_2(C_1^-)$ be a weak solution of
\[
P_0^* v=\dv \vec g\;\mbox{ in }\; C_1^-;\quad
v=0 \;\mbox{ on } \; \partial_p^+C_1^-,
\]
where $P_0^*$ is the adjoint operator of $P_0$.
By the definition of weak solutions, we have the identity
\[
\int_{C_1^-} Du \cdot \vec g = \int_{C_1^-} \vec b \cdot Dv = \int_{C_r^-(Y) \cap C_1^-} \vec b \cdot \left(Dv-Dv(Y) \right).	
\]
Therefore, for any $\delta \in (0,1)$,
\begin{align*}
\Abs{\int_{(C_{2R}(Y)\setminus C_R(Y))\cap C_1^-} Du \cdot \vec g\,} &\le \norm{\vec b}_{L^1(C_r^-(Y) \cap C_1^-)} \norm{Dv-Dv(Y)}_{L^\infty(C_r^-(Y)\cap C_1^-)}\\
&\lesssim r^\delta \norm{\vec b}_{L^1(C_r^-(Y) \cap C_1^-)} [Dv]_{C^{\delta/2,\delta}(C_r^-(Y)\cap C_1^-)}.
\end{align*}

Since $P_0^* v =0$ in $C_R(Y)\cap C_1^-$ and $v=0$ on $\partial_p^+ C_1^-$, we estimate
\begin{equation}				\label{eq1917sat}
[Dv]_{C^{\delta/2,\delta}(C_r^-(Y)\cap C_1^-)} \le [Dv]_{C^{\delta/2,\delta}(C_{R/2}^-(Y)\cap C_1^-)} \lesssim R^{-\delta-2-\frac{n}2} \norm{v}_{L^2(C_{7R/8}(Y)\cap C_1^-)}.
\end{equation}
Indeed, to see the second inequality, take a cut-off function $\eta \in C^\infty_c(C_{3R/4}(Y))$ satisfying $\eta=1$ on $C_{R/2}(Y)$, and $\abs{\partial_t \eta}+ \abs{D \eta}^2 +\abs{D^2 \eta}\lesssim R^{-2}$, choose $p>n+2$ such that $\delta=1-\frac{n+2}{p}$ and apply the Sobolev embedding (Lemma \ref{psobolev}) to get
\[
\norm{Dv}_{C^{\delta/2,\delta}(C_{R/2}^-(Y)\cap C_1^-)} \le \norm{D(\eta v)}_{C^{\delta/2,\delta}(C_1^-)} \lesssim \norm{\eta v}_{W^{1,2}_p(C_1^-)}.
\]
By the $W^{1,2}_p$ estimate for parabolic systems with coefficients depending only on $t$, the properties of $\eta$, and an iteration argument, we have (see the proof of \cite[Lemma~1]{DK11})
\[
\norm{\partial_t(\eta v)}_{L^p(C_1^-)} + \norm{D^2(\eta v)}_{L^p(C_1^-)} \lesssim R^{-2} \norm{v}_{L^p(C_{3R/4}(Y)\cap C_1^-)}.
\]
By the interpolation inequalities and using $R \le 2$, we have
\[
\norm{\eta v}_{W^{1,2}_p(C_1^-)} \lesssim R^{-2} \norm{v}_{L^p(C_{3R/4}(Y)\cap C_1^-)}.
\]
The above inequality together with the Sobolev embedding, H\"older's inequality, and an iteration argument yields
\[
\norm{v}_{L^p(C_{3R/4}(Y)\cap C_1^-)} \lesssim  R^{\frac{n+2}{p}-\frac{n+2}{2}}\norm{v}_{L^2(C_{7R/8}(Y)\cap C_1^-)}=R^{-\delta -\frac{n}{2}}\norm{v}_{L^2(C_{7R/8}(Y)\cap C_1^-)}.
\]
Combining all these together, we obtain \eqref{eq1917sat}.
Now we consider two cases.
When $B_{7R/8}(y)\cap \partial B_1\neq \emptyset$, we apply the boundary Poincar\'e inequality with respect to $x$ to get
\[
\norm{v}_{L^2(C_{7R/8}(Y)\cap C_1^-)} \le CR \norm{Dv}_{L^2(C_{R}(Y)\cap C_1^-)}.
\]
When $B_{7R/8}(y)\subset B_1$, the estimate \eqref{eq1917sat}  is available for $v-c$ with $c=\fint_{C_{7R/8}(Y) \cap C_1^- } v$, and thus, we apply \cite[Lemma 4.2.1]{Kr08} to get
\[
\norm{v-c}_{L^2(C_{7R/8}(Y)\cap C_1^-)} \le CR \norm{Dv}_{L^2(C_{R}(Y)\cap C_1^-)}.
\]
In any case, we have
\begin{align*}
[Dv]_{C^{\delta/2,\delta}(C_r^-(Y)\cap C_1^-)}
&\lesssim R^{-\delta-1-\frac{n}2} \norm{Dv}_{L^2(C_{R}(Y)\cap C_1^-)} \lesssim R^{-\delta-1-\frac{n}2} \norm{Dv}_{L^2(C_1^-)} \\
&\lesssim R^{-\delta-1-\frac{n}{2}} \norm{\vec g}_{L^2(C_1^-)}= R^{-\delta-1-\frac{n}{2}} \norm{\vec g}_{L^2((C_{2R}(Y)\setminus C_{R}(Y))\cap C_1^-)}.
\end{align*}
Therefore, by the duality, we get
\[
\norm{Du}_{L^2((C_{2R}(Y)\setminus C_R(Y))\cap C_1^-)}
\lesssim r^\delta R^{-\delta-1-\frac{n}2} \norm{\vec b}_{L^1(C_r^-(Y)\cap C_1^-)},
\]
and thus by H\"older's inequality, we have
\begin{equation}
                            \label{eq4.37}
\norm{Du}_{L^1((C_{2R}(Y)\setminus C_R(Y))\cap C_1^-)}
\lesssim r^\delta R^{-\delta}\norm{\vec b}_{L^1(C_r^-(Y) \cap C_1^-)}.
\end{equation}
Now let $N$ be the smallest positive integer such that $C_1^-\subset C_{2^{N+1}r}(Y)$.
By taking $R=2r,4r,\ldots,2^{N}r$ in \eqref{eq4.37}, we have
\begin{equation*}			%\label{14.06m}
\int_{C_1^- \setminus C_{2r}(Y)} \abs{D u} \lesssim \sum_{k=1}^{N} 2^{-k\delta} \norm{\vec b}_{L^1(C_r^-(Y)\cap C_1^-)} \sim \int_{C_r^-(Y)\cap C_1^-} \abs{\vec b}.
\end{equation*}
Therefore, $T$ satisfies the hypothesis of Lemma~\ref{lem01-stein} and  the lemma is proved.
\end{proof}
The following lemma follows from the interior $W^{1,2}_p$ estimate for parabolic systems with coefficients depending only on $t$ (see, e.g., \cite{DK11}), the Sobolev embedding, and an iteration argument.

\begin{lemma}\label{lem:c_half}
Let $P_0$ be as in \eqref{eq1225tu}.
If $u$ is a weak solution of
\[
P_0 u=0\; \text{ in } \; C_r^-(X_0),
\]
then, $u \in C^{1/2,1}(C_{r/2}^-(X_0))$ and for any $p \in (0,\infty)$, we have the estimate
\[
[u]_{C^{1/2,1}(C_{r/2}^-(X_0))} \lesssim_{n,\lambda, p}\, r^{-1} \left( \fint_{C_r^-(X_0)} \abs{u}^p\right)^{\frac1p}.
\]
\end{lemma}

\subsection{Proof of Theorem~\ref{thm:d-wolt}}\label{sec:int-d}
We shall first derive an a priori estimate of the modulus of continuity of $Du$ by assuming that $u$ is in $C^{1/2,1}(C_3^-)$.
The general case follows from a standard approximation argument.

For $X_0 \in C_3^-$ and $0<r<\frac13$, we consider the quantity
\[
\phi(X_0,r):=  \inf_{\vec q \in \bR^n} \left( \fint_{C_r^-(X_0)}
\abs{D u - \vec q}^{\frac12} \right)^{2}.
\]
First of all, we note that
\begin{equation}				\label{eq12.14fa}
\phi(X_0, r)  \le  \left(\fint_{C_r^-(X_0)}
\abs{D u}^{\frac12} \right)^2 \lesssim r^{-n-2} \norm{Du}_{L^1(C_r^-(X_0))}.
\end{equation}
We want to control the quantity $\phi(X_0, r)$.
To this end, we decompose $u=v+w$, where $w$ is the weak solution of the problem
\[
P_0 w = D_i ((a^{ij}- \bar a^{ij}) D_ju) +D_i(g^i-\bar g^i),\;\mbox{ in }\;C_r^-(X_0);\quad
w=0 \;\mbox{ on }\;\partial_p^- C_r^-(X_0).
\]
Here and below, we use the simplified notation
\begin{equation}					\label{eq1754bth}
\bar a^{ij}=\bar a^{ij}(t)=\fint_{B_r(x_0)} a^{ij}(t,x)\,dx, \quad \bar g^i=\bar g^i(t) = \fint_{B_r(x_0)} g^i(t,x) \,dx,
\end{equation}
and $P_0$ is the parabolic operator as in \eqref{eq1225tu} with the coefficients $\bar a^{ij}(t)$.
We shall write
\[
\bar{\mathbf A}=\bar{\mathbf A}^{\textsf x}_{x_0, r}(t)=(\bar a^{ij}(t))_{i,j=1}^n\ \text{and}\  \bar{\vec g}=\bar{\vec g}^{\textsf x}_{x_0.r}(t)=(\bar g^1(t),\ldots, \bar g^n(t)).
\]
By Lemma~\ref{lem02-weak11} and scaling, we have
\begin{multline}
                                    \label{eq10.20}
\abs{\set{Y\in C_r^-(X_0): \abs{Dw(Y)} > \alpha}} \\
\lesssim \frac{1}{\alpha}\left(\norm{Du}_{L^\infty(C_r^-(X_0))} \int_{C_r^-(X_0)} \abs{\mathbf A-\bar{\mathbf A}} +  \int_{C_r^-(X_0)}  \abs{\vec g -\bar {\vec g}} \,\right).
\end{multline}
Recall the formula
\begin{align*}
\int_{C_r^-(X_0)} \abs{Dw}^{\frac12}
&=\int_0^\infty \frac12 \alpha^{-\frac12}\abs{\set{Y\in C_r^-(X_0): \abs{Dw(Y)} > \alpha}}\,d\alpha\\
&=\int_0^\tau + \int_\tau^\infty \frac12 \alpha^{-\frac12}\abs{\set{Y\in C_r^-(X_0): \abs{Dw(Y)} > \alpha}}\,d\alpha,
\end{align*}
where $\tau>0$ is to be determined.
When $0< \alpha \le \tau$, we bound
$$\abs{\set{Y\in C_r^-(X_0): \abs{Dw(Y)} > \alpha}}$$
 simply by $\abs{C_r^-(X_0)}$.
When $\alpha > \tau$, we use \eqref{eq10.20}.
It then follows that
\[
\int_{C_r^-(X_0)} \abs{Dw}^{\frac12}\lesssim \tau^{\frac12}\abs{C_r^-(X_0)}+\tau^{-\frac12}\left(\norm{Du}_{L^\infty(C_r^-(X_0))} \int_{C_r^-(X_0)} \abs{\mathbf{A}-\bar{\mathbf A}} +  \int_{C_r^-(X_0)}  \abs{\vec g -\bar {\vec g}} \,\right).
\]
By optimizing over $\tau$,
we get
\[
\int_{C_r^-(X_0)} \abs{Dw}^{\frac12} \lesssim \abs{C_r^-(X_0)}^{\frac12} \left(\norm{Du}_{L^\infty(C_r^-(X_0))} \int_{C_r^-(X_0)} \abs{\mathbf{A}-\bar{\mathbf A}} +  \int_{C_r^-(X_0)}  \abs{\vec g -\bar {\vec g}} \,\right)^{\frac12}.
\]
Therefore, according with \eqref{13.24f} we have
\begin{equation}				\label{eq15.50ab}
\left(\fint_{C_r^-(X_0)} \abs{Dw}^{\frac12} \right)^{2} \lesssim \omega_{\mathbf A}^{\textsf x}(r) \,\norm{Du}_{L^\infty(C_r^-(X_0))} +  \omega_{\vec g}^{\textsf x}(r).
\end{equation}
On the other hand, $v=u-w$ satisfies
\[
P_0 v =0 \;\mbox{ in }\;C_r^-(X_0),
\]
and for any constant $\vec q=(q^1,\ldots, q^n) \in \bR^n$, $D_i v - q^i$ satisfies the same equation for $i=1,\ldots, n$. By Lemma~\ref{lem:c_half}, we get
\begin{equation}			\label{int-reg}
[Dv]_{C^{1/2,1}(C_{r/2}^-(X_0))} \le N_0 r^{-1} \left(\fint_{C_r^-(X_0)} \abs{Dv - \vec q}^{\frac12}\,\right)^{2},
\end{equation}
where $N_0>0$ is an absolute constant.
Let $0< \kappa \le \frac12$ to be a number to be fixed later.
Then, from \eqref{int-reg}
\begin{equation}\label{eq15.50br}
\left(\fint_{C_{\kappa r}^-(X_0)} \abs{Dv - Dv_{X_0, \kappa r}}^{\frac12} \right)^{2} \le 2\kappa r [Dv]_{C^{1/2,1}(C_{r/2}^-(X_0))} \le 2N_0 \kappa \left(\fint_{C_r^-(X_0)} \abs{Dv - \vec q}^{\frac12} \, \right)^{2}.
\end{equation}
Here, we recall the facts that for all $a, b \ge 0$, we have
\[
(a+b)^{\frac12} \le a^{\frac12} + b^{\frac12}, \quad (a^{\frac12}+ b^{\frac12})^{2} \le 2(a+b),
\]
and
\[
\norm{f+g}_{L^{1/2}} \le 2 \left( \norm{f}_{L^{1/2}} + \norm{g}_{L^{1/2}}\right).
\]
Using the decomposition $u=v+w$, we obtain from \eqref{eq15.50br} that
\begin{align*}
&\left(\fint_{C_{\kappa r}^-(X_0)} \abs{Du - Dv_{X_0, \kappa r}}^{\frac12} \right)^{2} \\
&\qquad\qquad \le 2 \left(\fint_{C_{\kappa r}^-(X_0)} \abs{Dv - Dv_{X_0, \kappa r}}^{\frac12} \right)^{2}+ 2\left(\fint_{C_{\kappa r}^-(X_0)} \abs{Dw}^{\frac12} \right)^{2}\\
&\qquad \qquad \le  8N_0\kappa \left(\fint_{C_r^-(X_0)} \abs{Du - \vec q}^{\frac12} \right)^{2} +(8N_0 \kappa +2\kappa^{-2n-4}) \left(\fint_{C_r^-(X_0)} \abs{Dw}^{\frac12} \right)^{2}.
\end{align*}
Since $\vec q\in \bR^n$ is arbitrary, by using \eqref{eq15.50ab}, we thus obtain
\[
\phi(X_0,\kappa r) \le 8N_0 \kappa \, \phi(X_0, r)+ C(8N_0 \kappa +2\kappa^{-2n-4}) \left( \omega_{\mathbf A}^{\textsf x}(r) \,\norm{Du}_{L^\infty(C_r^-(X_0))} +  \omega_{\vec g}^{\textsf x}(r) \right).
\]

For any given $\beta \in (0,1)$, let $\kappa \in (0,\frac12)$ be sufficiently small so that $8N_0 \le \kappa^{\beta-1}$.
Then, we obtain
\[
\phi(X_0, \kappa r) \le \kappa^\beta \phi(X_0,r)+ C \left( \omega_{\mathbf A}^{\textsf x}(r) \,\norm{Du}_{L^\infty(C_r^-(X_0))} +  \omega_{\vec g}^{\textsf x}(r) \right).
\]
Note that $\kappa^\beta <1$.
By iterating, for $j=1,2,\ldots$, we get
\[
\phi(X_0, \kappa^j r) \le \kappa^{j \beta} \phi(X_0,r) +C \norm{Du}_{L^\infty(C_r^-(X_0))} \sum_{i=1}^{j} \kappa^{(i-1)\beta} \omega_{\mathbf A}^{\textsf x}(\kappa^{j-i} r) + C \sum_{i=1}^{j} \kappa^{(i-1)\beta} \omega_{\vec g}^{\textsf x}(\kappa^{j-i} r).
\]
Therefore, we have
\begin{equation}				\label{eq22.25fr}
\phi(X_0, \kappa^j r) \le \kappa^{j\beta} \phi(X_0,r) +C \norm{Du}_{L^\infty(C_r^-(X_0))}\,\tilde \omega_{\mathbf A}^{\textsf x}(\kappa^{j} r) + C \tilde \omega_{\vec g}^{\textsf x}(\kappa^{j} r),
\end{equation}
where we set
\begin{equation}				\label{eq14.27w}
\tilde \omega_{\bullet}^{\textsf x}(t)=\sum_{i=1}^{\infty} \kappa^{i\beta} \left(\omega_{\bullet}^{\textsf x}(\kappa^{-i}t)\, [\kappa^{-i}t \le 1] + \omega_{\bullet}^{\textsf x}(1)\, [\kappa^{-i}t >1] \right).
\end{equation}
Here, we used the Iverson bracket notation, i.e., $[P]=1$ if $P$ is true and $[P]=0$ otherwise.
We remark that $\tilde \omega_{\bullet}^{\textsf x}(t)$ is a Dini function; see \cite[Lemma~1]{Dong2012}.

Note that for any $\rho$ satisfying $0<\rho\le r$, if we set $j$ to be the integer satisfying $\kappa^{j+1} < \rho/r \le \kappa^j$.
Let us momentarily assume that $j\ge 1$.
Then by \eqref{eq22.25fr} we get
\begin{align}
				\nonumber
\phi(X_0, \rho) & \le \kappa^{-\beta} \left(\frac{\rho}r\right)^{\beta}\, \phi(X_0, \kappa^{-j} \rho) +C \norm{Du}_{L^\infty(C_r^-(X_0))}\,\tilde \omega_{\mathbf A}^{\textsf x}(\rho) + C \tilde \omega_{\vec g}^{\textsf x}(\rho)\\
				\label{eq0946fo}
&\le \kappa^{-\beta-2d-4} \left(\frac{\rho}r\right)^{\beta}\, \phi(X_0, r) +C \norm{Du}_{L^\infty(C_r^-(X_0))}\,\tilde \omega_{\mathbf A}^{\textsf x}(\rho) + C \tilde \omega_{\vec g}^{\textsf x}(\rho)
\end{align}
and thus we have
\begin{equation}				\label{eq0947fo}
\phi(X_0, \rho)  \lesssim  \left(\frac{\rho}r\right)^{\beta}\,\phi(X_0, r)+ \norm{Du}_{L^\infty(C_r^-(X_0))}\,\tilde \omega_{\mathbf A}^{\textsf x}(\rho) +  \tilde \omega_{\vec g}^{\textsf x}(\rho).
\end{equation}
Notice that the above estimate is still true when $j=0$.
Therefore, \eqref{eq0947fo} is valid for any $\rho$ and $r$ satisfying $0<\rho\le r <\frac{1}{3}$.

Now, let $\vec q_{X_0,r}$ be chosen so that
\[
\left(\fint_{C_r^-(X_0)} \abs{Du - \vec q_{X_0,r}}^{\frac12} \right)^2 =\inf_{\vec q \in \bR^n} \left( \fint_{C_r^-(X_0)} \abs{D u - \vec q}^{\frac12} \right)^2=\phi(X_0, r).
\]
Since we have
\[
\abs{\vec q_{X_0,r} - \vec q_{X_0, \frac{1}{2}r}}^{\frac12} \le
\abs{Du(Y)-\vec q_{X_0,r}}^{\frac12} + \abs{Du(Y) - \vec q_{X_0,\frac{1}{2}r}}^{\frac12},
\]
taking the average over $Y \in C_{r/2}^-(X_0)$ and then taking the square, we obtain
\[
\abs{\vec q_{X_0, r} -\vec q_{X_0,\frac{1}{2}r}}\le 2\phi(X_0, r) +2 \phi(X_0,\tfrac{1}{2} r).
\]
Then, by iterating and using the triangle inequality
\[
\abs{\vec q_{X_0,2^{-k} r} - \vec q_{X_0, r}}  \le 4\sum_{j=0}^k \phi(X_0, 2^{-j} r).
\]
Since the right-hand side of \eqref{eq0947fo} goes to zero as $\rho\to 0$ by the assumption that $u \in C^{1/2,1}(C_3^-)$, we find that
\[
\lim_{k\to \infty} \vec q_{X_0,2^{-k} r}=Du(X_0).
\]
Therefore, by taking $k\to \infty$, using \eqref{eq0947fo} and \cite[Lemma~2.7]{DK17}, we get
\begin{align}
					\nonumber
\abs{Du(X_0)-\vec q_{X_0,r}} &\lesssim \sum_{j=0}^\infty \phi(X_0, 2^{-j} r) \\
					\label{eq13.13}
& \lesssim \phi(X_0, r)+\norm{Du}_{L^\infty(C_r^-(X_0))} \int_0^r \frac{\tilde\omega_{\mathbf A}^{\textsf x}(t)}t \,dt+\int_0^r \frac{\tilde\omega_{\vec g}^{\textsf x}(t)}t \,dt.
\end{align}
By averaging the inequality
\[
\abs{\vec q_{X_0, r}}^{\frac12} \le \abs{Du(Y) -\vec q_{X_0,r}}^{\frac12} + \abs{Du(Y)}^{\frac12}
\]
over $Y \in C_r^-(X_0)$, taking the square, and using \eqref{eq12.14fa} we get
\[
\abs{\vec q_{X_0, r}} \le 2 \phi(X_0,r) + 2 \left(\fint_{C_r^-(X_0)} \abs{Du}^{\frac12}\right)^{2} \lesssim r^{-n-2} \norm{Du}_{L^1(C_r^-(X_0))},
\]
where we used \eqref{eq12.14fa}.
Therefore, by combining the above with \eqref{eq13.13}, we get
\[
\abs{Du(X_0)}  \lesssim r^{-n-2} \norm{Du}_{L^1(C_r^-(X_0))} + \norm{Du}_{L^\infty(C_r^-(X_0))} \int_0^r \frac{\tilde\omega_{\mathbf A}^{\textsf x}(t)}t \,dt+\int_0^r \frac{\tilde\omega_{\vec g}^{\textsf x}(t)}t \,dt.
\]
Now, taking supremum for $X_0\in C_r^-(Y_0)$, where $Y_0 \in C_2^-$ and $r<\frac 13$, we have
\begin{multline*}
\norm{Du}_{L^\infty(C_r^-(Y_0))}  \\
\le C \left( r^{-n-2} \norm{Du}_{L^1(C_{2r}^-(Y_0))} + \norm{Du}_{L^\infty(C_{2r}^-(Y_0))} \int_0^r \frac{\tilde\omega_{\mathbf A}^{\textsf x}(t)}t \,dt +\int_0^r \frac{\tilde\omega_{\vec g}^{\textsf x}(t)}t \,dt \right).
\end{multline*}
We fix $r_0<\frac13$ such that
\[
C \int_0^{r_0} \frac{\tilde\omega_{\mathbf A}^{\textsf x}(t)}t \,dt \le \frac1{3^{n+2}}.
\]
Then, we have for any $Y_0 \in C_2^-$ and $0<r\le r_0$ that
\[
\norm{Du}_{L^\infty(C_r^-(Y_0))} \le
3^{-n-2}\norm{Du}_{L^\infty(C_{2r}^-(Y_0))} + C r^{-n-2} \norm{Du}_{L^1(C_{2r}^-(Y_0))} + C \int_0^r \frac{\tilde\omega_{\vec g}^{\textsf x}(t)}t \,dt.
\]

For $k=1,2,\ldots$, denote $r_k=3-2^{1-k}$.
Note that $r_{k+1}-r_k=2^{-k}$ for $k\ge 1$ and $r_1=2$.
For $Y_0\in C_{r_k}^-$ and $r\le 2^{-k-2}$, we have $C_{2r}^-(Y_0) \subset C_{r_{k+1}}^-$. We take $k_0\ge 1$ sufficiently large such that $2^{-k_0-2}\le r_0$.
It then follows that for any $k\ge k_0$,
\[
\norm{Du}_{L^\infty(C_{r_k}^-)} \le 3^{-n-2} \norm{Du}_{L^\infty(C_{r_{k+1}}^-)}+C 2^{k(n+2)} \norm{Du}_{L^1(C_4^-)} + C \int_0^{1} \frac{\tilde\omega_{\vec g}^{\textsf x}(t)}t \,dt.
\]
By multiplying the above by $3^{-k(n+2)}$ and then summing over $k=k_0, k_0+1,\ldots$, we get
\begin{equation*}
\sum_{k=k_0}^\infty 3^{-k(n+2)}\norm{Du}_{L^\infty(C^-_{r_k})} \le \sum_{k=k_0}^\infty 3^{-(k+1)(n+2)} \norm{Du}_{L^\infty(C^-_{r_{k+1}})}
+C \norm{Du}_{L^1(C^-_4)} + C \int_0^{1} \frac{\tilde\omega_{\vec g}^{\textsf x}(t)}t \,dt.
\end{equation*}
Since we assume that $u\in C^{1/2,1}(C^-_3)$, the summations on both sides are convergent, and we get
\begin{equation}					\label{eq10.23m}
\norm{Du}_{L^\infty(C^-_2)} \lesssim  \norm{Du}_{L^1(C^-_4)} + \int_0^{1} \frac{\tilde\omega_{\vec g}^{\textsf x}(t)}t \,dt.
\end{equation}

Next, we estimate a modulus of continuity of $Du$.
For $X$, $Y \in C^-_1$ with $r:=2\abs{X-Y} \in (0,1)$, we have
\begin{multline*}
\abs{Du(X)-Du(Y)}^{\frac12} \le \abs{Du(X)- \vec q_{X,r}}^{\frac12} + \abs{\vec q_{X,r}-\vec q_{Y,r}}^{\frac12}+\abs{Du(Y)-\vec q_{Y,r}}^{\frac12}\\
\le  \abs{Du(X) - \vec q_{X,r}}^{\frac12}+\abs{Du(Y) - \vec q_{Y,r}}^{\frac12}+ \abs{Du(Z)-\vec q_{X,r}}^{\frac12}+\abs{Du(Z)-\vec q_{Y,r}}^{\frac12}.
\end{multline*}
By averaging over $Z \in C_r^-(X)\cap C_r^-(Y)$,  taking the square, and using \eqref{eq13.13}, we get
\begin{align}
							\nonumber
\abs{Du(X)-Du(Y)} &\lesssim \sup_{X_0 \in C^-_1}\, \abs{Du(X_0) - \vec q_{X_0,r}} + \phi(X,r) + \phi(Y,r) \lesssim \sup_{X_0 \in C^-_1}\, \sum_{j=0}^\infty \phi(X_0, 2^{-j} r) \\
							\nonumber
& \lesssim \sup_{X_0 \in C^-_1} \phi(X_0, r) +\norm{Du}_{L^\infty(C_2^-)} \int_0^r \frac{\tilde\omega_{\mathbf A}^{\textsf x}(t)}t \,dt+\int_0^r \frac{\tilde\omega_{\vec g}^{\textsf x}(t)}t \,dt\\
							\label{eq13.27m}
& \lesssim \sup_{X_0 \in C^-_1} r^\beta \phi(X_0, 1) +\norm{Du}_{L^\infty(C_2^-)} \int_0^r \frac{\tilde\omega_{\mathbf A}^{\textsf x}(t)}t \,dt+\int_0^r \frac{\tilde\omega_{\vec g}^{\textsf x}(t)}t \,dt,
\end{align}
where we used \eqref{eq0947fo} and the fact $\tilde \omega_\bullet^{\textsf x}(r) \lesssim \int_0^r \tilde \omega_\bullet^{\textsf x}(t)/t\,dt$ in the last inequality.
Therefore, we get from \eqref{eq12.14fa}, \eqref{eq13.27m}, and \eqref{eq10.23m} that
\begin{multline*}				%\label{eq14.00}
\abs{Du(X)-Du(Y)} \lesssim \norm{Du}_{L^1(C^-_4)} \,\abs{X-Y}^\beta\\
+\left(\norm{Du}_{L^1(C^-_4)} + \int_0^{1} \frac{\tilde\omega_{\vec g}^{\textsf x}(t)}t \,dt\right) \int_0^{2\abs{X-Y}} \frac{\tilde \omega_{\mathbf A}^{\textsf x}(t)}t \,dt + \int_0^{2\abs{X-Y}} \frac{\tilde \omega_{\vec g}^{\textsf x}(t)}t \,dt
\end{multline*}
for any $X$ ,$Y \in C_1^-$ with $\abs{X-Y} < \frac12$.

Now, we show that $u$ satisfies \eqref{eq1121may9} in $C_1^-$.
First, let us fix a non-negative smooth function $\eta$ that is compactly supported in $B_1 \subset \bR^n$ and satisfying $\int_{B_1} \eta= 1$.
We shall further assume that $\eta$ is an even function, i.e.,
\begin{equation}				\label{eq1745bt}
\eta(-y)=\eta(y).
\end{equation}
For a function $f$ on $C^-_4$, we define the partial modification $f^{(\epsilon)}$ of $f$ by
\begin{equation}				\label{eq1526bt}
f^{(\epsilon)}(t,x)=\int_{\bR^n} f(t, x-\epsilon y) \eta(y)\,dy=\frac{1}{\epsilon^n}\int_{\bR^n} f(t,y)\eta\left(\frac{x-y}{\epsilon}\right)dy
\end{equation}
provided the integral makes sense.

Let  us fix $X=(t,x)$ in $C^-_1$ and let $r>0$ be any number such that $(t-r^2, x) \in C^-_1$.
We shall estimate $\abs{u(t-r^2,x)-u(t,x)}$ as follows:
\begin{multline*}
\abs{u(t-r^2,x)-u(t,x)} \\
\le  \abs{u(t-r^2,x)-u^{(r)}(t-r^2,x)} +\abs{u^{(r)}(t-r^2,x)-u^{(r)}(t,x)} + \abs{u(t,x)-u^{(r)}(t,x)}.
\end{multline*}
We denote
\[
\osc_{C_r^-(X)}f= \sup_{Y, Y' \in C_r^-(X)} \abs{f(Y)-f(Y')}.
\]
By using the condition \eqref{eq1745bt}, we have
\[
u^{(r)}(s,x)-u(s,x)  = \frac12\int_{B_1(0)}  \bigl\{ u(s,x+ry)+u(s,x-ry) -2u(s,x)\bigr\} \,\eta(y)\,dy.
\]
Since
\begin{align*}
u(s,x+ry)+u(s,x-ry) - 2u(s,x) &= \int_0^r \frac{d}{d\tau} u(s, x+\tau y)\,d\tau+ \int_0^r \frac{d}{d\tau} u(s, x-\tau y)\,d\tau \\
&= \int_0^r \left\{Du(s,x+\tau y)-Du(s,x-\tau y) \right\}\cdot y\,d\tau
\end{align*}
for $\abs{y} \le 1$ and $s \in [t-r^2, t]$, we have
\[
\Abs{u(s, x-r y)+u(s, x+r y) - 2u(s,x)} \le  r \osc_{C_r^-(X)} Du.
\]
Therefore, for $s \in [t-r^2, t]$, we have
\[
\abs{u^{(r)}(s,x)-u(s,x)} \le \frac{r}{2}  \osc_{C_r^-(X)} Du
\]
and thus, it follows that
\[
\abs{u(t-r^2,x)-u^{(r)}(t-r^2,x)}+\abs{u(t,x)-u^{(r)}(t,x)} \le r \osc_{C_r^-(X)} Du.
\]
On the other hand, we have
\begin{align}			\nonumber
u^{(r)}(t,x)-u^{(r)}(t-r^2,x) &= \int_{t-r^2}^{t} \partial_s u^{(r)}(s,x)\,ds\\
					\label{eq1607bt}
&= \int_{t-r^2}^{t} \left\{ D_i(a^{ij} D_j u)+ D_i g^i \right\}^{(r)}(s,x)\,ds.
\end{align}
Observe that by \eqref{eq1526bt}, we find that for $i=1,\ldots, n$,
\begin{align*}
D_i f^{(r)}(s,x) &= \frac{1}{r^{n+1}} \int_{\bR^n} f(s,y)\, D_i\eta\left(\frac{x-y}{r}\right) \,dy\\
&=\frac{1}{r^{n+1}} \int_{B_r(x)} \bigl( f(s,y) - q \bigr)\,D_i\eta\left(\frac{x-y}{r}\right) \,dy,\quad \forall q \in \bR,
\end{align*}
and thus, by taking $q=\bar{f}_{x,r}^{\textsf x}(s)$, we get
\[
\Abs{\int_{t-r^2}^t D f^{(r)}(s,x) \,ds} \le C r  \norm{D\eta}_{L^\infty} \,  \omega_f^{\textsf x}(r),
\]
where $C=C(n)$.
Therefore, from \eqref{eq1607bt} and the above, we have
\[
\abs{u^{(r)}(t,x)-u^{(r)}(t-r^2,x)} \le C r \left( \omega_{\mathbf{A}D u}^{\textsf x}(r)+  \omega_{\vec g}^{\textsf x}(r)\right).
\]
Combining together, we have
\begin{equation}				\label{eq13.19fs}
\frac{\Abs{u(t-r^2,x)-u(t,x)}}{r} \le C  \left( \osc_{C_r^-(X)} Du+ \omega^{\textsf x}_{\mathbf{A}D u}(r)+  \omega_{\vec g}^{\textsf x}(r)\right),
\end{equation}
where $C=C(n)$.
Since the right-hand side of the above inequality goes to zero as $r \to 0$, we see that $u$ satisfies \eqref{eq1121may9} as desired.
The theorem is proved.

\subsection{Proof of Theorem~\ref{thm:nd-wolt}}
						\label{sec:int-nd}
The proof of theorem is similar to that of Theorem~\ref{thm:d-wolt}.
First, we present a lemma that plays the role of Lemma~\ref{lem02-weak11}.
\begin{lemma}			\label{lem-weak11-nd}
Let $\bar a^{ij}=\bar a^{ij}(t)$ satisfy \eqref{parabolicity}.
Assume that $a^{ij}$ are symmetric and consider the operator $\cP_0$ defined by
\[
\cP_0 u := \partial_t u- a^{ij}(t) D_{ij} u.
\]
For $f \in L^2(C_1^-)$ let $u \in W^{1,2}_2(C_1^-)$ be the unique solution to
\[
\cP_0 u= f\;\mbox{ in }\; C_1^-;\quad
u=0 \;\mbox{ on } \; \partial_p^-C_1^-.
\]
Then for any $\alpha>0$, we have
\[
\Abs{\set{X \in C_1^- : \abs{D^2 u(X)} > \alpha}}  \lesssim_{n,\lambda}\, \frac{1}{\alpha} \int_{C_1^-} \abs{f}.
\]
\end{lemma}
\begin{proof}
The proof is similar to that of Lemma~\ref{lem02-weak11}.
Since by the $W^{1,2}_2$ estimate (see e.g., \cite{DK11}), the map $T: f \mapsto D^2 u$ is a bounded linear operator on $L^2(C_1^-)$, it suffices to show that $T$ satisfies the hypothesis of Lemma~\ref{lem01-stein}.

We again set $c=2$.
Fix $Y \in C_1^-$ and $0<r<\frac1 4$, let $b \in L^2(C_1^-)$ be supported in $C_r^-(Y)\cap C_1^-$ with mean zero.
Let $u \in W^{1,2}_2(C_1^-)$ be the unique strong solution of
\[
\cP_0 u= b\;\mbox{ in }\; C_1^-;\quad
u=0 \;\mbox{ on } \; \partial_p^- C_1^-.
\]
For any $R\ge cr$ such that $C_1^-\setminus C_R(Y) \neq \emptyset$ and $\mathbf g=(g^{ij})_{i,j=1}^n \in C^\infty_c((C_{2R}(Y)\setminus C_R(Y))\cap C_1^-)$,
let  $v \in L^2(C_1^-)$ be the adjoint solution of
\[
\cP_0^* v= \nabla^2 \mathbf g\;\mbox{ in }\; C_1^-;\quad
v=-\frac{ \mathbf g \nu \cdot \nu}{\mathbf{A} \nu \cdot  \nu} \;\mbox{ on } \; (-1,0)\times \partial B_1(0);\quad
v(0,\cdot)=0 \;\mbox{ on } \; \overline B_1.
\]
where $\cP_0^*$ is the adjoint operator of $\cP_0$.
By the duality, we have the identity
\[
\int_{C_1^-} \tr(D^2u\,  \mathbf g) = \int_{C_1^-} b v = \int_{C_r^-(Y) \cap C_1^-} b \left(v-v(Y)\right).	
\]
Therefore,
\begin{align*}
\Abs{\int_{(C_{2R}(Y)\setminus C_R(Y))\cap C_1^-} \tr(D^2u\,  \mathbf g)}\, &\le \norm{b}_{L^1(C_r^-(Y) \cap C_1^-)} \, \norm{v-v(Y)}_{L^\infty(C_r^-(Y)\cap C_1^-)}\\
&\lesssim r \norm{b}_{L^1(C_r^-(Y) \cap C_1^-)}\, [v]_{C^{1/2,1}(C_r^-(Y)\cap C_1^-)}.
\end{align*}
Since $\cP_0^* v =0$ in $C_R(Y)\cap C_1^-$ and $v$ vanishes on $\partial_p^+ C_1^- \cap C_R(Y)$, by an analogy of Lemma~\ref{lem:c_half} up to the boundary and the parabolic version of \cite[Lemma~2]{EM2016}, we get
\begin{align*}
[v]_{C^{1/2,1}(C_r^-(Y)\cap C_1^-)}
&\le  [v]_{C^{1/2,1}(C_{R/2}^-(Y)\cap C_1^-)} \\
&\lesssim R^{-2-\frac{n}2} \norm{v}_{L^2(C_{R}(Y)\cap C_1^-)} \lesssim R^{-2-\frac{n}2} \norm{v}_{L^2(C_1^-)} \\
&\lesssim R^{-2-\frac{n}{2}} \norm{\mathbf g}_{L^2(C_1^-)}= R^{-2-\frac{n}2} \norm{\mathbf g}_{L^2((C_{2R}(Y)\setminus C_R(Y))\cap C_1^-)}.
\end{align*}
Therefore, by duality
\[
\norm{D^2 u}_{L^2((C_{2R}(Y)\setminus C_R(Y))\cap C_1^-)}
\lesssim r R^{-2-\frac{n}2} \norm{b}_{L^1(C_r^-(Y)\cap C_1^-)},
\]
and by H\"older's inequality
\begin{equation}
                            \label{eq4.37b}
\norm{D^2 u}_{L^1((C_{2R}(Y)\setminus C_R(Y))\cap C_1^-)}
\lesssim r R^{-1}\norm{b}_{L^1(C_r^-(Y) \cap C_1^-)}.
\end{equation}
Now let $N$ be the smallest positive integer such that $C_1^-\subset C_{2^{N}cr}(Y)$.
By taking $R=2r,4r,\ldots,2^{N} r$ in \eqref{eq4.37b}, we have
\[
\int_{C_1^-\setminus C_{cr}(Y)} \abs{D^2 u} \lesssim \sum_{k=1}^{N} 2^{-k} \norm{b}_{L^1(C_r^-(Y)\cap C_1^-)} \lesssim \int_{C_r^-(Y)\cap C_1^-} \abs{b}.
\]
Therefore, $T$ satisfies the hypothesis of Lemma~\ref{lem01-stein} and
Lemma \ref{lem-weak11-nd} is proved.
\end{proof}

For $X_0 \in C_3^-$ and $0<r<\frac13$, we decompose $u=v+w$, where $w \in W^{1,2}_2(C_r^-(X_0))$ is a unique solution of the problem
\[
\cP_0 w = -(a^{ij}- \bar a^{ij}) D_{ij} u+ g-\bar g \;\mbox{ in }\;C_r^-(X_0);\quad
w=0 \;\mbox{ on }\;\partial_p^- C_r^-(X_0),
\]
where $\bar a^{ij}=\bar a^{ij}(t)$ and $\bar g=\bar g(t)$ are similarly defined as in \eqref{eq1754bth} and
\[
\cP_0=\partial_t-\bar a^{ij} D_{ij}.
\]
By Lemma~\ref{lem-weak11-nd} with scaling, we have for any $\alpha>0$,
\[
\abs{\set{X\in C_r^-(X_0): \abs{D^2w(X)} > \alpha}}
\lesssim \frac{1}{\alpha}\left(\norm{D^2 u}_{L^\infty(C_r^-(X_0))} \int_{C_r^-(X_0)} \abs{\mathbf{A}-\bar{\mathbf A}^{\textsf x}} +  \int_{C_r^-(X_0)}  \abs{g -\bar g^{\textsf x}}\right),
\]
where we used the notation
\[
\bar{\mathbf A}^{\textsf x}=\bar{\mathbf A}^{\textsf x}(t)=(\bar a^{ij}(t))_{i,j=1}^n,\quad \bar g^{\textsf x}=\bar g^{\textsf x}(t)=\bar g(t).
\]
Therefore, we have
\[
\left(\fint_{C_r^-(X_0)} \abs{D^2w}^{\frac12} \right)^{2} \lesssim \omega_{\mathbf A}^{\textsf x}(r) \,\norm{D^2 u}_{L^\infty(C_r^-(X_0))} + \omega_{g}^{\textsf x}(r).
\]
Since $v=u-w$ satisfies
\[
\cP_0 v =\bar g(t)\;\text{ in }\;C_r^-(X_0),
\]
and $\cP_0$ is an operator with coefficients depending only on $t$, we observe that for any $\mathbf q=(q^{ij}) \in \mathrm{Sym}(n)$, the set of all $n\times n$ symmetric matrices, we have
\[
\cP_0 (D_{ij} v-q^{ij}) =0 \;\mbox{ in }\;C_r^-(X_0),\quad   \forall i, j=1,\ldots, n.
\]
Then, similar to \eqref{int-reg}, we have
\[
[D^2v]_{C^{1/2,1}(C_{r/2}^-(X_0))} \le  N_0r^{-1} \left(\fint_{C_r^-(X_0)} \abs{D^2 v - \mathbf q}^{\frac12} \right)^{2}
\]
and thus, similar to \eqref{eq15.50br}, we obtain (recall $0<\kappa<\frac12$)
\[
\left(\fint_{C_{\kappa r}^-(X_0)} \abs{D^2 v - D^2 v_{X_0, \kappa r}}^{\frac12} \right)^{2} \le 2\kappa r [D^2v]_{C^{1/2,1}(C_{r/2}^-(X_0))} \le 2N_0 \kappa \left(\fint_{C_r^-(X_0)} \abs{D^2 v -\mathbf q}^{\frac12} \right)^{2}.
\]
If we set
\[
\phi(X_0,r):=\inf_{\mathbf q\in \mathrm{Sym}(n)}\left( \fint_{C_r^-(X_0)} \abs{D^2u -\mathbf q}^{\frac12} \right)^{2},
\]
then by the same argument that led to \eqref{eq22.25fr}, for any given $\beta \in (0,1)$, we can find $\kappa \in (0,\frac12)$ so small that
\[
\phi(X_0, \kappa^j r) \le \kappa^{j\beta} \phi(X_0,r) +C \norm{D^2 u}_{L^\infty(C_r^-(X_0))}\,\tilde \omega_{\mathbf A}^{\textsf x}(\kappa^{j} r) + C \tilde \omega_{g}^{\textsf x}(\kappa^{j} r).
\]
Now, by repeating the same line of proof of Theorem~\ref{thm-main-d}, we reach the following estimate:
for $X$, $Y \in C_1^-$ such that $\abs{X-Y} <\frac12$, we have
\begin{multline*}				%\label{eq14.00nd}
\abs{D^2 u(X)-D^2u(Y)} \lesssim \norm{D^2 u}_{L^1(C_4^-)} \,\abs{X-Y}^\beta\\
+\left(\norm{D^2 u}_{L^1(C_4^-)} + \int_0^{1} \frac{\tilde\omega_{\vec g}^{\textsf x}(t)}t \,dt\right) \int_0^{2\abs{X-Y}} \frac{\tilde \omega_{\mathbf A}^{\textsf x}(t)}t \,dt + \int_0^{2\abs{X-Y}} \frac{\tilde \omega_{g}^{\textsf x}(t)}t \,dt.
\end{multline*}

Finally, we investigate the continuity of $\partial_t u$.
Since
\[
\partial_t u = a^{ij} D_{ij}u+g,
\]
it is clear that $\partial_t u \in L^\infty(C_1^-)$.
Moreover, if $\mathbf A$ and $g$ are both continuous over $\overbar{C_1^-}$, $\partial_t u$ is also continuous over $\overbar{C_1^-}$.
The theorem is proved.

\subsection{Proof of Theorem~\ref{thm:dd-wolt}}
						%\label{sec:int-adj}

The proof of this theorem is similar to that of Theorem~\ref{thm:d-wolt}.
We begin with a lemma that is an adjoint version of Lemma~\ref{lem-weak11-nd}.
\begin{lemma}			\label{lem-weak11-adj}
Let $\bar a^{ij}=\bar a^{ij}(t)$ satisfy \eqref{parabolicity}.
Assume that $a^{ij}$ are symmetric and consider the adjoint operator $\cP_0^*$ defined by
\begin{equation}				\label{eq1254bf}
\cP_0^* u := -\partial_t u- D_{ij}(\bar a^{ij}(t) u).
\end{equation}
For $\mathbf{f}=(f^{ij})_{i,j=1}^n\in L^2(C_1^+)$, let $u \in L^2(C^+_2)$ be the unique solution to the adjoint problem
\[
\cP_0^* u= \nabla^2 (\mathbf f 1_{C^+_1})\;\mbox{ in }\; C^+_2;\quad
u=0 \;\mbox{ on } \; \partial_p^+ C^+_2,
\]
where $1_{C^+_1}$ is the usual indicator function.
Then for any $\alpha>0$, we have
\[
\Abs{\set{X \in C_1^+ : \abs{u(X)} > \alpha}}  \lesssim_{n,\lambda}\, \frac{1}{\alpha} \int_{C_1^+} \abs{\mathbf f}.
\]
\end{lemma}
\begin{proof}
By the backward parabolic version of \cite[Lemma~2]{EM2016}, the map $T: \mathbf{f} \mapsto u$ is a bounded linear operator on $L^2(C_1^+)$. Now we apply to this operator $T$ the backward parabolic version of Lemma \ref{lem01-stein}. For this porpoise we set $c=4$, fix $Y \in C^+_1$, $0<r<\frac1 4$ and let $\mathbf  b=(b^{ij})_{i,j=1}^n \in L^2(C^+_1)$ be supported in $C_r^+(Y)\cap C^+_1$ with mean zero.
Let $u \in W^{1,2}_2(C^+_2)$ be the unique adjoint solution of
\[
\cP^*_0 u=\nabla^2 (\mathbf b 1_{C^+_1})\;\mbox{ in }\; C^+_2;\quad
u=0 \;\mbox{ on } \; \partial_p^+ C^+_2.
\]
For any $R\ge 4r$ such that $C^+_1 \setminus C_R(Y) \neq \emptyset$ and $ g \in C^\infty_0((C_{2R}(Y)\setminus C_R(Y))\cap C^+_1)$,
let  $v \in W^{1,2}_2(C^+_2)$ be the strong solution of
\[
\cP_0 v=  g\;\mbox{ in }\; C^+_2;\quad
v=0 \;\mbox{ on } \; \partial_p^- C^+_2.
\]
By duality, we have the identity
\[
\int_{C^+_2} u g = \int_{C^+_1} \tr\left(\mathbf b D^2v\right) = \int_{C^+_r(Y) \cap C^+_1} \tr\left[\mathbf b\left(D^2v-D^2 v(Y)\right)\right].	
\]
Therefore,
\begin{align*}
\Abs{\int_{(C_{2R}(Y)\setminus C_R(Y))\cap C^+_1} u g\,} &\le \norm{\mathbf b}_{L^1(C^+_r(Y) \cap C^+_1)} \norm{D^2 v-D^2 v(Y)}_{L^\infty(C^+_r(Y)\cap C^+_1)}\\
&\lesssim r \norm{\mathbf b}_{L^1(C^+_r(Y) \cap C^+_1)} [D^2 v]_{C^{1/2,1}(C^+_r(Y)\cap C^+_1)}.
\end{align*}
Note that $\cP_0 v =0$ in $C_R(Y)\cap C_2^+$.
Since $Y \in C^+_1$ and $R<2$, the cylinder $C_{R/2}(Y)$ does not intersect with the lateral boundary of $C^+_2$.
Therefore, by Lemma~\ref{lem:c_half} applied to $D^2v$, we have (recall $r\le R/c=R/4$)
\[
[D^2 v]_{C^{1/2,1}(C^+_r(Y)\cap C^+_1)}
\le [D^2v]_{C^{1/2,1}(C_{R/4}(Y)\cap C^+_2)}
\lesssim R^{-2-\frac{n}2} \norm{D^2v}_{L^2(C_{R/2}(Y)\cap C^+_2)}.
\]
On the other hand, by $W^{1,2}_2$ estimate for parabolic systems with coefficients depending only on $t$ (see, for instance, \cite{DK11}), we have
\[
\norm{D^2v}_{L^2(C^+_2)} \lesssim \norm{g}_{L^2(C^+_2)}=  \norm{g}_{L^2((C_{2R}(Y)\setminus C_R(Y))\cap C_1^+)}.
\]
Combining these together, we obtain
\[
[D^2 v]_{C^{1/2,1}(C^+_r(Y)\cap C^+_1)} \lesssim R^{-2-\frac{n}2} \norm{g}_{L^2((C_{2R}(Y)\setminus C_r(Y))\cap C^1_+)}.
\]
Therefore, by the duality, we get
\[
\norm{u}_{L^2((C_{2R}(Y)\setminus C_R(Y))\cap C^+_1)}
\lesssim r R^{-2-\frac{n}2} \norm{\mathbf b}_{L^1(C^+_r(Y)\cap C^+_1)}.
\]
The rest of the proof is the same as that of Lemma~\ref{lem02-weak11} and Lemma \ref{lem-weak11-adj} is proved.
\end{proof}

For $X \in C_3^+$ and $0<r<\frac13$, we decompose $u=v+w$, where $w \in L^2(C_{2r}^+(X_0))$ is the unique solution to the problem (see \cite[Lemma~2]{EM2016})
\begin{equation*}
\left\{
\begin{aligned}
\cP_0^{*} w &= \dv^2 ((\mathbf{A}- \bar{\mathbf A}^{\textsf x}) u1_{C_r^+(X_0)}) + \dv^2 ((\mathbf{g}-\bar{\mathbf g}^{\textsf x}) 1_{C_r^+(X_0)})\;\text{ in }\; C_{2r}^+(X_0),\\
w &=0\;\text{ on }\; \partial_p^+ C_{2r}^+(X_0).
\end{aligned}
\right.
\end{equation*}
where $\bar{\mathbf A}^{\textsf x}=\bar{\mathbf A}^{\textsf x}(t)=(\bar a^{ij}(t))$ and $\bar{\mathbf g}^{\textsf x}=\bar{\mathbf g}^{\textsf x}(t)=(\bar g^{ij}(t))$ are defined as in \eqref{eq1754bth} and $\cP_0^*$ is as in \eqref{eq1254bf}.
By Lemma~\ref{lem-weak11-adj} and scaling, we have
\[
\abs{\set{Y\in C_r^+(X_0): \abs{w(Y)} > \alpha}} \lesssim \frac{1}{\alpha}\left(\norm{u}_{L^\infty(C_r^+(X_0))} \int_{C_r^+(X_0)} \abs{\mathbf{A}-\bar{\mathbf A}^{\textsf x}} +  \int_{C_r^+(X_0)}  \abs{\mathbf g -\bar{\mathbf g}^{\textsf x}}\,\right).
\]
Therefore, we have
\[
\left(\fint_{C_r^+(X_0)} \abs{w}^{\frac12} \right)^{2} \lesssim \omega_{\mathbf A}^{\textsf x}(r) \,\norm{u}_{L^\infty(C_r^+(X_0))} + \omega_{\mathbf g}^{\textsf x}(r).
\]
Note that $v$ is a solution of  $\cP_0^* v = 0$ in $C_r^+(X_0)$ and so is $v-q$ for any constant $q \in \bR$ because $\cP_0^*$ is an operator with coefficients depending only on $t$.
Then, $v-q$ satisfies an estimate corresponding to Lemma~\ref{lem:c_half}; namely,
\[
[v]_{C^{1/2,1}(C_{r/2}^+(X_0))}  \le N_0 r^{-1} \left(\fint_{C_r^+(X_0)} \abs{v-q}^{\frac12} \right)^{2}
\]
and thus, similar to \eqref{eq15.50br}, we obtain that for $\kappa \in (0,\frac12)$
\[
\left(\fint_{C_{\kappa r}^+(X_0)} \abs{v - v_{X_0, \kappa r}}^{\frac12} \right)^{2} \le 2\kappa r [v]_{C^{1/2,1}(C_{r/2}^+(X_0))}  \le 2N_0 \kappa \left(\fint_{C_r^+(X_0)} \abs{v - q}^{\frac12} \right)^{2}.
\]
If we set
\[
\phi(X_0,r):=\inf_{q\in \bR}\left( \fint_{C_r^+(X_0)} \abs{u - q}^{\frac12} \,\right)^{2},
\]
then by the same argument that led to \eqref{eq22.25fr},  for any given $\beta \in (0,1)$, we can find $\kappa \in (0,\frac12)$ so small that
\[
\phi(X_0, \kappa^j r) \le \kappa^{j\beta} \phi(X_0,r) +C \norm{u}_{L^\infty(C_r^+(X_0))}\,\tilde \omega_{\mathbf A}^{\textsf x}(\kappa^{j} r) + C \tilde \omega_{\mathbf g}^{\textsf x}(\kappa^{j} r).
\]
Now, by repeating the same line of argument in Section \ref{sec:int-d}, we reach the following estimate:
for $X$, $Y \in C_1^+$ with $\abs{X-Y} <\frac12$, we have
\begin{multline}\label{eq14.00adj}
\abs{u(X)-u(Y)} \lesssim \norm{u}_{L^1(C^+_4)} \,\abs{X-Y}^\beta\\
+\left(\norm{u}_{L^1(C^+_4)} + \int_0^{1} \frac{\tilde\omega_{\mathbf g}^{\textsf x}(t)}t \,dt\right) \int_0^{2\abs{X-Y}} \frac{\tilde \omega_{\mathbf A}^{\textsf x}(t)}t \,dt + \int_0^{2\abs{X-Y}} \frac{\tilde \omega_{\mathbf g}^{\textsf x}(t)}t \,dt \qedhere
\end{multline}
and proves Theorem \ref{thm:dd-wolt}.

\begin{remark}					\label{rmk4.15}
We shall use the following observation in next section.
\begin{enumerate}[i.]
\item
In Lemma~ \ref{lem02-weak11} and \ref{lem-weak11-nd}, one can replace $C_1^-=(-1,0)\times B_1$ by $(-1,0) \times \mathcal{D}$, where $\mathcal D \subset \bR^n$ is a smooth set satisfying $B^+_{2/3}(0) \subset \cD \subset B^+_{3/4}(0)$.
The same proofs work.
\item
In Lemma~\ref{lem-weak11-adj}, one can replace $C_2^+=(0,4)\times B_2$ by $(0,4)\times 2\mathcal{D}$ and $C_1^+=(0,1)\times B_1$ by $Q_1^+=(0,1)\times B_1^+$.
This can be proved in the same way using the interior estimates for homogeneous parabolic equations with constant coefficients and the following estimate near a flat boundary
\[
\norm{D^2v}_{C^{1/2,1}(Q^+_{R/4})}
\lesssim R^{-4-\frac{n}2} \norm{v}_{L^2(Q^+_{R/2})}
\lesssim R^{-3-\frac{n}2} \norm{D_1v}_{L^2(Q^+_{R/2})},
\]
and after replacing $v$ with $v-cx^1$, which satisfies the same equation and boundary condition,
\[
\norm{D^2v}_{C^{1/2,1}(Q^+_{R/4})}
\lesssim R^{-3-\frac{n}2} \norm{D_1v-c}_{L^2(Q^+_{R/2})}
\lesssim R^{-2-\frac{n}2} \norm{D^2v}_{L^2(Q^+_{R/2})},
\]
where we used the Poincar\'e inequality by taking $c= \fint_{Q_{R/2}^+} D_1 v$.
\end{enumerate}

We would like to mention that there is a flaw in \cite[Lemma~2.23]{DK17}, which is an elliptic version of Lemma~\ref{lem-weak11-adj}.
The lemma there should be stated similar to Lemma~\ref{lem-weak11-adj} so that  one can get around with the boundary estimate for $v$ in the proof.
The result depending on the lemma there, which is \cite[Theorem~1.10]{DK17}, remains intact because it is an interior estimate.
There is a similar flaw in \cite[Lemma~2.4]{DEK18} and it should be stated in such a way as in ii) above.
Again, the conclusion depending on the lemma there, which is \cite[Theorem~1.8]{DEK18}, remains unchanged.
\end{remark}

\section{Boundary estimates}				\label{sec:bdry}

\subsection{Proof of Theorem~\ref{thm-main-d}}
We extend the coefficients and data $b^i$, $c^i$, $d$, $g^i$, and $f$ to zero in $(-T, 0] \times \Omega$ and take the even extension of the $a^{ij}$ with respect to $t=0$.
We note that these extension do not affect the conditions of the theorem.
Moreover, if we also extend $u$ to be zero over $(-T, 0)\times \Omega$, then $u$ satisfies
\[
Pu=\dv \vec g + f \;\text{ in }\;(-T, T) \times \Omega,\quad u=0 \;\text{ on }\; (-T, T) \times \partial \Omega.
\]
In the proof, we shall assume that these extensions have been made.

\begin{proposition}				\label{prop2.5p}
For any open set $\Omega' \subset\subset \Omega$, we have $u \in \mathring C^{1/2,1}(\overbar{\Omega'_T})$.
\end{proposition}
\begin{proof}
Since $a^{ij}$ are $\mathsf{VMO_x}$ in $\overline \Omega_T$ and $\partial\Omega$ is $C^1$, by moving the lower-order terms to the right-hand side of the equation, we can show that $u$, $Du \in L^p(\Omega'_T)$ for any $1<p<\infty$.
Indeed, let $v$ solve
\begin{equation}				\label{eq0224bsun}
v_t -\Delta v=f-c^i D_i u-du\;\text{ in }\;\Omega_T,\quad v=0 \;\text{ on }\;\partial_p^- \Omega_T.
\end{equation}
Then $w=u-v$ satisfies
\begin{equation}				\label{eq0225bsun}
w_t-D_i(a^{ij}D_j w)=D_i h^i\;\text{ in }\;\Omega_T,\quad w=0 \;\text{ on }\;\partial_p^- \Omega_T,
\end{equation}
where
\begin{equation}				\label{eq0210bsun}
h^i:=g^i +b^i u +(a^{ij}-\delta^{ij}) D_j v.
\end{equation}

Note that $f- c^i D_i u-du$ in \eqref{eq0224bsun} belongs to $L^{q_1}=L^{q_1}(\Omega_T)$, where $\frac{1}{q_1}=\frac{1}{q}+\frac{1}{2}$.
By the parabolic $L^p$ estimates, H\"older's inequality, and Sobolev embedding (Lemma~\ref{psobolev}), we have
\begin{equation}			\label{eq0202bsun}
\norm{v}_{L^{p_1}} +\norm{D v}_{L^{p_1}} \lesssim \norm{v}_{W^{1,2}_{q_1}}  \lesssim \norm{f}_{L^q} + \norm{\vec c}_{L^q} \norm{D u}_{L^2}+ \norm{d}_{L^q}\norm{u}_{L^2},
\end{equation}
where
\[
\frac{1}{p_1}:=\frac{1}{q_1}-\frac{1}{n+2}=\frac{1}{q}+\frac{1}{2}-\frac{1}{n+2}.
\]
Therefore, from \eqref{eq0202bsun} and H\"older's inequality, we see that $\vec h=(h^1,\ldots, h^n)$ in \eqref{eq0210bsun} satisfies
\[
\norm{\vec h}_{L^{p_1}}
\lesssim \left(1+ \norm{\vec b}_{L^\infty} + \norm{\vec c}_{L^q}+ \norm{d}_{L^q}\right)  \left( \norm{f}_{L^q}+\norm{\vec g}_{L^\infty}  \right).
\]
Then we apply the parabolic $L^p$ theory (see, for instance, \cite{DK11}) to $w$ and get
\[
\norm{w}_{\cH^1_{p_1}} \le C \left( \norm{f}_{L^q}+ \norm{\vec g}_{L^\infty}\right),
\]
where $C$ is a constant depending only on $n$, $\lambda$, $q$, $Q$, $\partial \Omega$, $\omega_{\mathbf A}^{\textsf x}$, $\norm{\vec c}_{L^q}$, $\norm{d}_{L^q}$, and $\norm{\vec b}_{L^\infty}$.
Therefore, we have $u$, $Du \in L^{p_1}$ and
\[
\norm{u}_{L^{p_1}}+\norm{Du}_{L^{p_1}} \lesssim\norm{f}_{L^q}+\norm{\vec g}_{L^\infty}.
\]
Feeding it back to the equations \eqref{eq0224bsun} and \eqref{eq0225bsun} (i.e. bootstrapping), we eventually get $u,Du \in L^{p}$, for any $1<p<\infty$, and
\[
\norm{u}_{L^p}+\norm{Du}_{L^p} \le C \left(\norm{f}_{L^q}+ \norm{\vec g}_{L^\infty}\right)
\]
as claimed, with $C$ depending additionally on $p$.
It then follows from the equation of $u$ that $u\in \cH^1_p$ for any $p\in (n+2,q)$ and thus by the Sobolev embedding (Lemma~\ref{psobolev2}), we particularly have $u \in C^{\alpha/2, \alpha}(\Omega_T)$ for some $\alpha\in (0,1)$.
Recall that $v$ solves \eqref{eq0224bsun} with $f - c^i D_iu -du \in L^p(\Omega_T)$ for $p \in (n+2,q)$.
By the parabolic $L^p$ theory with Sobolev embedding (Lemma~\ref{psobolev}), we find $D v \in C^{\delta/2,\delta}(\Omega'_T)$ with $\delta=1-\frac{n+2}{p}$.
Therefore, by Lemma~\ref{lem00}, we see that $\vec h \in \textsf{DMO}_{\textsf x} \cap L^\infty$ in the interior of $\Omega_T$.

In summary, $w=u-v$ is a weak solution of \eqref{eq0225bsun}, where $\vec h \in \textsf{DMO}_{\textsf x} \cap L^\infty$, and $\omega_{\vec h}^{\textsf x}$ and $\norm{\vec h}_{L^\infty}$ are completely determined by the given data (namely $n$, $\lambda$, $\Omega$, $T$, $\omega_{\mathbf A}^{\textsf x}$, $q$, $\norm{f}_{L^q}$, $\norm{\vec c}_{L^q}$, $\norm{d}_{L^q}$, $\omega_{\vec b}^{\textsf x}$, $\norm{\vec b}_{L^\infty}$, $\omega_{\vec g}^{\textsf x}$ and $\norm{\vec g}_{L^\infty}$).

By the interior estimates in Section~\ref{sec:int-d}, we find that $w \in \rC^{1/2,1}(\overbar{\Omega'_T})$ and $\norm{w}_{C^{1/2,1}(\Omega'_T)}$ is bounded by a constant $C$ depending only on the above mentioned given data  and $\Omega'$. Since $v \in C^{(1+\delta)/2,1+\delta}(\Omega'_T)$, we see that $u \in \mathring C^{1/2,1}(
\overbar{\Omega'_T})$.
\end{proof}

Next, we turn to $\mathring C^{1/2,1}$ estimate near the lateral boundary.
Under a volume preserving mapping of locally flattening boundary
\[
y=\vec \Phi(x)=(\Phi^1(x),\ldots, \Phi^n(x)),	\quad (\det D\vec \Phi =1)
\]
let $\tilde u(t, y)=u(t, x)$,
which satisfies
\[
\tilde u_t-D_i(\tilde a^{ij} D_j \tilde u+\tilde a^i\tilde u)+\tilde b^iD_i \tilde u= \dv \tilde{\vec g}+\tilde f
\]
with
\begin{gather*}
\tilde a^{ij}(t,y)= D_l\Phi^i(x) D_k\Phi^j(x) a^{kl}(t,x),
\quad \tilde a^{i}(t,y)=D_l \Phi^i(x) a^l(t,x),\\
\tilde b^i(t,y)=D_l \Phi^i(x) b^l(t,x), \quad \tilde c(t,y)=c(t,x),\\
\tilde g^i(t,y)= D_l\Phi^i(x) g^l(t,x),\quad \tilde f(t,y)=f(t,x).
\end{gather*}
Without loss of generality, we assume that the above equation is satisfied in $(-16,0)\times \cD$, with a smooth set $\cD \subset \bR^n$ satisfying $B^+_{4}(0) \subset \cD \subset B^+_{5}(0)$.
By Lemma~\ref{lem00}, we see that the coefficients and data satisfy the same conditions in $(-16,0)\times \cD$. Now write $u=v+w$, where $v$ and $w$ are as in the proof of Proposition~\ref{prop2.5p} with $\Omega_T$ replaced with $(-16,0)\times \cD$. Since by the global parabolic $L^p$ estimate and Sobolev embedding $v \in C^{(1+\delta)/2,1+\delta}((-16,0)\times \cD)$,
it is enough to show that $w$ is $\mathring C^{1/2,1}$ near the flat lateral boundary, and thus we are reduced to prove the following proposition.
Hereafter, for any $\overbar X \in \partial\bR_+^{n+1}=\set{x^1=0}$ we write
\begin{equation*}					%\label{eq:1845ap28}
Q^-_r({\overbar X})=C_r^- \cap \set{x^1>0}+\overbar X,\quad \Delta_r^-({\overbar X})=C^-_r \cap \set{x^1=0} + \overbar X.
\end{equation*}

\begin{proposition}					\label{prop01}
Assume $\mathbf A \in \mathsf{DMO_x}$ and  $\vec g \in \mathsf{DMO_x} \cap L^\infty$.
If $u \in \cH^1_2(Q^-_4)$ is a weak solution of
\[
u_t-D_i(a^{ij} D_j u)= \dv \vec g\;\mbox{ in }\;Q^-_4,
\]
satisfying $u=0$ on $\Delta_4^-(0)$, then $u \in \mathring C^{1/2,1}(\overline{Q^-_1})$.
\end{proposition}

\begin{lemma}				\label{lemma1252bm}
Let $\bar a^{ij}=\bar a^{ij}(t)$ satisfy \eqref{parabolicity}.
Consider the operator $P_0$ defined by
\begin{equation*}			%\label{eq1225bf}
P_0 u := \partial _t u- D_i(\bar a^{ij}(t) D_j u).
\end{equation*}
If $u$ is a weak solution of
\[
P_0 u=0\; \text{ in } \; Q^{-}_r,\quad u=0\; \text{ on } \; \Delta_r^-
\]
then, for any $\alpha \in (0,1)$ we have the estimates
\[
[u]_{C^{\alpha/2,\alpha}(Q_{r/2}^{-})} \le C r^{-\alpha} \left( \fint_{Q_r^{-}} \abs{u}^{\frac12}\right)^{2},
\]
and
\[
[Du]_{C^{\alpha/2,\alpha}(Q_{r/2}^{-})} \le C r^{-\alpha} \left( \fint_{Q_r^{-}} \abs{D_1 u}^{\frac12}\right)^{2},
\]
where $C$ is a constant depending only on $\alpha$, $n$ and $\lambda$.
\end{lemma}
\begin{proof}
We set $r=1$. The general case follows from the scaling.
First, for given $\alpha \in (0,1)$, choose $p \in (n+2,\infty)$ such that $\alpha=1-\frac{n+2}{p}$.
Then, by Sobolev embedding (Lemma~\ref{psobolev}) followed by the parabolic $L^p$ estimates and the Poincar\'e's inequality, we have
\[
\norm{Du}_{C^{\alpha/2,\alpha}(Q^{-}_{1/2})} \le \norm{u}_{W^{1,2}_p(Q^{-}_{1/2})} \lesssim \norm{u}_{L^2(Q^{-}_{3/4})} \lesssim \norm{D_1 u}_{L^2(Q^{-}_{3/4})}.
\]
Note that for $\epsilon >0$ we have
\[
\norm{D_1 u}_{L^2(Q^{-}_{3/4})} \le \epsilon \norm{Du}_{L^\infty(Q^{-}_{3/4})}+ \epsilon^{-3} \norm{D_1 u}_{L^{1/2}(Q^{-}_{3/4})}
\]
and an iteration gives the second inequality. The proof of the first inequality is similar.
\end{proof}

The rest of this subsection is devoted to the proof of Proposition~\ref{prop01}.
We shall derive an a priori estimate of the modulus of continuity of $Du$ by assuming that $u$ is in $C^{1/2,1}(\overbar{Q^-_3})$.
The general case follows from a standard approximation argument.

For $X \in Q^-_4$ and  $r>0$, define
\begin{equation}				\label{eq12.14fo}
\phi(X,r):=\inf_{\vec q \in \bR^n} \left( \fint_{C^-_r(X) \cap Q^{-}_4} \,\abs{Du - \vec q}^{\frac12} \right)^{2}
\end{equation}
and choose a vector $\vec q_{X,r}\in \bR^n$ satisfying
\begin{equation}				\label{eq0955w}
\phi(X,r) = \left( \fint_{C^-_r(X) \cap Q^{-}_4} \,\abs{Du - \vec q_{X,r}}^{\frac12} \right)^{2}.
\end{equation}
Also, for $\overline X \in \Delta_4^-(0)$ and $r>0$, we introduce an auxiliary quantity
\begin{equation}				\label{eq12.14fop}
\varphi(\overline X, r):=\inf_{q\in \bR}\, \left( \fint_{C^-_r(\overline X) \cap Q^{-}_4} \,\abs{Du - q \vec e_1}^{\frac12} \right)^{2}
\end{equation}
and fix a number $q_{\overline X, r} \in \bR$ satisfying
\begin{equation*}				%\label{eq0955wp}
\varphi(\overline X, r) = \left( \fint_{C^-_r(\overline X) \cap Q^{-}_4} \,\abs{Du - q_{\overline X, r}\, \vec e_1}^{\frac12} \right)^{2}.
\end{equation*}
We present a series of lemmas (and their proofs) that will provide key estimates for the proof of Proposition~\ref{prop01}.

\begin{lemma}						\label{lem-01}
Let $\beta \in (0,1)$.
For any $\overline X_0 \in \Delta_3^-(0)$ and $0<\rho \le r \le \frac12$, we have
\begin{equation*}					%\label{eq1225f}
\varphi(\overline X_0, \rho)
\lesssim_{n,\lambda, \beta} \left(\frac{\rho}{r}\right)^{\beta} \varphi(\overline X_0, r)+ \norm{Du}_{L^\infty(Q^{-}_{2r}(\overline X_0))}\,\tilde\omega_{\mathbf A}^{\textsf x}(2\rho) +  \tilde \omega_{\vec g}^{\textsf x}(2\rho),
\end{equation*}
where $\tilde\omega_\bullet^{\textsf x}$ is the Dini function given by \eqref{eq14.27w}.
\end{lemma}
\begin{proof}
Let us write $\overline X_0=(t_0, \bar x_0)$,  $B^+_r(\bar x_0)=B_r(\bar x_0) \cap \set{x^1>0}$ and
\begin{equation*}
                %\label{eq10.26}
\bar a^{ij}=\bar a^{ij}(t)=\fint_{B^+_{2r}(\bar x_0)} a^{ij}(t,x)\,dx, \quad \bar g^i=\bar g^i(t) = \fint_{B^+_{2r}(\bar x_0)} g(t,x)\,dx.
\end{equation*}
Also, as in \cite{DEK18}, we fix a smooth set $\cD \subset \bR^n$ satisfying $B^+_{2/3}(0) \subset \cD \subset B^+_{3/4}(0)$ and denote
\begin{equation}				\label{eq2153ap28}
\cD_{r}(\bar x_0)=r\cD+\bar x_0.
\end{equation}
We decompose $u=v+w$, where $w \in \cH^1_2$ is the weak solution of the problem
\begin{equation*}
\left\{
\begin{aligned}
\partial_t w-D_i (\bar a^{ij} D_j w) &= D_i \left((a^{ij}- \bar a^{ij}) D_j u +(g^i-\bar g^i)\right) \;\text{ on }\; (t_0-4r^2, t_0)\times \cD_{2r}(\bar x_0),\\
w&=0\;\text{ on }\;\partial_p^- \left( (t_0-4r^2, t_0)\times \cD_{2r}(\bar x_0) \right).
\end{aligned}
\right.
\end{equation*}
We apply a modified and scaled version of Lemma~\ref{lem02-weak11} to $w$ to find that for any $\alpha>0$, we have (see Remark~\ref{rmk4.15})
\[
\abs{\set{X\in Q^{-}_r(\overline X_0): \abs{Dw(X)} > \alpha}}
\lesssim \frac{1}{\alpha}\left(\norm{Du}_{L^\infty(Q^{-}_{2r}(\overline X_0))} \int_{Q^{-}_{2r}(\overline X_0)} \abs{\mathbf A-\overline{\mathbf A}} +  \int_{Q^{-}_{2r}(\overline X_0)}  \abs{\vec g -\overline{\vec g}} \,\right),
\]
where we also used that $Q^{-}_r(\overline X_0) \subset (t_0-4r^2, t_0)\times \cD_{2r}(\bar x_0) \subset Q^{-}_{2r}(\overline X_0)$.
Then, we have
\begin{equation}				\label{eq15.50ad}
\left(\fint_{Q^{-}_r(\overline X_0)} \abs{Dw}^{\frac12} \right)^{2} \lesssim \omega_{\mathbf A}^{\textsf x}(2r) \,\norm{Du}_{L^\infty(Q^{-}_{2r}(\overline X_0))} +  \omega_{\vec g}^{\textsf x}(2r).
\end{equation}
On the other hand, $v=u-w$ satisfies
\begin{equation}				\label{eq08.48th}
v_t-D_i(\bar a^{ij} D_j v) =0 \;\mbox{ in }\;  Q^{-}_r(\overline X_0);\quad v=0 \;\mbox{ on }\;  \Delta_r^-(\overline X_0).
\end{equation}
Let $\mu=\frac12 (\beta+1)$ so that we have $0<\beta<\mu<1$.
Then, by Lemma~\ref{lemma1252bm}, we obtain
\[
[D v]_{C^{\mu/2,\mu}(Q^{-}_{r/2}(\overline X_0))}\lesssim r^{-\mu} \left(\fint_{Q^{-}_r(\overline X_0)} \abs{D_{1}v}^{\frac12}\,\right)^{2}.
\]
Since $v-q x_1$, for any $q\in \bR$, also satisfies \eqref{eq08.48th}, we have
\begin{equation}			\label{bdry-reg}
[D v]_{C^{\mu/2,\mu}(Q^{-}_{r/2}(\overline X_0))} \lesssim  r^{-\mu} \left(\fint_{Q^{-}_r(\overline X_0)} \abs{D_1v -q}^{\frac12}\,\right)^{2}, \quad \forall q \in \bR.
\end{equation}
Let $0< \kappa < \frac12$ be a number to be fixed later.
Note that we have
\[
\left(\fint_{Q^{-}_{\kappa r}(\overline X_0)} \Abs{D_1v -D_1 v_{\overline X_0, \kappa r}}^{\frac12} \right)^{2} \le (2\kappa r)^\mu [D v]_{C^{\mu/2,\mu}(Q^{-}_{r/2}(\overline X_0))},
\]
while, for $j=2, \ldots, d$, we have
\[
\left(\fint_{Q^{-}_{\kappa r}(\overline X_0)} \abs{D_j v}^{\frac12} \right)^{2} = \left(\fint_{Q^{-}_{\kappa r}(\overline X_0)} \abs{D_j v - D_jv(\overline X_0)}^{\frac12} \right)^{2} \le (2\kappa r)^\mu [D v]_{C^{\mu/2,\mu}(Q^{.}_{r/2}(\overline X_0))}.
\]
Hence, by \eqref{bdry-reg} we obtain
\begin{equation}				\label{eq15.50b}
\left(\fint_{Q^{-}_{\kappa r}(\overline X_0)} \Abs{Dv - D_1 v_{\overline X_0, \kappa r}\,\vec e_1}^{\frac12} \right)^{2}
\le N_0 \kappa^\mu \left(\fint_{Q^{-}_r(\overline X_0)} \abs{D_1v -q}^{\frac12} \, \right)^{2},\quad \forall q \in \bR,
\end{equation}
where $N_0$ is an absolute constant determined only by $n$, $\lambda$ and $\beta$.
By using the decomposition $u=v+w$, we obtain from \eqref{eq15.50b} that
\begin{align*}
&\left(\fint_{Q^{-}_{\kappa r}(\overline X_0)} \abs{Du - D_1 v_{\overline X_0, \kappa r}\vec e_1}^{\frac12} \right)^{2} \\
&\qquad\qquad \le 2 \left(\fint_{Q^{-}_{\kappa r}(\overline X_0)} \abs{Dv - D_1 v_{\overline X_0, \kappa r} \vec e_1}^{\frac12} \right)^{2}+ 2\left(\fint_{Q^{-}_{\kappa r}(\overline X_0)} \abs{Dw}^{\frac12} \right)^{2}\\
&\qquad \qquad \le  4N_0 \kappa^\mu \left(\fint_{Q^{-}_r(\overline X_0)} \abs{D_1u - q }^{\frac12} \right)^{2} +(4N_0\kappa^\mu + 2\kappa^{-2n-4}) \left(\fint_{Q^{-}_r(\overline X_0)} \abs{Dw}^{\frac12} \right)^{2}.
\end{align*}
Since $q\in \bR$ is arbitrary, by using \eqref{eq15.50ad}, we thus obtain
\[
\varphi(\overline X_0,\kappa r)
 \le 4 N_0 \kappa^\mu \, \varphi(\overline X_0, r) + C(4N_0\kappa^\mu + 2\kappa^{-2n-4}) \left( \omega_{\mathbf A}^{\textsf x}(2r) \,\norm{Du}_{L^\infty(Q^{-}_{2r}(\overline X_0))} + \omega_{\vec g}^{\textsf x}(2r) \right).
\]
Let $\kappa \in (0,\frac12)$ be sufficiently small so that $4N_0 \le \kappa^{\beta-\mu}$; we may take this $\kappa$ and the one in \eqref{eq22.25fr} to be the same.
Then, we obtain
\[
\varphi(\overline X_0,\kappa r) \le \kappa^{\beta} \varphi(\overline X_0,r)+ C \left( \omega_{\mathbf A}^{\textsf x}(2r) \,\norm{Du}_{L^\infty(Q^{-}_{2r}(\overline X_0))} +  \omega_{\vec g}^{\textsf x}(2r) \right).
\]
Note that $\kappa^\beta <1$.
By iterating, for $j=1,2,\ldots$, we get
\begin{multline*}
\varphi(\overline X_0,\kappa^j r) \le \kappa^{j \beta} \varphi(\overline X_0,r)\\
+C \left(\norm{Du}_{L^\infty(Q^{-}_{2r}(\overline X_0))} \sum_{i=1}^{j} \kappa^{(i-1)\beta} \omega_{\mathbf A}^{\textsf x}(2\kappa^{j-i} r) + \sum_{i=1}^{j} \kappa^{(i-1)\beta} \omega_{\vec g}^{\textsf x}(2\kappa^{j-i} r) \right).
\end{multline*}
Therefore, we have
\begin{equation}				\label{eq22.25f}
\varphi(\overline X_0,\kappa^j r) \le \kappa^{j \beta} \varphi(\overline X_0,r) +C \norm{Du}_{L^\infty(Q^{-}_{2r}(\overline X_0))}\,\tilde \omega_{\mathbf A}^{\textsf x}(2\kappa^{j} r) + C \tilde \omega_{\vec g}^{\textsf x}(2\kappa^{j} r),
\end{equation}
where the Dini function $\tilde \omega_{\bullet}^{\textsf x}$ is given by \eqref{eq14.27w}.

Note that \eqref{eq22.25f} is a boundary version of the estimate \eqref{eq22.25fr}  in the previous section and thus, the lemma follows from the exactly same arguments as used in deriving \eqref{eq0947fo} from \eqref{eq22.25fr}.
\end{proof}

\begin{lemma}					\label{lem-02}
Let $\beta \in (0,1)$.
For any $X_0 \in Q^-_3$ and $0<\rho\le r \le \frac14$, we have
\begin{equation}					\label{eq10.39tu}
\phi(X_0, \rho) \lesssim_{n,\lambda, \beta}   \rho^\beta r^{-\beta-n-2} \norm{Du}_{L^1(C^-_{3r}(X_0)\cap Q^-_4)} +\norm{Du}_{L^\infty(C^-_{5r}(X_0)\cap Q^-_4)}\, \hat\omega_{\mathbf A}^{\textsf x}(\rho) +  \hat\omega_{\vec g}^{\textsf x}(\rho),
\end{equation}
where $\hat\omega_\bullet^{\textsf x}(t)$ is the Dini function defined by
\begin{equation}					\label{eq1538th}
\hat \omega_{\bullet}^{\textsf x}(t):= \sup_{s \in [t,1]} \,\left(\frac{t}{s}\right)^\beta\, \tilde\omega_{\bullet}^{\textsf x}(s) \quad (0<t\le 1).
\end{equation}
\end{lemma}
\begin{proof}
In this proof we shall denote
\[
X_0= (t_0, x_0)= (t_0, x^1_0, x^2_0, \ldots, x^{n}_0)\quad\text{and}\quad \overline X_0= (t_0, \bar x_0)= (t_0, 0, x^2_0, \ldots, x^{n}_0).
\]
First, we note that $Q^{-}_{\nu r}(\overline X_0) \subset Q^-_4$ for $\nu \le 4$ and
\begin{equation}				\label{eq12.14fb}
\varphi(\overline X_0, \nu r) \le \left(\fint_{Q_{\nu r}^{-}(\overline X_0)} \abs{Du}^{\frac12}\right)^{2} \lesssim_{n, \nu} r^{-n-2} \norm{Du}_{L^1(Q_{\nu r}^{-}(\overline X_0))}.
\end{equation}
There are three possibilities.
\begin{enumerate}[i.]
\item
$\rho \le r \le x^1_0$\,:\,
We utilize an interior estimate developed in Section~\ref{sec:int-d} as follows.
Since $C^-_r(X_0) \subset Q^-_4$, we observe that $\phi(X_0,\rho)$ is identical to that introduced in Section~\ref{sec:int-d}.
We recall that it satisfies \eqref{eq0947fo}, and thus by \eqref{eq12.14fa} that
\begin{equation}
							\label{eq1447fr}
\phi(X_0, \rho)  \lesssim \left(\frac{\rho}r\right)^{\beta}\, r^{-n-2}\norm{Du}_{L^1(C^-_r(X_0))} + \norm{Du}_{L^\infty(C^-_r(X_0))}\,\tilde \omega_{\mathbf A}^{\textsf x}(\rho) +  \tilde \omega_{\vec g}^{\textsf x}(\rho).
\end{equation}

\item
$x^1_0 \le \rho \le r$\,:\,
Since $C^-_{\rho}(X_0) \cap Q^{-}_4 \subset Q^{-}_{2\rho}(\overline X_0) \subset Q^-_4$, we have
\begin{align*}
\phi(X_0, \rho) &=\left( \fint_{C^-_\rho(X_0) \cap Q^{-}_4} \,\abs{Du - \vec q_{X_0, \rho}}^{\frac12} \right)^{2}  \le \left( \fint_{C^-_\rho(X_0) \cap Q^{-}_4} \,\abs{Du -  q_{\overline X_0, 2\rho}\, \vec e_1}^{\frac12} \right)^{2} \\
&\le  4^{n+2} \left( \fint_{Q^{-}_{2\rho}(\overline X_0)} \,\abs{Du - q_{\overline X_0, 2\rho}\,\vec e_1}^{\frac12} \right)^{2} = 4^{n+2} \varphi(\overline X_0, 2\rho).
\end{align*}
Therefore, by Lemma~\ref{lem-01} and \eqref{eq12.14fb}, we have
\begin{align}
							\nonumber
\phi(X_0, \rho) &\lesssim  \left(\frac{2\rho}{2r}\right)^\beta\varphi(\overline X_0, 2r) + \norm{Du}_{L^\infty(Q^{-}_{4r}(\overline X_0))}\, \tilde \omega_{\mathbf A}^{\textsf x}(4\rho) +  \tilde \omega_{\vec g}^{\textsf x}(4\rho)\\
							\label{eq0919tu}
&\lesssim \rho^\beta r^{-\beta-n-2} \norm{Du}_{L^1(Q^{-}_{2r}(\overline X_0))}+\norm{Du}_{L^\infty(Q^{-}_{4r}(\overline X_0))}\, \hat \omega_{\mathbf A}^{\textsf x}(\rho) +\hat \omega_{\vec g}^{\textsf x}(\rho),
\end{align}
where we used the fact
\[
\tilde\omega_{\bullet}^{\textsf x}(4\rho) \lesssim \hat \omega_{\bullet}^{\textsf x}(\rho)
\]
in the last step.
\item
$\rho \le x^1_0 \le r$\,:\,
Take $R=x^1_0$.
Since $C^-_R(X_0) \subset Q^-_4$ and $C^-_R(X_0) \subset Q^{-}_{2R}(\overline X_0)$, we have
\begin{align*}
\phi(X_0, R) &= \left( \fint_{C^-_{R}(X_0)} \,\abs{Du - \vec q_{X_0, R}}^{\frac12} \right)^{2} \le \left( \fint_{C^-_{R}(X_0)} \,\abs{Du -  q_{\overline X_0, 2R}\, \vec e_1}^{\frac12} \right)^{2} \\
&\le 4^{n+1} \left( \fint_{Q^{-}_{2R}(\overline X_0)} \,\abs{Du -  q_{\overline X_0, 2R}\, \vec e_1}^{\frac12} \right)^{2} = 4^{n+1} \varphi(\overline X_0, 2R).
\end{align*}
Therefore, by \eqref{eq0946fo} and Lemma~\ref{lem-01}, we get
\begin{align*}
\phi(X_0, \rho) &\lesssim \left(\frac{\rho}{R}\right)^{\beta} \varphi(\overline X_0, 2R) + \norm{Du}_{L^\infty(C^-_R(X_0))}\,\tilde \omega_{\mathbf A}^{\textsf x}(\rho) +  \tilde \omega_{\vec g}^{\textsf x}(\rho) \\
&\lesssim  \left(\frac{\rho}{R}\right)^{\beta}  \left\{ \left(\frac{2R}{2r}\right)^{\beta}\varphi(\overline X_0, 2r)+  \norm{Du}_{L^\infty(Q^{-}_{4r}(\overline X_0))}\,\tilde\omega_{\mathbf A}^{\textsf x}(4R) +\tilde\omega_{\vec g}^{\textsf x}(4R)  \right\} \\
&\qquad \qquad +\norm{Du}_{L^\infty(Q^{-}_r(\overline X_0))}\,\tilde \omega_{\mathbf A}^{\textsf x}(\rho) +  \tilde \omega_{\vec g}^{\textsf x}(\rho).
\end{align*}
Notice that from \eqref{eq1538th}, we find
\[
\left(\frac{\rho}{R}\right)^{\beta} \tilde\omega_{\bullet}^{\textsf x}(4R) \lesssim  \hat \omega_{\bullet}^{\textsf x}(\rho),\quad \tilde\omega_{\bullet}^{\textsf x}(\rho) \le \hat \omega_{\bullet}^{\textsf x}(\rho).
\]
Therefore, we have
\begin{equation}
						\label{eq14.24m}
\phi(X_0, \rho) \lesssim  \rho^\beta r^{-\beta-n-2}\norm{Du}_{L^1(Q^-_{2r}(\overline X_0))}+\norm{Du}_{L^\infty(Q^{-}_{4r}(\overline X_0))}\, \hat \omega_{\mathbf A}^{\textsf x}(\rho) + \hat \omega_g^{\textsf x}(\rho).
\end{equation}
\end{enumerate}
We have thus covered all three possible cases and obtained bounds for $\phi(X_0, \rho)$, namely, \eqref{eq1447fr}, \eqref{eq0919tu}, and \eqref{eq14.24m}.
Notice that $\abs{X_0-\overline X_0}=x^1_0 \le r$ in cases ii) and iii).
Therefore, we have
\[
Q^{-}_{\nu r}(\overline X_0) \subset C^-_{(\nu+1)r}(X_0) \cap Q_4^-
\]
and \eqref{eq10.39tu} follows immediately.
We note that $\hat \omega_{\bullet}^{\textsf x}$ is a Dini function; see \cite{DEK18}.
\end{proof}

\begin{lemma}						\label{lem-03}
We have
\begin{equation}					\label{eq10.23mb}
\norm{Du}_{L^\infty(Q^-_2)} \le C \norm{Du}_{L^1(Q^-_4)} + C\int_0^{1} \frac{\hat \omega_{\vec g}^{\textsf x}(t)}t \,dt,
\end{equation}
where $C>0$ is a constant depending only on $n$, $\lambda$ and $\omega_{\mathbf A}^{\textsf x}$.
\end{lemma}
\begin{proof}
The proof is essentially the same as that of \eqref{eq10.23m}.
For $X \in Q^-_3$ and $0<r \le \frac14$, let $\set{\vec q_{X,2^{-k}r}}_{k=0}^\infty$ be a sequence of vectors in $\bR^n$ as given in \eqref{eq0955w}.
Since we have
\[
\abs{\vec q_{X,r} - \vec q_{X, \frac12 r}}^{\frac12} \le
\abs{Du(Y)-\vec q_{X,r}}^{\frac12} + \abs{Du(Y) - \vec q_{X, \frac12 r}}^{\frac12},
\]
by taking average over $Y \in C^-_{r/2}(X) \cap Q^{-}_4$ and then taking squares, we obtain
\[
\abs{\vec q_{X,r} -\vec q_{X,\frac12 r}} \le 2 \phi(X,r) + 2 \phi(X,\tfrac12 r).
\]
Then, by iterating, we get
\begin{equation}				\label{eq1408th}
\abs{\vec q_{X,2^{-k} r} - \vec q_{X, r}} \le 4 \sum_{j=0}^k \phi(X, 2^{-j} r).
\end{equation}
Therefore, by taking $k\to \infty$ in \eqref{eq1408th}, using \eqref{eq10.39tu}  and \cite[Lemma~2.7]{DK17}, we get
\[
\abs{Du(X)-\vec q_{X,r}} \lesssim r^{-n-2}\norm{Du}_{L^1(C^-_{3r}(X)\cap Q^-_4)}+
\norm{Du}_{L^\infty(C^-_{5r}(X)\cap Q^-_4)} \int_0^r \frac{\hat \omega_{\mathbf A}^{\textsf x}(t)}t \,dt+\int_0^r \frac{\hat\omega_{\vec g}^{\textsf x}(t)}t \,dt.
\]
By averaging the obvious inequality
\[
\abs{\vec q_{X, r}}^{\frac12} \le \abs{Du(Y) -\vec q_{X, r}}^{\frac12} + \abs{Du(Y)}^{\frac12}
\]
over $Y \in C^-_r(X)\cap Q^-_4$ and taking square, we get
\[
\abs{\vec q_{X, r}} \le  2 \phi(X, r) + 2 \left(\fint_{C^-_r(X)\cap Q^-_4} \abs{Du}^{\frac12}\right)^{2}.
\]
Combining these together and using
\[
\phi(X,r)  \le \left(\fint_{C^-_r(X)\cap Q^-_4} \abs{Du}^{\frac12}\right)^{2} \lesssim r^{-n-2} \norm{Du}_{L^1(C^-_r(X)\cap Q^-_4)},
\]
we obtain
\[
\abs{Du(X)}  \lesssim r^{-n-2} \norm{Du}_{L^1(C^-_{3r}(X)\cap Q^-_4)} + \norm{Du}_{L^\infty(C^-_{5r}(X)\cap Q^-_4)} \int_0^r \frac{\hat\omega_{\mathbf A}^{\textsf x}(t)}t \,dt+\int_0^r \frac{\hat\omega_{\vec g}^{\textsf x}(t)}t \,dt.
\]
Now, taking supremum for $X\in C^-_r(X_0)\cap Q^-_4$, where $X_0 \in Q^-_3$ and $r\le \frac14$, we have
\begin{multline*}
\norm{Du}_{L^\infty(C^-_r(X_0)\cap Q^-_4)} \\
\le C r^{-n-2} \norm{Du}_{L^1(C^-_{4r}(X_0)\cap Q^-_4)}
+C \norm{Du}_{L^\infty(C^-_{6r}(X_0)\cap Q^-_4)} \int_0^r \frac{\hat\omega_{\mathbf A}^{\textsf x}(t)}t \,dt + C \int_0^r \frac{\hat\omega_{\vec g}^{\textsf x}(t)}t \,dt .
\end{multline*}
We fix $r_0 <\frac14$ such that for any $0<r\le r_0$,
\[
C \int_0^{r} \frac{\hat\omega_{\mathbf A}^{\textsf x}(t)}t \,dt \le \frac1{3^{n+2}}.
\]
Then, we have for any $X_0 \in Q^-_3$ and $0<r\le r_0$ that
\begin{multline*}
\norm{Du}_{L^\infty(C^-_r(X_0)\cap Q^-_4)}\\
\le 3^{-n-2}\norm{Du}_{L^\infty(C^-_{6r}(X_0) \cap Q^-_4)}
+C r^{-n-2} \norm{Du}_{L^1(C^-_{4r}(X_0) \cap Q^-_4)} + C \int_0^r \frac{\hat\omega_{\vec g}^{\textsf x}(t)}t \,dt.
\end{multline*}
For $k=1,2,\ldots$, denote $r_k=3-2^{1-k}$.
Note that $r_{k+1}-r_k=2^{-k}$ for $k\ge 1$ and $r_1=2$.
For $X_0\in Q^-_{r_k}$ and $r\le 2^{-k-3}$, we have $C^-_{6r}(X_0) \cap Q^-_4 \subset Q^{-}_{r_{k+1}}$.
We take $k_0$ sufficiently large such that $2^{-k_0-3}\le r_0$.
It then follows that for any $k\ge k_0$,
\[
\norm{Du}_{L^\infty(Q^-_{r_k})} \le C 2^{k(n+2)} \norm{Du}_{L^1(Q^-_4)} + C \int_0^{1} \frac{\hat\omega_{\vec g}^{\textsf x}(t)}t \,dt+3^{-n-2} \norm{Du}_{L^\infty(Q^-_{r_{k+1}})}.
\]
By multiplying the above by $3^{-k(n+2)}$ and then summing over $k \ge k_0$, we reach
\[
\sum_{k=k_0}^\infty 3^{-k(n+2)}\norm{Du}_{L^\infty(Q^-_{r_k})}
 \le C \norm{Du}_{L^1(Q^-_4)} + C \int_0^{1} \frac{\hat\omega_{\vec g}^{\textsf x}(t)}t \,dt
+\sum_{k=k_0}^\infty 3^{-(k+1)(n+2)} \norm{Du}_{L^\infty(Q^-_{r_{k+1}})}.
\]
Since we assume that $u\in C^{1/2,1}(\overbar{Q^-_3})$, the summations on both sides are convergent and we obtain \eqref{eq10.23mb}.
\end{proof}

\begin{lemma}\label{lem-04} Let $\beta \in (0,1)$.
For any $X_0 \in Q^-_3$ and $0< r \le \frac15$, we have
\begin{multline*}				%	\label{eq11.44fr}
\abs{Du(X_0)-\vec q_{X_0,r}} \lesssim_{n,\lambda, \beta} r^{\beta}\,\norm{Du}_{L^1(C^-_{3/5}(X_0) \cap Q^-_4)}\\
+ \norm{Du}_{L^\infty(C^-_1(X_0)\cap Q^-_4)} \int_0^r \frac{\hat \omega_{\mathbf A}^{\textsf x}(t)}{t}dt  +  \int_0^r \frac{\hat \omega_{\vec g}^{\textsf x}(t)}{t}dt.
\end{multline*}
\end{lemma}
\begin{proof}
Let $\set{\vec q_{X_0,2^{-k} r}}_{k=0}^\infty \in \bR^n$ be as in the proof of Lemma~\ref{lem-03}.
By taking $k \to \infty$ in \eqref{eq1408th}, we get
\[
\abs{Du(X_0) -\vec q_{X_0,r}} \le \sum_{k=0}^{\infty} \,\abs{\vec q_{X_0,2^{-k} r}-\vec q_{X_0,\kappa^{i+1}r}}
\lesssim \sum_{k=0}^\infty \phi(X_0, 2^{-k} r).
\]
Note that by taking $2^{-k}r$ and $\frac{1}{5}$ in place of $\rho$ and $r$ in \eqref{eq10.39tu}, we have
\[
\phi(X_0, 2^{-k}r) \lesssim  2^{-k\beta} r^\beta \norm{Du}_{L^1(C^-_{3/5}(X_0)\cap Q^-_4)}+\norm{Du}_{L^\infty(C^-_{1}(X_0)\cap Q^-_4)}\, \hat\omega_{\mathbf A}^{\textsf x}(2^{-k}r) + \hat\omega_{\vec g}^{\textsf x}(2^{-k}r).
\]
Therefore, the lemma follows from \cite[Lemma~2.7]{DK17}.
\end{proof}

Now, we are ready to show that $u \in C^{1/2,1}(\overbar{Q^-_1})$.
For $X$, $Y \in Q^-_1$, we have
\[
\abs{Du(X)-Du(Y)} \le \abs{Du(X)- \vec q_{X,r}}\ + \abs{\vec q_{X,r}-\vec q_{Y,r}}+\abs{Du(Y)-\vec q_{Y,r}}.
\]
In the case when $\abs{X-Y} <\frac1 {2}$, set $r=2\abs{X-Y}$ and apply Lemma~\ref{lem-04} to get
\begin{multline*}
\abs{Du(X)- \vec q_{X,r}} +\abs{Du(Y)-\vec q_{Y,r}} \\
\lesssim r^{\beta}\,\norm{Du}_{L^1(Q^-_2)} + \norm{Du}_{L^\infty(Q^-_2)}\int_0^r \frac{\hat \omega_{\mathbf A}^{\textsf x}(t)}{t}dt+ \int_0^r \frac{\hat \omega_{\vec g}^{\textsf x}(t)}{t}dt.\end{multline*}
Take the average over $Z \in C^-_r(X)\cap C^-_r(Y) \cap Q^-_4$ in the inequality
\[
\abs{\vec q_{X,r}-\vec q_{Y,r}}^{\frac12} \le \abs{Du(Z)- \vec q_{X,r}}^{\frac12} +\abs{Du(Z)-\vec q_{Y,r}}^{\frac12},
\]
take squares and apply Lemma~\ref{lem-02} to get
\begin{align*}
\abs{\vec q_{X,r}-\vec q_{Y,r}} &\lesssim \phi(X,r) + \phi(Y,r) \\
&\lesssim r^\beta\, \norm{Du}_{L^1(Q^-_2)} + \norm{Du}_{L^\infty(Q^-_2)} \int_0^r \frac{\hat \omega_{\mathbf A}^{\textsf x}(t)}{t}dt+ \int_0^r \frac{\hat \omega_{\vec g}^{\textsf x}(t)}{t}dt.
\end{align*}
Combining these together and using Lemma~\ref{lem-03}, we obtain
\begin{multline}					\label{eq08.33st}
\abs{Du(X)-Du(Y)} \lesssim  \norm{Du}_{L^1(Q^-_2)}\,\abs{X-Y}^\beta \\
+\left( \norm{Du}_{L^1(Q^-_4)} + \int_0^1 \frac{\hat \omega_g(t)}t\,dt \right)
\int_0^{2\abs{X-Y}} \frac{\hat \omega_{\mathbf A}^{\textsf x}(t)}{t}dt
+\int_0^{2\abs{X-Y}} \frac{\hat \omega_{\vec g}^{\textsf x}(t)}{t}dt.
\end{multline}
In case when $\abs{X-Y} \ge \frac1 {2}$, we use
\[
\abs{Du(X) - Du(Y)} \le 2\norm{Du}_{L^\infty(Q^-_1)},
\]
apply Lemma~\ref{lem-03} and still obtain \eqref{eq08.33st}.

Finally, by almost the same proof as for \eqref{eq13.19fs}, one sees that  $u\in \mathring C^{1/2,1}(\overbar{Q_1^+})$.
This completes the proof of Proposition~\ref{prop01} and that of Theorem~\ref{thm-main-d}.
\qed

\begin{remark}
The estimate \eqref{eq08.33st} together with definition of $\hat\omega_{\bullet}(t)$ in \eqref{eq1538th} shows that in the case when $\mathbf A$ and $\vec g$ are $C^{\alpha/2, \alpha}$ functions with $\alpha \in (0,1)$, by choosing $\beta \in (\alpha, 1)$, $Du$ is a $C^{\alpha/2, \alpha}$ function.
In short, we recover the classical Schauder estimates.
This observation also holds for solutions to non-divergence and its adjoint equations.
\end{remark}
\subsection{Proof of Theorem~\ref{thm-main-nd}}
We take the even extension of the coefficients and data $\mathbf A$, $\vec b$, $c$ and extend $u$ and $g$ as zero over $(-T, 0] \times \Omega$. We note that these extensions do not affect the conditions because the extended $u$ satisfies
\[
\mathcal Pu=g  \;\text{ in }\;(-T, T) \times \Omega,\quad u=0 \;\text{ on }\; (-T, T) \times \partial \Omega
\]
for the extended operator $\mathcal P$, whose new coefficients are in $\mathsf{DMO_x}$ or $\mathsf{DMO}$ over $(-T,T)\times\Omega$ in both the first and the second parts of the theorem. As before, in the proof we assume that these extensions have been made.

The idea of proof is essentially the same as that of Theorem~\ref{thm-main-d} and we first establish interior $C^{1,2}$ estimates.
\begin{proposition}				\label{prop2.45p}
For any $p\in (1,\infty)$, we have $u\in W^{1,2}_{p}(\Omega_T)$.
Moreover, for any $\Omega' \subset\subset \Omega$, we have $u \in C^{1,2}(\overbar{\Omega'_T})$.
\end{proposition}
\begin{proof}
By the $L^p$ theory, we have $u \in W^{1,2}_p(\Omega_T)$ for any $1<p<\infty$ and
\[
\norm{u}_{W^{1,2}_p} \le C \norm{g}_{L^\infty} + C \norm{u}_{L^1},
\]
where $C$ is a constant depending only on $n$, $\lambda$, $p$, $\Omega_T$, and the coefficients of $\cP$.
Therefore, by the Morrey-Sobolev embedding, for any $0<\mu<1$, we have
\[
\norm{u}_{C^{\mu/2, \mu}}+\norm{Du}_{C^{\mu/2, \mu}} \le C\norm{g}_{L^\infty}+C\norm{u}_{L^1}.
\]
In particular, we have
\[
\varrho_{Du}(t)+ \varrho_u(t) \le C \left( \norm{g}_{L^\infty} + \norm{u}_{L^1}\right) t^{\mu}.
\]
We rewrite the equation as
\[
\partial_t u-a^{ij} D_{ij} u= g-b^i D_i u-cu =: g_1.
\]
Then $g_1 \in \mathsf{DMO_x} \cap L^\infty$ by Lemma~\ref{lem00}.
Moreover, by \eqref{eq11.39w}, we have
\[
\omega_{g_1}^{\textsf x}(t) \le \omega_g^{\textsf x}(t) + C\left( \norm{g}_{L^\infty} + \norm{u}_{L^1}\right) \left\{\omega_{\vec b}^{\textsf x}(t)+ \omega_c^{\textsf x}(t)+ \left(\norm{\vec b}_{L^\infty} + \norm{c}_{L^\infty}\right)t^\mu \right\}.
\]
Therefore, $\omega_{g_1}^{\textsf x}$ is a Dini function that is completely determined by the given data (namely $n$, $\lambda$, $\Omega$, $T$, $\omega_{\mathbf A}^{\textsf x}$,  $\omega_{\vec b}^{\textsf x}$, $\norm{\vec b}_{L^\infty}$, $\omega_c^{\textsf x}$, $\norm{c}_{L^\infty}$, $\omega_g^{\textsf x}$, and $\norm{g}_{L^\infty}$) and $\norm{u}_{L^1}$.
By Section~\ref{sec:int-nd}, we thus find that $u \in C^{1,2}(\overbar{\Omega'_T})$ and $\norm{u}_{C^{1,2}(\Omega'_T)}$ is bounded by a constant $C$ depending only on the above mentioned given data, $\norm{u}_{L^1(\Omega_T)}$, and $\Omega'$.
\end{proof}

Next, we turn to $C^{1,2}$ estimates near the boundary.
Hereafter, we shall assume that the coefficients $a^{ij}$, $b^i$, $c$, and the data $g$ are in $\mathsf{DMO}$.

Let $g_1$ be as given in the proof of Proposition~\ref{prop2.45p}.
Under a mapping of locally flattening boundary
\[
y=\vec \Phi(x)=(\Phi^1(x),\ldots, \Phi^n(x)),	
\]
let $\tilde u(t,y)=u(t,x)$, which satisfies
\[
\partial_t u-\tilde a^{ij} D_{ij} \tilde u=\tilde g_1+ \tilde b^iD_i\tilde u=:\tilde h,
\]
where
\[
\tilde a^{ij}(t,y)= D_l\Phi^i(x) D_k\Phi^j(x) a^{kl}(t,x),\quad  \tilde b^i(t,y)= D_{kl}\Phi^i(x) a^{kl}(t,x),\quad \tilde g_1(t,y)=g_1(t,x).
\]
By Lemmas~\ref{lem00}, we see that the coefficients $\tilde a^{ij}$ and the data $\tilde h$ are in $\mathsf{DMO}$. Therefore, we are reduced to prove the following.
\begin{proposition}					\label{prop01nd}
Let $\mathbf A$ and $g$ be in $\mathsf{DMO}$.
If $u \in W^{1,2}_2(Q^-_4)$ is a strong solution of
\[
\partial_t u-a^{ij} D_{ij} u= g\;\mbox{ in }\;Q^-_4
\]
satisfying $u=0$ on $\Delta_4^-(0)$, then $u \in C^{1,2}(\overbar{Q^-_1})$.
\end{proposition}

The rest of this subsection is devoted to the proof of Proposition~\ref{prop01nd}.
As in the proof of Proposition~\ref{prop01}, we shall derive an a priori estimate of the modulus of continuity of $D^2u$ and $\partial_t u$ by assuming that $u$ is in $C^{1,2}(\overbar{Q^-_3})$.

In the proof, we shall denote
\[
D_{xx'} u =
 \begin{pmatrix}
  0 & D_{12}u & \cdots & D_{1n}u \\
  0& D_{22}u & \cdots & D_{2n}u \\
  \vdots  & \vdots  & \ddots & \vdots  \\
  0 & 0 & \cdots & D_{nn}u
 \end{pmatrix}
\quad\text{and}\quad
D_{1x'} u =
 \begin{pmatrix}
  0 & D_{12}u & \cdots & D_{1n}u \\
  0& 0 & \cdots & 0 \\
  \vdots  & \vdots  & \ddots & \vdots  \\
  0 & 0 & \cdots & 0
 \end{pmatrix}.
\]
Also, let $\mathbf{U}_0(n)$ and  $\mathbf{U}_1(n)$, respectively,  be the set of all $n\times n$ matrices of the form
\[
 \begin{pmatrix}
  0 & q^{12} & \cdots & q^{1n} \\
  0& q^{22} & \cdots & q^{2n} \\
  \vdots  & \vdots  & \ddots & \vdots  \\
  0 & 0 & \cdots & q^{nn}
 \end{pmatrix}
\quad\text{and}\quad
 \begin{pmatrix}
  0 & q^{12} & \cdots & q^{1n} \\
  0& 0 & \cdots & 0 \\
  \vdots  & \vdots  & \ddots & \vdots  \\
  0 & 0 & \cdots & 0
 \end{pmatrix}.
\]

We shall first show that $D^2_{xx'}u$ and $\partial_t u$ are continuous over $\overbar{Q_1^-}$.
For $X\in Q^-_4$ and $r>0$, define
\[
\phi(X,r):=\inf_{\mathbf q \in \mathbf{U}_0(n)} \left(  \fint_{C^-_r(X) \cap Q^{-}_4} \, \abs{D_{xx'}u - \mathbf q}^{\frac12} \right)^{2} +
\inf_{q \in \bR} \left(  \fint_{C^-_r(X) \cap Q^{-}_4} \, \abs{\partial_t u - q}^{\frac12} \right)^{2}
\]
and fix a matrix $\mathbf q_{X,r}\in \mathbf U_0(n)$ and a number $q_{X,r} \in \bR$ satisfying
\[
\phi(X,r) = \left( \fint_{C^-_r(X) \cap Q^{-}_4} \,\abs{D_{xx'}u - \mathbf q_{X,r}}^{\frac12} \right)^{2} + \left(  \fint_{C^-_r(X) \cap Q^{-}_4} \, \abs{\partial_t u - q_{X,r}}^{\frac12} \right)^{2}.
\]
Also, similar to \eqref{eq12.14fop}, for $\overline X \in \Delta^-_4$ and $r>0$, we introduce an auxiliary quantity
\[
\varphi(\overline X, r):=\inf_{\mathbf q \in \mathbf U_1(n)}\left( \fint_{Q^{-}_r(\overline X)} \abs{D_{xx'}u -\mathbf q}^{\frac12} \right)^{2} + \left(  \fint_{Q^{-}_r(\overline X)} \, \abs{\partial_t u}^{\frac12} \right)^{2}.
\]

The following lemma is in parallel with Lemma~\ref{lem-01}.

\begin{lemma}						\label{lem-01nd}
Let $\beta \in (0,1)$.
For any $\overline X_0 \in \Delta_3^-$ and $0<\rho \le r \le \frac12$, we have
\[
\varphi(\overline X_0, \rho) \lesssim_{n,\lambda, \beta} \left(\frac{\rho}{r}\right)^{\beta} \varphi(\overline X_0,r)+ \norm{D^2u}_{L^\infty(Q^{-}_{2r}(\overline X_0))}\,\tilde\omega_{\mathbf A}(2\rho) +  \tilde\omega_g(2\rho).
\]
\end{lemma}
\begin{proof}
For $\overline X_0=(t_0,\bar x_0) \in \Delta_3^-(0)$ and $0<r \le \frac12$, we decompose $u=v+w$, where $w \in W^{1,2}_2$ is the unique solution of the problem
\begin{equation*}
\left\{
\begin{array}{rcl}
\partial_t w- \bar a^{ij} D_{ij} w= (a^{ij}- \bar a^{ij}) D_{ij} u +g-\bar g &\text{on} & (t_0-4r^2, t_0)\times \cD_{2r}(\bar x_0),\\
w=0 &\text{on} & \partial_p^- \left( (t_0-4r^2, t_0)\times \cD_{2r}(\bar x_0) \right),
\end{array}
\right.
\end{equation*}
where $\cD_r(\bar x_0)$ is as defined in \eqref{eq2153ap28} and
\[
\bar a^{ij} = \fint_{Q^-_{2r}(\overline X_0)} a^{ij}\quad\text{and}\quad \bar g=\fint_{Q^-_{2r}(\overline X_0)} g.
\]
We apply a modified and scaled version of Lemma~\ref{lem-weak11-nd} to $w$ to find that for any $\alpha>0$, we have (see Remark~\ref{rmk4.15})
\begin{multline*}
\abs{\set{X\in Q^{-}_r(\overline X_0): \abs{D^2w(X)} > \alpha}} + \abs{\set{X\in Q^{-}_r(\overline X_0): \abs{\partial_t w(X)} > \alpha}}\\
\lesssim \frac{1}{\alpha}\left(\,\norm{D^2 u}_{L^\infty(Q^{-}_{2r}(\overline X_0))} \int_{Q^{-}_{2r}(\overline X_0)} \abs{\mathbf{A}-\bar{\mathbf A}} +  \int_{Q^{-}_{2r}(\overline X_0)}  \abs{g -\bar g}\,\right).
\end{multline*}
Therefore, we have
\[
\left(\fint_{Q^{-}_r(\overline X_0)} \abs{D^2w}^{\frac12} \right)^{2} + \left(\fint_{Q^{-}_r(\overline X_0)} \abs{\partial_t w}^{\frac12} \right)^{2} \lesssim \omega_{\mathbf A}(2r) \,\norm{D^2 u}_{L^\infty(Q^{-}_{2r}(\overline X_0))} +  \omega_g(2r).
\]
Since $v=u-w$, it satisfies
\[
v_t-\bar a^{ij} D_{ij} v= v_t-D_i (\bar a^{ij} D_j v)= \bar g\;\mbox{ in }\;Q^{-}_r(\overline X_0);\quad v=0\;\mbox{ on }\;\Delta_r^-(\overline X_0),
\]
and thus, we see that $D_k v$ satisfies \eqref{eq08.48th} for $k=2, \ldots, d$. Let $\mu=\frac12 (\beta+1)$ as before.
Then, by \eqref{bdry-reg} and the symmetry of $D^2u$, we have
\[
[D_{xx'}v]_{C^{\mu/2,\mu}(Q^{-}_{r/2}(\overline X_0))}\lesssim r^{-\mu} \left(\fint_{Q^{-}_r(\overline X_0)} \abs{D_{1x'}v-\mathbf q}^{\frac12}\,\right)^{2},\quad \forall \mathbf q\in \mathbf U_1(n).
\]
Also, since $\partial_t v$ satisfies \eqref{eq08.48th}, by Lemma~\ref{lemma1252bm}, we have
\[
[\partial_t v]_{C^{\mu/2,\mu}(Q^{-}_{r/2}(\overline X_0))}\lesssim r^{-\mu} \left(\fint_{Q^{-}_r(\overline X_0)} \abs{v_t}^{\frac12}\,\right)^{2}.
\]
Note that for $2\le i \le j \le n$ and $0<\kappa < \frac12$, we have
\[
\left(\fint_{Q^{-}_{\kappa r}(\overline X_0)} \abs{D_{ij} v}^{\frac12} \right)^{2} = \left(\fint_{Q^{-}_{\kappa r}(\overline X_0)} \abs{D_{ij} v - D_{ij}v(\overline X_0)}^{\frac12} \right)^{2} \le (2\kappa r)^\mu  [D_{xx'}v]_{C^{\mu/2,\mu}(Q^{-}_{r/2}(\overline X_0))}
\]
and
\[
\left(\fint_{Q^{-}_{\kappa r}(\overline X_0)} \abs{\partial_t v}^{\frac12} \right)^{2} = \left(\fint_{Q^{-}_{\kappa r}(\overline X_0)} \abs{\partial_t v - \partial_t v(\overline X_0)}^{\frac12} \right)^{2} \le (\kappa r)^{2 \mu/2} [\partial_t v]_{C^{\mu/2,\mu}(Q^{-}_{r/2}(\overline X_0))}.
\]

Therefore, if we take $\mathbf q_{\overline X_0, \kappa r} \in \mathbf U_1(n)$ whose $(1,i)$ entry is $D_{1i}v_{\overline X_0, \kappa r}$ for  $i=2,\ldots, n$, then similar to \eqref{eq15.50b}, we have
\begin{multline*}
\left(\fint_{Q^{-}_{\kappa r}(\overline X_0)} \abs{D_{xx'}v - \mathbf q_{\overline X_0, \kappa r}}^{\frac12} \right)^{2} + \left(\fint_{Q^{-}_{\kappa r}(\overline X_0)} \abs{\partial_t v}^{\frac12} \right)^{2} \\
\le N_0 \kappa^\mu\left\{ \left(\fint_{Q^{-}_{r}(\overline X_0)} \abs{D_{1x'}v -\mathbf q}^{\frac12}\, \right)^{2} + \left(\fint_{Q^{-}_r(\overline X_0)} \abs{\partial_t v}^{\frac12}\,\right)^{2} \right\}, \quad \forall \mathbf q \in \mathbf U_1(n).
\end{multline*}
By the same argument that led to \eqref{eq22.25f}, there is $\kappa \in (0, \frac12)$ such that
\[
\varphi(\bar X_0, \kappa^j r) \le \kappa^{j \beta} \varphi(\bar X_0, r) +C \norm{D^2 u}_{L^\infty(Q^{+}_{2r}(\bar X_0))}\,\tilde \omega_{\mathbf A}(2\kappa^{j} r) + C \tilde \omega_g(2\kappa^{j} r),
\]
where $\tilde\omega_\bullet(t)$ is the same as in \eqref{eq14.27w}.
The rest of proof is the same as that of Lemma~\ref{lem-01}.
\end{proof}

By modifying the proof of Lemmas \ref{lem-02}, \ref{lem-03} and \ref{lem-04} in a straightforward way, we obtain the following lemmas.

\begin{lemma}						%\label{lem-02nd}
Let $\beta \in (0,1)$.
For any $X \in Q^-_3$ and $0<\rho\le r \le \frac14$, we have
\begin{multline*}
\phi(X, \rho) \lesssim_{n,\lambda,\beta}  \rho^{\beta}\, r^{-\beta-n-2} \left(\norm{D_{xx'}u}_{L^1(C^-_{3r}(X)\cap Q^-_4)}+\norm{\partial_t u}_{L^1(C^-_{3r}(X)\cap Q^-_4)} \right)\\
+\left(\norm{D_{xx'}u}_{L^\infty(C^-_{5r}(X)\cap Q^-_4)}+\norm{\partial_t u}_{L^\infty(C^-_{5r}(X)\cap Q^-_4)}\right) \hat\omega_{\mathbf A}(\rho) + C \hat\omega_g(\rho).
\end{multline*}
\end{lemma}

\begin{lemma}						%\label{lem-03nd}
We have
\[
\norm{D_{xx'} u}_{L^\infty(Q^-_2)} +\norm{u_t}_{L^\infty(Q^-_2)} \le C \left( \norm{D_{xx'} u}_{L^1(Q^-_4)} + \norm{u_t}_{L^1(Q^-_4)}\right)+ C\int_0^{1} \frac{\hat\omega_g(t)}t \,dt,
\]
where $C>0$ is a constant depending only on $n$, $\lambda$ and $\omega_{\mathbf A}$.
\end{lemma}

\begin{lemma}						%\label{lem-04nd}
Let $\beta \in (0,1)$.
For any $X \in Q^-_3$ and $0< r \le \frac15$, we have
\begin{multline*}
\abs{D_{xx'}u(X)-\mathbf q_{X,r}} +\abs{u_t(X)-q_{X,r}}
 \lesssim_{n, \lambda, \beta} r^{\beta}\left(\norm{D_{xx'}u}_{L^1(C^-_{3/5}(X) \cap Q^-_4)}+\norm{\partial_t u}_{L^1(C^-_{3/5}(X) \cap Q^-_4)} \right) \\
+ \left(\norm{D_{xx'}u}_{L^\infty(C^-_1(X)\cap Q^-_4)} +\norm{\partial_t u}_{L^\infty(C^-_1(X)\cap Q^-_4)} \right) \int_0^r \frac{\hat\omega_{\mathbf A}(t)}{t}dt +  \int_0^r \frac{\hat\omega_{g}(t)}{t}dt.
\end{multline*}
\end{lemma}

With the above lemmas at hand, we obtain, similar to \eqref{eq08.33st}, the following estimates for $X$, $Y \in Q^-_1$:
\begin{multline*}					%\label{eq08.33stnd}
\abs{D_{xx'} u(X)-D_{xx'} u(Y)} +\abs{\partial_t u(X)- \partial_t u(Y)}  \le  C\norm{D^2u}_{L^1(Q^-_2)}\,\abs{x-y}^\beta \\
+C\left( \norm{D^2 u}_{L^1(Q^-_4)} + \int_0^1 \frac{\hat\omega_g(t)}t\,dt \right)  \int_0^{2\abs{X-Y}} \frac{\hat\omega_{\mathbf A}(t)}{t}dt +  \int_0^{2\abs{X-Y}} \frac{\hat\omega_{g}(t)}{t}dt,
\end{multline*}
where $\beta \in (0,1)$ is any given number, $C=C(n, \lambda, \omega_{\mathbf A}, \beta)$, and $\hat\omega_\bullet(t)$ is a Dini function as defined in \eqref{eq1538th}.
This shows that $D_{xx'}u$ and $\partial_t u$ are continuous in $\overbar{Q^-_1}$.
We now use the equation to write
\[
D_{11}u = \frac{1}{a^{11}}
 \left(u_t-\sum_{(i,j)\neq (1,1)} a^{ij}D_{ij}u-g \right).
\]
This shows that $D_{11}u$ is also continuous in $\overbar{Q^-_1}$ and
we have shown that $u \in C^{1,2}(\overbar{Q^-_1})$ as desired.
This completes the proof of Proposition~\ref{prop01nd} and that of Theorem~\ref{thm-main-nd}.
\qed

\subsection{Proof of Theorem~\ref{thm-main-adj}}%\label{subsec2.4}
We first prove a boundary lemma and establish the interior $C^{0,0}$ estimates.
\begin{lemma}
                                        \label{lem-bdry-Lip}
Suppose that $u\in L^2(Q^+_2)$ satisfies
\[
-\partial_t u-D_{ij}(a^{ij} u)=\dv^2 \mathbf g\;\text{ in }\; Q^+_2,\quad
u=-\frac{\mathbf g\nu\cdot\nu}{\mathbf A\nu\cdot\nu}\;\text{ on }\;\Delta_2^+,
\]
where $\mathbf A=(a^{ij})$ and $\mathbf g=(g^{ij})$ are constant symmetric matrices and $\mathbf A$ satisfies \eqref{parabolicity}.
Then, we have
\begin{equation}
                                    \label{eq2.02}
[u]_{C^{1/2,1}(Q_{1}^+)} \lesssim_{n,\lambda} \left(\int_{Q_2^+} \abs{u-q}^{\frac12} \right)^{2} ,\quad \forall q \in \bR.
\end{equation}
\end{lemma}
\begin{proof}
First we note that $u$ satisfies
\begin{equation}
                            \label{eq1.59b}
-\partial_t u -a^{ij}D_{ij} u=0\;\text{ in }\; Q^+_2,\quad  u=\text{constant}\;\text{ on }\;\Delta_2^+.
\end{equation}
By the boundary $L^p$ estimate for parabolic equations (cf. \cite{DK11}) via reversing time, we have $u\in W^{1,2}_p(Q^+_R)$ for any $p\in (1,\infty)$ and $R\in (0,2)$.
Obviously, $u-q$ enjoys the same properties for any $q\in \bR$.
Thus, without loss of generality, it is enough to establish \eqref{eq2.02} assuming that $q=0$.

By differentiating \eqref{eq1.59b} in the tangential direction $x_k$ for $k=2, \ldots, n$, we see that $v_k=D_{k}u$ satisfies
\begin{equation*}
-\partial_t v_k-a^{ij}D_{ij} v_k=0\;\text{ in }\; Q^+_2,\quad v_k=0\;\text{ on }\;\Delta_2^+.
\end{equation*}
Again by the boundary $L^p$ estimate for parabolic equations with zero Dirichlet boundary condition (via reversing time) and the Sobolev embedding, we have
\begin{equation}
                                    \label{eq2.09}
\norm{D_{k} u}_{L^\infty(Q_{1}^+)}+\norm{DD_{k} u}_{L^\infty(Q_{1}^+)} \le C \norm{D_{k} u}_{L^2(Q_{3/2}^+)},\quad k=2, \ldots, n.
\end{equation}
Next, by differentiating in the tangential direction again, we find that $D_{ij}u=0$ on $\Delta_2^+$ for $2\le i,j \le n$.
Similarly, we have that $\partial_t u=0$ on $\Delta_2^+$.
Let us denote
\[
\hat{\mathbf A}=(\hat a^{ij})=
\begin{pmatrix}
a^{11} & 0& \cdots & 0 \\
2a^{12}& a^{22}& \cdots & a^{2n} \\
\vdots  & \vdots  & \ddots & \vdots  \\
2a^{1n} & a^{2n} & \cdots & a^{nn}
\end{pmatrix}.
 \]
Then from the equation, we find that
\[
\sum_{j=1}^n \hat a^{1j} D_{1j}u = -\partial_t u -\sum_{i,j=2}^n  a^{ij} D_{ij} u=0\quad\text{on}\ \Delta_2^+.
\]
Therefore, we find that $v_1:=D_1 u$ satisfies the divergence form parabolic equation
\begin{equation*}
-\partial_t v_1- D_i(\hat a^{ij} D_{j} v_1)=0\;\text{ in }\; Q^+_2
\end{equation*}
with the conormal boundary condition $\hat a^{1j} D_{j}v_1=0$ on $\Delta_2^+$.
By the boundary $\cH^1_p$ estimate for  divergence form parabolic equations with the conormal boundary condition (cf. \cite[Theorem 2.5]{DK11b}) and the Sobolev embedding, we have for any $\alpha\in (0,1)$,
\begin{equation}
                                    \label{eq2.09b}
\norm{D_{1}u}_{L^\infty(Q_{1}^+)}+[D_{1}u]_{C^{\alpha/2,\alpha}(Q_{1}^+)} \le C \norm{D_{1}u}_{L^2(Q_{3/2}^+)}.
\end{equation}
Combining \eqref{eq2.09} and \eqref{eq2.09b} yields
\begin{equation}
                                    \label{eq2.13}
\norm{Du}_{C^{\alpha/2,\alpha}(Q_{1}^+)}=\norm{Du}_{L^\infty(Q_{1}^+)}+ [D u]_{C^{\alpha/2,\alpha}(Q_{1}^+)} \le C\norm{Du}_{L^2(Q_{3/2}^+)}.
\end{equation}
Now, \eqref{eq2.02} is obtained from \eqref{eq2.13} in view of the proof of \eqref{eq13.19fs},
the following interpolation inequalities (proved by using the standard mollification technique; see, for instance, \cite[Ch. 3]{Kr96}):
\[
\norm{Du}_{L^2(Q_{3/2}^+)} \le C\norm{Du}_{L^\infty(Q_{3/2}^+)} \le \varepsilon^\alpha [D u]_{C^{\alpha/2,\alpha}(Q_{3/2}^+)}
+C(\varepsilon,n)\, \norm{u}_{L^{1}(Q_{3/2}^+)}
\]
and
\begin{align*}
\norm{u}_{L^1(Q_{3/2}^+)} & \le \norm{u}_{L^\infty(Q_{3/2}^+)}^{\frac12} \int_{Q_{3/2}^+} \abs{u}^{\frac12} \le C\left( \norm{Du}_{L^\infty(Q_{3/2}^+)} + \norm{u}_{L^1(Q_{3/2}^+)}  \right)^{\frac12}\int_{Q_{3/2}^+} \abs{u}^{\frac12}\\
&\le \varepsilon' \norm{Du}_{L^\infty(Q_{3/2}^+)} + \varepsilon' \norm{u}_{L^1(Q_{3/2}^+)} + C(n, \varepsilon') \left(\int_{Q_{3/2}^+} \abs{u}^{\frac12}\right)^2.
\end{align*}
and a standard iteration argument.
\end{proof}

Now let $u$ be the weak adjoint solution to \eqref{master-adj-prob4}. Then, the following holds.

\begin{proposition}\label{prop2.45p-adj}
For any $\Omega' \subset\subset \Omega$,  we have $u \in C^{0,0}(\overbar{\Omega'_T})$.
\end{proposition}
\begin{proof}
We decompose $u=v+w$, where $v$ is the unique $\cH^1_2(\Omega_T)$ solution of
\begin{equation}			\label{eq18.16hgm}
-\partial_t v -\Delta v= D_i(b^i u)+f-cu \;\text{ in }\;\Omega_T,\quad v=0 \;\text{ on }\;\partial_p^+\Omega_T.
\end{equation}
If we set $\mathbf g_1=\mathbf g+ (\mathbf A -\mathbf I)v$, $w=u-v$ satisfies
\begin{equation*}
\begin{cases}
-\partial_t w - D_{ij}(a^{ij}w)=\nabla^2 \mathbf g_1,\ &\text{in}\ \Omega_T,\\
w= -\frac{\mathbf g_1\nu\cdot \nu}{\mathbf{A}\nu\cdot \nu} +\psi, \ &\text{on}\ (0,T)\times \partial\Omega,\\
w(T,\cdot)=\varphi,\ &\text{on}\ \Omega.
\end{cases}
\end{equation*}
Similar to the proof of Proposition~\ref{prop2.5p}, by the parabolic $L^p$ estimates, H\"older's inequality, and Sobolev embedding (Lemma~\ref{psobolev2}), we have
\[
\norm{v}_{L^{p_1}} \le C \left( \norm{\vec b}_{L^q} \norm{u}_{L^2}+\norm{f}_{L^{q/2}}+ \norm{c}_{L^{q/2}} \norm{u}_{L^2} \right),
\]
where $\frac{1}{p_1}=\frac{1}{q}+\frac{1}{2}-\frac{1}{n+2}$.
Therefore, we find that $\mathbf g_1 \in L^{p_1}(\Omega_T)$.
By $L^p$ estimates for parabolic adjoint equation, which follow easily by combining the results in \cite{DK11} and the proof in \cite{EM2016}, we have $w \in L^{p_1}(\Omega_T)$, which in turn implies that $u \in L^{p_1}(\Omega_T)$.

Then, by bootstrapping argument, we have $u \in L^p(\Omega_T)$ for all $p \in (1,\infty)$.
Also, since $\vec b \in L^q$ and $c$, $f \in L^{q/2}$ for $q>n+2$, we obtain from \eqref{eq18.16hgm} that $v\in C^{\delta/2, \delta}(\Omega_T)$ for some $\delta>0$.

Therefore, we find that $\mathbf g_1  \in \mathsf{DMO_x} \cap L^\infty$ by Lemma~\ref{lem00}.
Moreover, $\omega_{\mathbf g_1}^{\textsf x}$ is a Dini function that is completely determined by the given data (namely $n$, $\lambda$, $\Omega$, $T$, $\omega_{\mathbf A}^{\textsf x}$, $q$, $\norm{f}_{L^p}$, $\norm{\vec b}_{L^q}$, $\norm{c}_{L^{q/2}}$, $\omega_{\mathbf g}^{\textsf x}$,  and $\norm{\mathbf g}_{L^\infty}$) and $\norm{u}_{L^2}$.
By Theorem \ref{thm:dd-wolt} applied to the zero extension of $u$ and $\mathbf g$ for $t\ge T$ or to the one of $w$ and $\mathbf g_1$ for $t\ge T$, we find that $u \in C^{0,0}(\overbar{\Omega'_{T}})$ and that $\norm{u}_{C^{0,0}(\overbar{\Omega'_{T}})}$ is bounded by a constant $C$ depending only on the above mentioned given data, $\norm{u}_{L^2(\Omega_T)}$ and $\Omega'$.
\end{proof}

Next, we turn to $C^{0,0}$ estimates near the lateral boundary. Lemma \ref{thm-main-adj2} and Proposition \ref{prop2.45p-adj} show that it suffices to consider the case when $\varphi$ and $\psi$ in Theorem \ref{thm-main-adj} are identically zero.

\begin{lemma}\label{thm-main-adj2}
Let $\Omega$ be bounded Lipschitz domain in $\bR^n$, $T>0$, $\mathbf{A}$ be $\mathsf{DMO_x}$ over $\bR^{n+1}$ and satisfy \eqref{parabolicity}. Then, the weak adjoint solution $u$ to
\begin{equation*}%\label{master-adj-prob3}
\partial_tu +D_{ij}(a^{ij}(t,x)u) =0\ \ \text{over}\ \ \Omega_T,\quad u= \psi \ \ \text{on}\ \ (0,T) \times \partial \Omega, \quad u(T)=\varphi\ \ \text{on}\ \  \Omega,
\end{equation*}
is in $C^{0,0}(\overline\Omega_T)$, when $\varphi$ and $\psi$ define a continuous function over $\partial_p^+\Omega_T$.
\end{lemma}
The proof of Lemma \ref{thm-main-adj2} will be given in section \ref{sec5}.
Thus, we only need to prove Theorem \ref{thm-main-adj} in the case when $\varphi$ and $\psi$ are identically zero, i.e., when $u$ is the weak adjoint solution to
\begin{equation}\label{master-adj-prob4}
\mathcal P^\ast u=\nabla^2\mathbf g\; \text{ in }\; \Omega_T,\quad u= -\frac{\mathbf{g}\nu\cdot \nu}{\mathbf{A}\nu\cdot \nu} \; \text{ on }\; (0,T) \times \partial \Omega, \quad u(T,\cdot)=0\; \text{on}\;  \Omega.
\end{equation}

Let $\mathbf g_1$ be as given in the proof of Proposition~\ref{prop2.45p-adj}.
Under a volume preserving mapping of flattening boundary
\[
y=\vec \Phi(x)=(\Phi^1(x),\ldots, \Phi^n(x))	
\]
as before, let $\tilde w(t,y)=w(t,x)$,
which satisfies
\[
-\partial_t \tilde w-D_{ij}(\tilde a^{ij} \tilde w)=\nabla^2 \tilde{\mathbf g}_1 -D_i(\tilde h^i+\tilde b^i \tilde w),
\]
where
\begin{gather*}
\tilde a^{ij}(t,y)= D_l\Phi^i(x) D_k\Phi^j(x) a^{kl}(t,x),\quad \tilde b^{i}(t,y)=D_{kl}\Phi^i(x) a^{kl}(t,x),\\
\tilde{\mathbf g}_1(t,y)= D \mathbf{\Phi}(x)\tran \mathbf g_1(t,x) D\mathbf{\Phi}(x),\quad \tilde h^i(t,y)=g_1^{kl}(t,x)D_{kl}\Phi^i(x).
\end{gather*}
We may assume without loss of generality that $\vec\Phi$ is a local $C^{1,1}$-diffeomorphism on $\bR^n$, and $\tilde w$ satisfies the equation above in $(-16,0)\times \cD$, with a smooth set $\cD \subset \bR^n$ satisfying $B^+_{4}(0) \subset \cD \subset B^+_{5}(0)$, and the boundary condition on the flat lateral boundary $(-16,0)\times (\overline \cD\cap \partial \bR^n_+)$.

Let us set $\tilde v$ to be a solution to
\[
-\partial_t \tilde v-\Delta \tilde v=D_i(\tilde h^i+\tilde b^i \tilde w)\;\text{ in }\; (-16,0)\times \cD, \quad \tilde v=0\; \text{ on }\;  \partial_p^+\left((-16,0)\times \cD\right).
\]
Recall from the proof of Proposition \ref{prop2.45p-adj} that we have $w \in L^p(\Omega_T)$ for any $p \in (1, \infty)$.
Therefore, we have $\tilde h^i+ \tilde b^i \tilde w \in L^p((-16,0)\times \cD)$ for any $p \in (1, \infty)$, and thus by the parabolic $L^p$ theory and the Sobolev-Morrey embedding (Lemma~\ref{psobolev2}), we have $\tilde v \in C^{\delta/2,\delta}((-16,0)\times \cD)$ for some $\delta>0$.

If we set $\tilde u=\tilde w -\tilde v$ and $\tilde {\mathbf g}=\tilde {\mathbf g}_1+(\mathbf A-\mathbf I) \tilde v$, then $\tilde u$ satisfies
\[
\begin{cases}
-\partial_t \tilde u -D_{ij}(\tilde a^{ij} \tilde u)=\nabla^2 \tilde {\mathbf g}\;&\text{ in }\;(-16,0)\times \cD\\
\tilde u=-\frac{\tilde{\mathbf g}\nu\cdot\nu}{\mathbf A\nu\cdot\nu}\;&\text{ on }\;(-16,0)\times  (\overline \cD\cap \partial \bR^n_+).
\end{cases}\]
By Lemma~\ref{lem00} and Proposition \ref{prop2.45p-adj}, we see that the coefficients $\tilde a^{ij}$ and the data $\tilde{\mathbf g}$ are in $\mathsf{DMO}$.

As before, we are thus reduced to prove the following.
\begin{proposition}					\label{prop01adj}
Let $\mathbf A$ and $\mathbf g$ be in $\mathsf{DMO}$.
If $u \in L^2(Q^+_4)$ is an adjoint solution satisfying
\[
-\partial_t u-D_{ij}(a^{ij} u)= \nabla^2 \mathbf{g}\;\mbox{ in }\;Q^+_4=Q^+_4(0),\quad u=-\frac{\mathbf{g}\nu\cdot \nu}{\mathbf{A}\nu\cdot\nu}\; \text{ on }\;\Delta_4^+(0),
\]
then $u \in C^{0,0}(\overbar{Q^+_1})$.
\end{proposition}

The rest of this subsection is devoted to the proof of Proposition~\ref{prop01adj}.
As in the proof of Propositions~\ref{prop01} and \ref{prop01nd}, we shall derive an a priori estimate of the modulus of continuity of $u$ by assuming that $u$ is in $C^{0,0}(\overbar{Q^+_3})$.
Similar to \eqref{eq12.14fo} and \eqref{eq0955w}, for $X\in Q^+_4$ and $r>0$, we define
\[
\phi(X,r):=\inf_{q \in \bR} \left( \fint_{C^+_r(X) \cap Q^{+}_4} \,\abs{u - q}^{\frac12} \right)^{2}
\]
and fix a number $q_{X,r}\in \bR$ satisfying
\[
\phi(X,r) = \left( \fint_{C^+_r(X) \cap Q^{+}_4} \, \abs{u - q_{X,r}}^{\frac12} \right)^{2}.
\]
The following lemma is in parallel with Lemmas~\ref{lem-01} and \ref{lem-01nd}.

\begin{lemma}						%\label{lem-01nd}
Let $\beta \in (0,1)$.
For any $\overline X_0 \in \Delta_3^+(0)$ and $0<\rho \le r \le \frac12$, we have
\[
\phi(\overline X_0, \rho) \lesssim_{n, \lambda, \beta} \left(\frac{\rho}{r}\right)^{\beta}\,\phi(\overline X_0, r)+ \norm{u}_{L^\infty(Q^{+}_{2r}(\overline X_0))}\,\tilde\omega_{\mathbf A}(\rho) + C \tilde\omega_{\mathbf g}(\rho).
\]
\end{lemma}
\begin{proof}
For $\overline X_0=(t_0, \bar x_0) \in \Delta_3^+(0)$ and $0<r \le \frac12$, we decompose $u=v+w$, where $w \in L^2$ is a unique adjoint solution of the problem
\begin{equation*}
\left\{
\begin{array}{rcl}
-\partial_t w-D_{ij}(\bar a^{ij}  w) = \nabla^2 \left(1_{Q_r^+(\overline X_0)} \left((\mathbf{A}- \bar{\mathbf A}) u+\mathbf{g}-\bar{\mathbf g}\right)\right) &\text{in}& (t_0,t_0+4r^2)\times \cD_{2r}(\bar x_0),\\
w=1_{Q_r^+(\overline X_0)} \left((\bar{\mathbf A}-\mathbf{A}) u+\mathbf{g}-\bar{\mathbf g}\right)\nu\cdot\nu/ \left(\bar{\mathbf A}\nu\cdot\nu \right) &\text{on}&(t_0,t_0+4r^2)\times \partial \cD_{2r}(\bar x_0),\\
w=0&\text{on}& \set{t_0+4r^2} \times \cD_{2r}(\bar x_0),
\end{array}
\right.
\end{equation*}
where $\cD_r(\bar x_0)$ is defined as in \eqref{eq2153ap28} and $\bar{\mathbf A}$ and $\bar{\mathbf g}$ are the constant matrix and the column vector whose entries are defined by
\[
\bar a^{ij} = \fint_{Q^+_{2r}(\overline X_0)} a^{ij}\quad\text{and}\quad \bar g=\fint_{Q^+_{2r}(\overline X_0)} g.
\]
We apply a modified and scaled version of Lemma~\ref{lem-weak11-adj} to $w$ to find that for any $\alpha>0$, we have (see Remark~\ref{rmk4.15})
\[
\abs{\set{X\in Q^{+}_r(\overline X_0): \abs{w(X)} > \alpha}} \lesssim \frac{1}{\alpha}\left(\,\norm{ u}_{L^\infty(Q^{+}_{r}(\overline X_0))} \int_{Q^{+}_{r}(\overline X_0)} \abs{\mathbf{A}-\bar{\mathbf A}} +  \int_{Q^{+}_{r}(\overline X_0)}  \abs{\mathbf{g}-\bar{\mathbf g}}\,\right).
\]
Therefore,
\[
\left(\fint_{Q^{+}_r(\overline X_0)} \abs{w}^{\frac12} \right)^{2} \lesssim \omega_{\mathbf A}(r) \,\norm{u}_{L^\infty(Q^{+}_{r}(\overline X_0))} +  \omega_{\mathbf g}(r).
\]
Since $v=u-w$ satisfies
\[
-\partial_t v-D_{ij}(\bar a^{ij}  v)=\nabla^2\bar{\mathbf g}\;\mbox{ in }\;Q^{+}_r(\overline X_0),\quad v=-\frac{\bar{\mathbf g}\nu\cdot\nu}{\bar{\mathbf A}\nu\cdot\nu}\;\mbox{ on }\;\Delta_r^+(\overline X_0),
\]
by Lemma \ref{lem-bdry-Lip} with scaling, we have
\[
[v]_{C^{1/2,1}(Q^{+}_{r/2}(\overline X_0))} \lesssim r^{-1} \left(\fint_{Q^{+}_r(\overline X_0)} \abs{v -q}^{\frac12}\,\right)^{2},\quad \forall q \in \bR.
\]
Thus similar to \eqref{eq15.50br}, for any $\kappa \in (0,\frac12)$ we have
\[
\left(\fint_{Q^{+}_{\kappa r}(\overline X_0)} \abs{v - {v}_{\bar X_0, \kappa r}}^{\frac12} \right)^{2}
\le N_0 \kappa \left(\fint_{Q^{+}_r(\overline X_0)} \abs{v -q}^{\frac12}\, \right)^{2}, \quad \forall q \in \bR.
\]
By the same argument that led to \eqref{eq22.25fr}, there is $\kappa \in (0, \frac12)$ such that
\[
\phi(\overline X_0, \kappa^j r) \le\kappa^{j \beta}  \phi(\bar X_0, r) +C \norm{u}_{L^\infty(Q^{+}_{r}(\overline X_0))}\,\tilde \omega_{\mathbf A}(\kappa^{j} r) + C \tilde \omega_{\mathbf g}(\kappa^{j} r),
\]
where $\tilde\omega_\bullet(t)$ is the same as in \eqref{eq14.27w}.
The rest of proof is the same as that of Lemma~\ref{lem-01}.
\end{proof}

By modifying the proof of Lemmas \ref{lem-02}, \ref{lem-03}, and \ref{lem-04} in a straightforward way, we obtain the following lemmas.

\begin{lemma}						%\label{lem-02nd}
Let $\beta \in (0,1)$.
For any $X \in Q^+_3$ and $0<\rho\le r \le \frac14$, we have
\[
\phi(X, \rho) \lesssim_{n, \lambda, \beta} \, \rho^{\beta}\, r^{-\beta-n-2} \norm{u}_{L^1(C^+_{3r}(X)\cap Q^+_4)}+\norm{u}_{L^\infty(C^+_{5r}(X)\cap Q^+_4)}\, \hat\omega_{\mathbf A}(\rho) +  \hat\omega_{\mathbf g}(\rho).
\]
\end{lemma}

\begin{lemma}						%\label{lem-03nd}
We have
\[
\norm{u}_{L^\infty(Q^+_2)} \le C \norm{u}_{L^1(Q^+_4)} + C\int_0^{1} \frac{\hat\omega_{\mathbf g}(t)}t \,dt,
\]
where $C>0$ is a constant depending only on $n$, $\lambda$ and $\omega_{\mathbf A}$.
\end{lemma}

\begin{lemma}						%\label{lem-04nd}
Let $\beta \in (0,1)$.
For any $x \in Q^+_3$ and $0< r \le \frac15$, we have
\[
\abs{u(X)-q_{X,r}} \lesssim_{n, \lambda, \beta} r^{\beta}\,\norm{u}_{L^1(C^+_{3/5}(X) \cap Q^+_4)}
+ \norm{u}_{L^\infty(C^+_1(X)\cap Q^+_4)} \int_0^{r} \frac{\hat\omega_{\mathbf A}(t)}{t}dt + \int_0^{r} \frac{\hat\omega_{\mathbf g}(t)}{t}dt.
\]
\end{lemma}

With the above lemmas at hand, we obtain, similar to \eqref{eq08.33st}, the following estimates for $X$, $Y \in Q^+_1$:
\begin{multline*}					%\label{eq08.33stnd}
\abs{u(X)-u(Y)} \le  C\norm{u}_{L^1(Q^+_2)}\,\abs{X-Y}^\beta \\
+C\left( \norm{u}_{L^1(Q^+_4)} + \int_0^1 \frac{\hat\omega_{\mathbf g}(t)}t\,dt \right)  \int_0^{2\abs{X-Y}} \frac{\hat\omega_{\mathbf A}(t)}{t}\,dt+ C\int_0^{2\abs{X-Y}} \frac{\hat\omega_{\mathbf g}(t)}{t}\,dt,
\end{multline*}
where $\beta \in (0,1)$ is any given number, $C=C(d, \lambda, \omega_{\mathbf A}, \beta)$, and $\hat\omega_\bullet(t)$ is the Dini function defined by \eqref{eq1538th}.
We have shown that $u \in C^{0,0}(\overbar{Q^+_1})$ as desired.
This completes the proof of Proposition~\ref{prop01adj} and that of Theorem~\ref{thm-main-adj}.
\qed

\section{Further topics on adjoint and normalized adjoint solutions}	\label{sec5}
This section is devoted to a study of adjoint and normalized adjoint solutions of the parabolic operator in non-divergence form
\begin{equation}\label{E: parabolicsinterminosbajos}
\mathcal P = \partial_t-  a^{ij}(X)D_{ij}
\end{equation}
with $\mathbf A=(a^{ij})$ verifying \eqref{parabolicity}; see  the definition below.
We present some new results when the coefficients $\mathbf A$ are of $\mathsf{DMO_x}$, extending corresponding results in \cite{Es00}.
As a consequence, we give here a proof to Lemma~\ref{thm-main-adj2}.

\begin{definition}			%\label{D:1}
A continuous function $\tilde v$ is said to be a normalized adjoint solution (n.a.s.) for $\mathcal P$ relative to $W$ in an open set $V\subset\bR^{n+1}$ if $\tilde v W$ is an adjoint solution for $\mathcal P$ in $V$, i.e., $\mathcal P^\ast (\tilde v W)=0$ in $V$.
\end{definition}

The study of adjoint and normalized adjoint solution has a rich history.
Properties of adjoint solutions were systemically studied by Sj\"ogren \cite{S1973} in the elliptic setting.
In particular, it was shown in \cite{S1973} that ``weak'' adjoint solutions are continuous when the coefficients are H\"older continuous.
When the coefficients are less regular, say just continuous, then it is no longer true.
They can then be unbounded \cite[p. 833]{Es00} or do not have good restrictions to nice lower dimensional interior sets as the counterexamples in \cite{Ba84b} and \cite{FabesKenig1} show.  However, it turned out normalized adjoint solutions still enjoy some nice properties.
The concept of normalized adjoint solution in the elliptic case were first studied in \cite{Ba84} and \cite{FGMS1988} while some integrability properties adjoint solutions in the parabolic setting were first studied in \cite{FS1984} and \cite{FabesGarofaloSalsa1}.

\subsection{Some known properties of normalized adjoint solutions}
				\label{sec5.1}
Here, we first summarize some known results for normalized adjoint solutions.
These results are quoted from \cite{Es00} but many ideas are gathered from other sources, such as \cite{FS1984, FGMS1988, FabesSafonov1, FSY1, Garofalo1984, SafonovYuan}.
Recall that we write $X=(t,x)$, $Y=(s,y)$, $Z=(\tau, z)$, etc. for points in $\bR^{n+1}$ and denote $\abs{X-Y}=\max(\sqrt{\abs{t-s}}, \abs{x-y})$.
In \cite[Theorem 1.3]{Es00} it is shown that $\mathcal P$ has a global non-negative adjoint solution $W$, i.e., a non-negative function $W$ in $L^1_{\rm loc}(\bR^{n+1})$ satisfying
\begin{equation}\label{E: 2 condidicndeunicidad}
W \left([0,1]\times B_1\right) = \int_{[0,1]\times B_1}W= \abs{B_1},
\end{equation}
and
\begin{equation*}
\mathcal P^{\ast} W =- \partial_s W  -D_{ij} (a^{ij}(Y)W) =0\ \text{ in }\ \bR^{n+1}
\end{equation*}
 in the sense of distributions, i.e.,
\begin{equation*}%\label{E: condicion de ser adjunta}
\int_{\bR^{n+1}} W\mathcal{P} \varphi=0\quad \text{for all }\; \varphi\in C_c^\infty(\bR^{n+1}),
\end{equation*}
with the following additional properties:
\begin{enumerate}
\item
$W$ is a parabolic Muckenhoupt weight in the reverse H\"older class $\mathcal{B}_{\frac{n+1}{n}}(\bR^{n+1})$ with a constant $N = N(\lambda, n)$, i.e.,
\begin{equation}\label{E: Mackenhoupt properety}
\left(\fint_{C_r(Z)}W^{\frac{n+1}{n}}\right)^{\frac{n}{n+1}}\le N\fint_{C_r(Z)}W\quad \text{and}\quad W(C_{2r}(Z))\le NW(C_r(Z))
\end{equation}
for all $Z\in \bR^{n+1}$ and $r>0$.
\item
For every $\tau\in\bR$, $W(\tau, \cdot)$ must be interpreted as a non-negative measure $W(\tau, dy)$ on $\bR^n$ such that
\begin{equation}\label{E: condidebil}
\int_{(\tau,\infty)\times \bR^n}W\,\mathcal{P}\varphi =\int_{\bR^n}\varphi(\tau, y)W(\tau, dy)\;\text{ for all }\; \varphi\in C_c^\infty(\bR^{n+1}).
\end{equation}
\item There is  $N=N(\lambda,n)$ such that
\begin{equation}\label{E:4 condicionea adicionales}
\begin{split}
&W(\tau +\theta r^2, B_{2r}(z))\le NW(\tau, B_{r}(z)),\\
&W(\tau, B_{2r}(z))\le NW(\tau+\theta r^2, B_{r}(z))
\end{split}
\end{equation}
for all $Z \in \bR^{n+1}$, $r>0$ and $0\le \theta\le 1$.
\end{enumerate}

The first paragraph after \cite[Theorem 3.8]{Es00} shows the uniqueness of a global non-negative adjoint solution for $\mathcal P$ verifying \eqref{E: 2 condidicndeunicidad} and \eqref{E:4 condicionea adicionales}, when the coefficients matrix $\mathbf A$ of $\mathcal P$ is continuous or $\mathbf A$ is in parabolic $\mathsf{VMO_x}$ over $\bR^{n+1}$, i.e., $\lim_{r\to 0^+}\sigma_\mathbf A^{\textsf x}(r) = 0$, with
\begin{equation*}
\sigma_\mathbf A^{\textsf x}(r)= \sup_{\substack{Q_{\rho}^-(X)\\  0<\rho \le r}} \fint_{Q_\rho^-(X)} \,\Abs{\mathbf A(s,y)- \bar{\mathbf A}^{\textsf x}_{x,\rho}(s)},
\quad \left(\;\bar{\mathbf A}^{\textsf x}_{x,\rho}(s) :=\fint_{B_\rho(x)} \mathbf A(s,\cdot)\;\right).
\end{equation*}

When $\mathbf A$ is smooth in $\bR^{n+1}$, $W$ is smooth, $W >0$ everywhere, and for $\tilde v$ in $C^{1,2}(V)$,
\begin{equation*}
W^{-1}\mathcal P^\ast(\tilde v\, W) = -\partial_s\tilde v - a^{ij} D_{ij} \tilde v -2 \frac{D_i (a^{ij}W)}{W}\,D_j\tilde v
\end{equation*}
is a backward parabolic operator for which the standard maximum principle holds. In fact, \cite[Theorems 3.7 and 3.8]{Es00} show that n.a.s.'s satisfy a backward Harnack inequality, the strong maximum principle, and the interior H\"older continuity, as stated below.
\begin{lemma}\label{T:1 Theorem 3.7} {\rm (Harnack inequality for n.a.s.)} There is $N=N(\lambda, n)$ such that
\begin{equation*}
\sup_{(\tau + 2r^2,\tau + 3r^2)\times B_r(z)}{\tilde v}\le N(\lambda, n) \inf_{C^+_r(Z)}\tilde v,
\end{equation*}
when $\tilde v$ is a  non-negative normalized adjoint solution relative to
$W$ in $C_{2r}^+(Z)$, $r>0$ and $Z\in \bR^{n+1}$.
\end{lemma}
\begin{lemma}\label{T: 3 Theorem 3.8} There are $N = N(\lambda, n)$ and $\alpha=\alpha(\lambda,n) \in (0,1]$ such that
\begin{equation*}
\abs{\tilde v(Y)-\tilde v(Z)} \le N \left(\frac{\abs{Y-Z}}{r}\right)^{\alpha}
\sup_{C_r^+(Z)} \abs{\tilde v},\ \ \text{when}\ Y\in C_{r}^+(Z),
\end{equation*}
and $\tilde v$ is a n. a. s. relative to $W$ in $C_r^+(Z)$, $r>0$ and $Z\in\bR^{n+1}$.
\end{lemma}

\subsection{Boundary H\"older continuity of normalized adjoint solutions}
The next result shows that n.a.s.\,are H\"older continuous up to Lipschitz lateral boundaries.
This result is not proved in the literature but it was claimed in \cite[Remark 4.1]{Es00}.

\begin{lemma}			\label{T: 4 Theorem 3.9}
Let $\Omega$ be a Lipschitz domain in $\bR^n$, $z\in\partial\Omega$ and $\tau\in\bR$. Let $\tilde v$ be a normalized adjoint
solution relative to $W$ in $Q_r^+(Z)=[\tau, \tau+r^2) \times (B_r(z)\cap \Omega)$ with $\tilde v =0$ on $\Delta_r^+(Z)=[\tau,\tau+r^2) \times ( B_r(z)\cap\partial\Omega)$.
Then, there is $\alpha=\alpha(n,\lambda, \Omega) \in (0,1]$ such that
\begin{equation*}
\abs{\tilde v(Y)} \le N \left(\frac{\abs{Y-Z}}{r}\right)^{\alpha}
\sup_{Q_r(Z)} \abs{\tilde v},\quad \text{when }\; Y\in Q_{r}^+(Z).
\end{equation*}
\end{lemma}
\begin{proof}
Let $\rho \in (0, r/2]$ and let $\tilde u$ be the n.a.s. satisfying
\begin{equation*}
\begin{cases}
\mathcal{P}^\ast\left(\tilde u\,W\right) =0&\text{in }\; Q_{2\rho}^+(Z),\\
\tilde u = 1&\text{in }\; \Delta_{2\rho}^+(Z),\\
\tilde u=0&\text{in }\; \partial_p^+Q_{2\rho}^+(Z)\setminus\Delta_{2\rho}^+(Z).
\end{cases}
\end{equation*}
It suffices to show that there exists $N>0$ independent of $\rho$ such that \begin{equation}\label{E: 5 locorieite}
\tilde u\ge N^{-1}\ \text{on}\ Q_\rho^+(z).
\end{equation}
This is because the following inequality
\begin{equation*}
\textstyle\pm\, \tilde v\le \left(\sup_{Q_{2\rho}^+(Z)}|\tilde v| \right) (1-\tilde u)
\end{equation*}
obviously holds on the backward parabolic boundary $\partial_p^+Q_{2\rho}^+(Z)$,
the maximum principle for n.a.s. shall show that the above inequality also holds inside $Q_{2\rho}^+(Z)$, and thus \eqref{E: 5 locorieite} implies that
\begin{equation*}
\sup_{Q_{\rho}^+(Z)}|\tilde v|\le \theta\sup_{Q_{2\rho}^+(Z)}|\tilde v|,
\end{equation*}
with $\theta = 1-N^{-1}$, for $0<\rho\le r/2$, which in turn implies Lemma \ref{T: 4 Theorem 3.9}.

To derive \eqref{E: 5 locorieite}, recall that there are $\mu >0$ and $z_0$ in $\bR^n$ such that $B_{\mu\rho}(z_0)\subset B_{\rho}(z)\setminus\overline\Omega$. Let then $\tilde w$ be the n.a.s. in $C_{2\rho}^+(Z)$ verifying
\begin{equation*}
\begin{cases}
\mathcal{P}^\ast\left(\tilde w\,W\right) =0\ &\text{in}\ C_{2\rho}^+(Z),\\
\tilde w = 1\ &\text{in}\ \{\tau +4\rho^2\}\times B_{\mu\rho}(z_0),\\
\tilde w=0\ &\text{elsewhere over}\ \partial_p^+C_{2\rho}^+(Z).
\end{cases}
\end{equation*}
The extension of $\tilde w$ to $C_{2\rho}(Z)\cup [\tau +4\rho^2,\tau+8\rho^2]\times B_{\mu\rho}(z_0)$, as $\tilde w\equiv 1$ over  $ [\tau +4\rho^2,\tau+8\rho^2]\times B_{\mu\rho}(z_0)$, makes out of the extended $\tilde w$ a non-negative n.a.s.\,relative to $W$ over $C_{2\rho}(Z)\cup [\tau +4\rho^2,\tau+8\rho^2]\times B_{\mu\rho}(z_0)$. Then, the Harnack inequality for n.a.s. implies there is $N$ such that, $\tilde w\ge N^{-1}$ over $ [\tau,\tau+8\rho^2]\times B_{\mu\rho/2}(z_0)$. A second application to $\tilde w$ of the same Harnack inequality gives that
\begin{equation}\label{E:5 casicasialli}
\tilde w\ge N^{-1}\ \text{over}\  Q_\rho^+(Z).
\end{equation}
 Finally, the maximum principle for n.a.s.\,and the fact that $0\le \tilde w\le 1$ over $C_{2\rho}^+(Z)$, show that $\tilde u\ge \tilde w$ over $Q_{2\rho}^+(Z)$, which combined with \eqref{E:5 casicasialli} gives \eqref{E: 5 locorieite}.
\end{proof}
%%%%%%%%%%%%%%%%%

\begin{lemma}\label{L: acotanecesa}
There is $N= N(\lambda, n)$ such that if $\tilde v$ is a n. a. s. with respect to $W$ in $C_{r}^+(Z)$, then
\begin{equation}\label{eq14.16tu}
\|\tilde v\|_{L^\infty(C^+_{r/2}(Z))} \le \frac{N}{W(C^+_{r}(Z))} \int_{C^+_{r}(Z)} |\tilde v| W\, .
\end{equation}
\end{lemma}
\begin{proof}
Without loss of generality we may assume $r=2$.
Let then, $G_{\bR\times B_\rho(z)}(X,Y)$ be the Green's function for $\mathcal P$ over the cylinder  $\bR\times B_\rho(z)$.
The facts that
\begin{equation*}
\mathcal PG_{\bR\times B_\rho(z)}(\cdot,Y) = -\delta_Y,
\end{equation*}
the Dirac delta function at $Y$ and $G_{\bR \times B_\rho(z)}(X,Y)= 0$, when $X$ is in the lateral boundary of the cylinder $\bR\times B_\rho(z)$ and $Y$ is in its interior, show that the following representation formula holds for $\tilde v\,W$, for $Y$ in $C^+_\rho(Z)$
\begin{equation}\label{E:otraformula}
\begin{split}
\tilde v(Y)W(Y) = &-\int_s^{\tau+\rho^2}\int_{\partial B_\rho(z)}\mathbf A D_x G_{\bR\times B_\rho(z)}(X,Y)\cdot\nu_x\, \tilde v(X)W(X)\,d\sigma_x\,dt\\
&+\int_{B_\rho(z)}G_{\bR\times B_\rho(z)}(\tau+\rho^2,x,Y)\tilde v(\tau+\rho^2,x)W(\tau+\rho^2,x)\,dx,
\end{split}
\end{equation}
with $\nu_x$ the exterior unit normal vector to $\partial B_\rho(z)$.
Thus,
\begin{equation*}
\begin{split}
\tilde v(Y) = &-\int_s^{\tau+\rho^2}\int_{\partial B_\rho(z)}\mathbf A D_x \tilde G_{\bR\times B_\rho(z)}(X,Y)\cdot \nu_x\,\tilde v(X)W(X)\,d\sigma_xdt\\
&+\int_{B_\rho(z)}\tilde G_{\bR\times B_\rho(z)}(\tau+\rho^2,x,Y)\tilde v(\tau+\rho^2,x)W(\tau+\rho^2,x)\,dx,
\end{split}
\end{equation*}
where $\tilde G_{\bR\times B_\rho(z)}(X,Y)$ is the normalized Green's function for $\mathcal P$ over $\bR\times B_\rho(z)$, i.e.,
\begin{equation*}
\tilde G_{\bR\times B_\rho(z)}(X,Y) = \frac{G_{\bR\times B_\rho(z)}(X,Y)}{W(Y)}.
\end{equation*}
The maximum principle shows that
\begin{equation*}
0\le G_{\bR\times B_\rho(z)}(X,Y)\le G(X,Y)\ \text{for}\ X, Y\in \bR\times B_\rho(z),
\end{equation*}
where $G(X,Y)$ stands for the global fundamental solution for $\mathcal P$ over $\bR^{n+1}$  \cite[Theorem 1.3]{Es00}, \cite[Chapter 1, \S 6]{Friedman1}. Also, the barrier arguments in \cite[Theorem 3.3]{Es00} and \cite[Theorem 3.1]{Garofalo1984}, the Gaussian bounds for the normalized global fundamental solution of $\mathcal P$,
\begin{equation*}
\tilde G(X,Y) = \frac{G(X,Y)}{W(Y)}
\end{equation*}
in \cite[Theorem 1.3]{Es00}, together with \eqref{E:4 condicionea adicionales} show that
\begin{equation*}
|\mathbf A D_x \tilde G_{\bR\times B_\rho(z)}(X,Y)\cdot\nu_x| +\tilde G_{\bR\times B_\rho(z)}(\tau+\rho^2,x,Y)\le \frac N{W(\tau,B_2(z))}\,
\end{equation*}
for $X\in [\tau,\tau+\rho^2]\times \partial B_\rho(z)$, $x\in B_\rho(z)$ and $Y\in C^+_{1}(Z)$, when $\frac32\le\rho\le 2$. Thus,
\begin{equation}\label{E:6 average}
\begin{split}
|\tilde v(Y)|\le\ &\frac N{W(\tau,B_2(z))}\int_{[\tau,\tau+\rho^2]\times \partial B_\rho(z)}|\tilde v|\,W\,d\sigma_x\,dt\\
+&\frac N{W(\tau,B_2(z))}\int_{B_\rho(z)}|\tilde v(\tau+\rho^2,x)|W(\tau+\rho^2,x)\,dx,
\end{split}
\end{equation}
when $\frac32\le\rho\le 2$ and $Y\in C^+_1(Z)$. After taking the average of \eqref{E:6 average} with respect to $\rho$ over the interval $[\frac32,2]$, we get \eqref{eq14.16tu}  for $r=2$ from \eqref{E:4 condicionea adicionales}. Other cases follow by rescaling.
  \end{proof}

In some of the next lemmas we need to deal with normalized adjoint solutions of $\mathcal P$ that are normalized with respect to a different non-negative adjoint solution of $\mathcal P$. In particular, normalized with respect to the unique weak adjoint solution $V$ in $L^p(C^+_{2})$, $1<p < \infty$, to the problem
  \begin{equation}\label{E: definition20}
  \mathcal P^\ast V = 0\ \text{ on }\ C^+_{2},\quad V=1\ \text{ on }\ \partial_p^+C^+_{2},
  \end{equation}
 i.e., the unique $V$ in $L^p(C^+_{2})$ such that
 \begin{equation}\label{eq13.53pder}
\int_{C^+_{2}} V\,\mathcal P \varphi=\int_{B_{2}} \varphi(T)-\int_{[0,4]\times\partial B_{2}}\mathbf A D \varphi\cdot\nu
\end{equation}
for any $v \in W^{1,2}_{p'}(C^+_{2})\cap \mathring{\mathcal H}^1_{p'}(C^+_{2})$, with $\frac{1}{p}+\frac{1}{p'}=1$. Then, if we let $G_{\bR\times B_{2}}(X,Y)$ denote the Green's function for $\mathcal P$ over $\bR\times B_{2}$, by the same reasons as in \eqref{E:otraformula}, we have the following representation formula for $V$ when $Y$ is in $C^+_{2}$
\begin{equation}\label{E:fmulila}
V(Y) = -\int_s^{4}\int_{\partial B_{2}}\mathbf A D_x G_{\bR\times B_{2}}(X,Y)\cdot\nu_x\, d\sigma_xdt+\int_{B_{2}}G_{\bR\times B_{2}}(4,x,Y)\, dx.
\end{equation}
The maximum principle shows that
\begin{equation}\label{E: funciones de gress}
0\le G_{\bR\times B_{2}}(X,Y)\le G(X,Y)\ \text{for}\ X, Y\in \bR\times B_{2},
\end{equation}
where $G(X,Y)$ stands for the global fundamental solution for $\mathcal P$ over $\bR^{n+1}$  \cite[Theorem 1.3]{Es00}, \cite[Chapter 1, \S 6]{Friedman1}. Also, the barrier arguments in \cite[Theorem 3.3]{Es00} and \cite[Theorem 3.1]{Garofalo1984}, the Gaussian bounds for the normalized global fundamental solution of $\mathcal P$
\begin{equation*}
\tilde G(X,Y) = \frac{G(X,Y)}{W(Y)}
\end{equation*}
in \cite[Theorem 1.3]{Es00}, together with \eqref{E: 2 condidicndeunicidad} and \eqref{E:4 condicionea adicionales}, show that there is $N=N(\lambda, n)$ such that
\begin{equation}\label{E: otudesi}
|\mathbf A D_x \tilde G_{\bR\times B_{2}}(X,Y)\cdot\nu_x| +\tilde G_{\bR\times B_{2}}(4,x,Y)\le N,
\end{equation}
when $X\in(0,4)\times \partial B_{2}$, $x\in B_{2}$ and $Y\in C^+_{1}$.
Then,  the combination of \eqref{E:fmulila}, \eqref{E: funciones de gress} and \eqref{E: otudesi} imply that
\begin{equation}\label{E: ddd1}
V(Y)\le NW(Y)\ \text{when}\  Y\in C^+_{1}.
\end{equation}
At the same time, the Harnack inequality for normalized adjoint solutions in Lemma \ref{T:1 Theorem 3.7}, the fact that the first integral in \eqref{E:fmulila} is non-negative, together with \eqref{E: 2 condidicndeunicidad} and \eqref{E:4 condicionea adicionales}, imply that
\begin{equation}\label{E: ddd3}
 N\inf_{C^+_1}\tilde V\,  \ge \frac{\int_{(2,3)\times B_1}V}{W((2,3)\times B_1)}
 \ge N^{-1}\int_{B_{1/2}} \int_{(2,3)\times B_1}G_{B_2}(4,x,Y)\, dY\,dx
\end{equation}
 for some $N= N(\lambda,n)$.
Now, for $2\le s\le 3$, the function
\begin{equation*}
u(t,x)= \int_{B_1}G_{\bR\times B_2}(X,s,y)\,dy
\end{equation*}
is the solution to
\begin{equation*}
\begin{cases}
\mathcal Pu =0\, &\text{in}\ (s,\infty)\times B_2,\\
u=0\ &\text{on}\ (s,\infty)\times\partial B_2,\\
u(s) = 1_{B_1}\ &\text{on}\ B_2
\end{cases}
\end{equation*}
and its extension as $u\equiv 1$ over $(0,s]\times B_1$, makes out of the extended $u$ a non-negative solution to $\mathcal Pw =0$ over $(0,4]\times B_1$. Then, the Harnack inequality for non-negative forward solutions to parabolic equations \cite{KrSaf81} imply that $Nu(x,4)\ge 1$, when $x\in B_{1/2}$ and with $N = N(\lambda,n)$. Then, the later combined with \eqref{E: ddd3} and \eqref{E: ddd1} show that
\begin{equation}\label{E: aproximacion}
N^{-1}W(Y)\le V(Y)\le NW(Y),\ \text{when}\ Y\in C^+_1,
\end{equation}
with $N = N(\lambda, n)$.

Now let $v\in L^1_{\rm loc}(C^+_{2r}(Z))$ be an adjoint solution for $\mathcal P$ over $C^+_{2r}(Z)\subset C_1^+$ . Then, after writing $v/V$ as $\tilde v/\tilde V$, we get
\begin{equation*}
\sup_{C^+_r(Z)} \Abs{\frac{v}{V}} \le \frac{\sup_{C^+_r(Z)}\abs{\tilde v}}{\inf_{C^+_r(Z)} \tilde V}.
\end{equation*}
From Lemma \ref{L: acotanecesa} and \eqref{E: aproximacion} we have
\begin{equation*}
\sup_{C^+_r(Z)}{|\tilde v|}\le \frac{N}{W(C^+_{2r}(Z))} \int_{C^+_{2r}(Z)}|v|
\end{equation*}
and the last two inequalities and \eqref{E: aproximacion} show that
\begin{equation*}
\sup_{C^+_r(Z)}{\left |\frac{v}{V}\right |}\le \frac{N}{V(C^+_{2r}(Z))} \int_{C^+_{2r}(Z)}|v|\, .
\end{equation*}
Also, \eqref{E: aproximacion} shows that when $v\ge 0$ over $C^+_{2r}(Z)$
\begin{equation*}
\sup_{(\tau+2r^2,\tau+3r^2)\times B_r(z)}\frac{v}{V}\simeq \sup_{(\tau+2r^2,\tau+3r^2)\times B_r(z)}{\tilde v},\quad \inf_{C^+_r(Z)}{\tilde v}\simeq\inf_{C^+_r(Z)}\frac{v}{V}
\end{equation*}
and from Lemma \ref{T:1 Theorem 3.7},
\begin{equation*}
\sup_{(\tau+2r^2,\tau+3r^2)\times B_r(z)}{\tilde v}\le N\inf_{C^+_r(Z)}{\tilde v},
\end{equation*}
which combined with the previous inequality gives that
\begin{equation*}
\sup_{(\tau+2r^2,\tau+3r^2)\times B_r(z)} \frac{v}{V}\le N\inf_{C^+_r(Z)}\frac{v}{V}
\end{equation*}
with $N$ as before.

\begin{remark}\label{R: una nota}
The latter shows that Lemmas \ref{T:1 Theorem 3.7}, \ref{T: 3 Theorem 3.8}, and \ref{L: acotanecesa} also hold when one replaces normalized adjoint solutions with respect to $W$ over $C^+_{2r}(Z)\subset C^+_1$ by normalized adjoint solutions with respect to $V$ over $C^+_{2r}(Z)$.
\end{remark}

%%%%%%%%%%%%%%%%%%%%%%%%%

\subsection{Proof of Lemma~\ref{thm-main-adj2} and related results}

As explained in \cite[Remark 4.1]{Es00}, a standard consequence of the backward Harnack inequality for normalized adjoint solutions in Lemma \ref{T:1 Theorem 3.7} and of the interior and boundary H\"older continuity in Lemmas \ref {T: 3 Theorem 3.8} and \ref{T: 4 Theorem 3.9} is that when the coefficients of $\mathcal P$ are $\mathsf{VMO_x}$ over $\bR^{n+1}$, $\Omega$ is a bounded Lipschitz domain in $\bR^n$ and $\tilde\rho\in C^{0,0}(\partial_p^+\Omega_T)$, by using an approximation argument there is always a unique n.a.s.\,$\tilde v$ in $C^{0,0}(\overbar\Omega_T)$ verifying $\tilde v=\tilde\rho$ on $\partial_p^+\Omega_T$ with a modulus of continuity determined by $n$, $\lambda$ and the modulus of continuity of $\tilde\rho$ over $\partial_p^+\Omega_T$.
On the other hand, Theorem \ref{thm:dd-wolt} shows that when $\mathbf A$ is $\mathsf{DMO_x}$ over $\bR^{n+1}$, the global (non-negative) adjoint solution $W$ of $\mathcal P$ is continuous over any compact set $K\subset\bR^{n+1}$  with a modulus of continuity which depends only on $K$, $n$, $\lambda$ and $\omega_{\mathbf A}^{\textsf x}$.

Also, Lemma \ref{lem8.06} below shows that $W$ is bounded above and below on any compact set $K$ of $\bR^{n+1}$, with lower and upper bounds depending on $\lambda$, $n$, $K$ and $\omega_{\mathbf A}^{\textsf x}$. Therefore, if $\rho\in C^{0,0}(\partial_p^+\Omega_T)$ is the continuous function defined over $\partial_p^+\Omega_T$ by the pair of functions $\psi$ and $\varphi$ in Lemma \ref{thm-main-adj2}, by \cite[Remark 4.1]{Es00}, there is a unique n.a.s.\,$\tilde v$ which satisfies
\begin{equation*}
\partial_s\left(\tilde vW\right)+D_{ij}  (a^{ij}(Y)\tilde v W ) = 0 \ \ \text{in}\ \ \Omega_T,\quad \tilde v =\tilde\rho=\rho/W\ \ \text{on}\ \ \partial_p^+\Omega_T.
\end{equation*}
Moreover, $\tilde v\in C^{0,0}(\overbar\Omega_T)$ with a modulus of continuity controlled by $n$, $\lambda$, the Lipschitz character of $\Omega$ and the modulus of continuity of $\tilde\rho$.
The latter in turn is controlled by $n$, $\lambda$, $\omega_{\mathbf A}^{\textsf x}$ and the modulus of continuity of $\rho$. It is then clear that $v = \tilde vW$ satisfies the required properties in Lemma \ref{thm-main-adj2}.
This completes the proof of Lemma~\ref{thm-main-adj2}.
\qed

\begin{lemma}\label{lem8.06}
Assume the coefficients $\mathbf{A}$ of $\mathcal P$ are $\mathsf{DMO_x}$ over $\bR^{n+1}$ and
let $W$ be the unique global non-negative adjoint solution associated to $\mathcal P$ right above. Then, there is $N(\lambda, n, \omega_{\mathbf A}^{\textsf x})$ such that
\begin{equation*}
N^{-1}\fint_{C^+_r(Z)}W\le W(Y)\le N\fint_{C^+_r(Z)}W\ ,\quad \text{when}\ Y\in C^+_r(Z)\ \text{and}\ 0<r\le 1.
\end{equation*}
\end{lemma}
\begin{proof} First, according with Theorem \ref{thm:dd-wolt}, $W$ is continuous over $\bR^{n+1}$ and from the rescaled and translated version of \eqref{eq14.00adj} we find that for any $\beta \in (0,1)$, there are a constant $N>0$ depending only on $\lambda$, $n$, $\beta$, and $\omega_{\mathbf A}^{\textsf x}$ such that
\begin{equation}\label{eq14.00adjs}
\abs{W(X)-W(Y)} \le Nr^{-n-2}\norm{W}_{L^1(C^+_{4r}(Z))}\left[\left(\frac{\abs{X-Y}}{r}\right)^\beta+\int_0^{2\abs{X-Y}} \frac{\tilde \omega_{\mathbf A}^{\textsf x}(t)}t \,dt\right]
\end{equation}
for $X$, $Y \in C_r^+(Z)$, $0<r\le 1$, where $\tilde \omega_{\mathbf A}^{\textsf x}(t)$ is a Dini function given as in \eqref{eq14.27w}.

Taking the average with respect to $Y$ over $C_r^+(Z)$ of the triangle inequality
\[
W(X) \le \abs{W(X)-W(Y)} + W(Y)
\]
and using \eqref{eq14.00adjs} and recalling \eqref{E: Mackenhoupt properety}, we get
\begin{equation*}
\sup_{C_r^+(Z)}W \lesssim r^{-n-2} \norm{W}_{L^1(C^+_{r}(Z))}\ \text{for all}\ Z\in \bR^{n+1}
\end{equation*}
Assuming by contradiction that the lower bound is false, after translations and scalings, compactness  lead to the existence of a parabolic operator $\mathcal P$ with coefficients $\mathbf A$ in $\mathsf{DMO_x}$  over $\bR^{n+1}$  satisfying \eqref{parabolicity}, whose unique global adjoint solution $W$ verifying \eqref{E: 2 condidicndeunicidad} and \eqref{E:4 condicionea adicionales}, vanishes at the origin.  Let then $\delta \in (0,1/2)$ be a small constant to be specified later.
From \eqref{eq14.00adjs} with $Z=0$, for any $r\in (0,1]$, we have
\[
\fint_{C^+_{\delta R}}W\le N\left(\delta^\beta+\int_0^{\delta R}
\frac{\tilde\omega_{\mathbf A}^{\textsf x}(t)}{t} \,dt \right) \fint_{C^+_{4R}}W,
\]
where $N$ is independent of $\delta$.
We then fix $\delta$ small such that $2N\delta^\beta\le \delta^{\beta/2}$ and after we find $R_\delta$ such that
\[
2N\int_0^{\delta R}\frac{\tilde\omega_{\mathbf A}^{\textsf x}(t)}{t} \,dt\le \delta^{\beta/2},\ \ \text{when}\ R\le R_{\delta}.
\]
We obtain
\[
\fint_{C^+_{\delta R}}W\le \delta^{\beta/2} \fint_{C^+_{4R}}W \quad \text{for }\; R\le R_{\delta}.
\]
By iteration, we deduce $\fint_{C^+_{r}(Z)}W\le Nr^{\beta/2}$, when $r$ is small.
This, however, contradicts Lemma \ref{lem8.0671} which implies that for all $\epsilon >0$, there is $N_{\epsilon}$ such that
\[N_{\epsilon}\fint_{C_r^+(Z)}W \ge r^{\varepsilon} \quad\text{for all }\;r\in (0,1).\]
Indeed, without loss of generality we assume that $Z$ is the origin.
Denote \[a_k=\fint_{C_{e^{-k}}^+}\log W.\]
By the triangle inequality and because $\log W\in \textsf{VMO}$, we find $|a_k-a_{k+1}| \to 0$ as $k \to \infty$.
Then for any $\varepsilon>0$, there is $k_0\in \bN$ such that for any $k\ge k_0$, $a_{k+1}-a_k\ge -\varepsilon$, so that
\[
a_k-a_{k_0} \ge -(k-k_0)\varepsilon\quad \forall k\ge k_0.
\]
For any small $r>0$, we can find $k\ge k_0$ such that $e^{-k-1}< r\le e^{-k}$.
Then we have $\fint_{C^+_r}\log W \ge C_\varepsilon+\varepsilon \log r$.
By the convexity of the exponential function,
\[
\fint_{C^+_r}W=\fint_{C^+_r}e^{\log W}
\ge \exp \left(\fint_{C^+_r}\log W\right)\gtrsim r^{\varepsilon}.\qedhere
\]
\end{proof}

The following Lemma is an extension of \cite[Theorem 1.2]{Es94} to the parabolic setting.
\begin{lemma}\label{lem8.0671}
Assume the coefficients $\mathbf{A}$ of $\mathcal P$ are $\mathsf{VMO_x}$ over $\bR^{n+1}$ and
let $W$ be the unique global non-negative adjoint solution associated to $\mathcal P$. Then, $\log W$ is in parabolic $\mathsf{VMO}(C_T^+)$, for all $T>0$, i.e.,
\begin{equation*}
\lim_{r\to 0+} \sup_{C_r^+(X)\subset C^+_T} \fint_{C_r^+(X)} |\log{W}(Y) - \left(\log{W}\right)_{X,r}| =0,
\end{equation*}
where
$$
\left(\log{W}\right)_{X,r} :=\fint_{C_r^+(X)} \log{W}.
$$

\end{lemma}
We recall that by the parabolic version of the John-Nirenberg inequality \cite{JohnNirenberg}, the later implies that any positive or negative power of $W$ is integrable over compact sets of $\bR^{n+1}$.
Lemma \ref{lem8.0671} follows from the following results. The second one is borrowed from \cite[Lemma 3]{Sarason1}.
\begin{lemma}\label{L: unlemilla}
Let $\mathcal P$ be as in Lemma \ref{lem8.0671}. Then, there are $N= N(\lambda, n, \sigma_\mathbf A^{\textsf x})$, $\alpha = \alpha (\lambda, n)$, $\delta = \delta(\lambda, n)$ and $\epsilon = \epsilon (\lambda, n)$ such that
\begin{equation*}%\label{eq3.52}
\left(\fint_{C^+_r(Z)} V^{\frac {n+1} n} \right)^{\frac{n}{n+1}}\le \left[1+N\left(r^{\alpha/2}+{\sigma_{\mathbf A}^{\textsf x}}(r^\varepsilon)^{\frac{\delta}{(n+1)(n+1-\delta)}}\right)\right]\fint_{C^+_r(Z)} V,
\end{equation*}
when $C^+_{r^\epsilon}(Z)\subset C^+_{1/2}$.
\end{lemma}
\begin{lemma}\label{L: lemilla2}
Let $(X,\mu)$ be a measure space, $\mu(X)=1$. Let $v\ge 0$ and assume that
\begin{equation*}
\int_Xv\,d\mu\int_Xv^{-1}\,d\mu\le 1+c^3,\quad c<1/2.
\end{equation*}
Then,
\begin{equation*}
\int_X\left |\log{v}-\int_X\log{v}\,d\mu\right |\,d\mu\le 16c.
\end{equation*}
\end{lemma}

\begin{proof}[Proof of Lemma \ref{lem8.0671}] Lemma \ref{L: unlemilla} and H\"older's inequality show that we can apply Lemma \ref{L: lemilla2} to $X= C^+_r(Z)\subset C^+_{1/2}$, $d\mu = V\,dx/V(C^+_r(Z))$ and $v= V^{\frac 1{n}}$, with $c(r)$ tending to zero as $r$ tends to zero. The turn out is that $\log{V}$ is in parabolic $\mathsf{VMO}(C^+_{1/2})$ with respect to the measure $V\,dx$. From \eqref{E: Mackenhoupt properety} and \eqref{E: aproximacion}, $V$ is a parabolic Muckenhoupt weight inside $C^+_1$ and the parabolic version of John-Nirenberg's inequality \cite{JohnNirenberg} shows that $\log{V}$ is in parabolic $\mathsf{VMO}(C^+_{1/2})$ with respect to the measure $dX$. Finally, the local H\"older continuity of $\tilde V$ in Lemma \ref{T: 3 Theorem 3.8}, \eqref{E: aproximacion} and the identity $\log{\tilde{V}} = \log{V}-\log{W}$ imply Lemma \ref{lem8.0671} with $T=1/2$. Other cases follow by translation and compactness.
\end{proof}

\begin{proof}[Proof of Lemma \ref{L: unlemilla}]
Let $W^\theta$ and $V^\theta$ denote respectively the global non-negative adjoint solution defined at the beginning of Section \ref{sec5.1} and the non-negative adjoint solution defined in \eqref{E: definition20} and associated to the parabolic operators
\begin{equation*}
\mathcal P^\theta u =\partial_tu-\tr \left(\left(\theta\mathbf A+\left(1-\theta\right)\bar{\mathbf A}\right)D^2u\right)\quad\text{with}\quad\bar{\mathbf A}(t) =\fint_{B_r(z)}\mathbf A(t,x)\,dx.
\end{equation*}
It is simple to derive from \cite[Theorem 2.1]{Kr07} and \eqref{eq13.53pder} that $\dot{V}^{\theta} =\frac d{d\theta}V^{\theta}$ is well defined in a weak sense,
\begin{equation}\label{E: jotramas}
\|V^\theta\|_{L^{\frac{n+1}n}(C^+_2)}+\|\dot V^\theta\|_{L^{\frac{n+1}n}(C^+_2)}\le N(\lambda, n, \omega_{\mathbf A}^{\textsf x})
\end{equation}
and
\begin{equation}\label{E: unaformulaespecial}
\int_{C^+_2} \dot{V}^\theta\mathcal P^\theta u = \int_{C^+_2} V^\theta \tr \left(\left(\mathbf A-\bar{\mathbf A}\right)D^2 u\right)+\int_{(0,4)\times\partial B_2}\left(\bar{\mathbf A}-\mathbf A\right) D u\cdot\nu_x
\end{equation}
for all $u$ in $W^{1,2}_{n+1}(C^+_2)\cap \mathring{\mathcal H}^1_{n+1}(C^+_2)$. We also recall that $V^\theta$ is a parabolic Muckenhoupt weight in the reverse H\"older class $\mathcal{B}_{\frac{n+1}{n}}(C^+_1)$ with a constant $N = N(\lambda, n)$, i.e.,
\begin{equation}\label{E: Mackenhoupt properety5}
\left(\fint_{C_r^+(Y)}(V^\theta)^{\frac{n+1}{n}}\right)^{\frac{n}{n+1}}\le N\fint_{C_r(Y)}V^\theta\quad \text{and}\quad  \ V^\theta(C_{2r}^+(Y))\le NV^\theta(C^+_r(Y))
\end{equation}
for all $C^+_{2r}(Y)\subset C_1^+$.
Now let $C^+_r(Z)$, $Z= (\tau,z)$,  be a fixed parabolic cube contained in $C^+_{1/2}$ and for $f\ge 0$ in $L^{n+1}(C^+_r(Z))$ with $\|f\|_{L^{n+1}(C^+_r(Z))}\le 1$, set
\begin{equation*}
\gamma(\theta) = \frac {|C_r^+|^{1/(n+1)}}{V^\theta(C^+_r(Z))}\int_{C^+_r(Z)} fV^\theta\quad\text{for}\ 0\le\theta\le1.
\end{equation*}
 Differentiating in $\theta$ gives
\begin{equation}\label{eq4.25}
\dot{\gamma}(\theta)=\frac {|C_r^+|^{1/(n+1)}}{V^\theta(C^+_r(Z))}\int_{C^+_r(Z)} \left(f-f_{Z,r}\right)\dot{V}^\theta\quad \text{with}\  f_{Z,r}=\frac{1}{V^\theta(C_r^+(Z))}\int_{C_r^+(Z)}fV^\theta,
\end{equation}
while \eqref{E: Mackenhoupt properety5} implies that
\begin{equation*}
|f_{Z,r}|\le \frac 1 {V^\theta(C^+_r(Z))}\|f\|_{L^{n+1}(C^+_r(Z))}\|V^\theta\|_{L^{\frac{n+1}{n}}(C^+_r(Z))}\le Nr^{-\frac{n+2}{n+1}}
\end{equation*}
and
\begin{equation}\label{eq4.18}
\|f-f_{Z,r}\|_{L^{n+1}(C^+_r(Z))}\le N
\end{equation}
with $N = N(\lambda,n)$.
Let $u^\theta$ in $W^{1,2}_{n+1}((\tau, 4)\times B_2)\cap \mathring{\mathcal H}^1_{n+1}((\tau,4)\times B_2)$ be the unique strong solution to
\begin{equation*}
\mathcal P^{\theta}u^\theta=1_{C^+_r(Z)}\left(f-f_{Z,r}\right)\ \text{in}\ (\tau,4)\times B_2,\quad u^\theta = 0\ \text{on}\ \partial_p^- ((\tau,4)\times B_2).
\end{equation*}
By the $W^{1,2}_{n+1}$-estimate for parabolic equations with parabolic $\mathsf{VMO_x}$ coefficients and \eqref{eq4.18}
\begin{equation}\label{eq4.22}
\|u^{\theta}\|_{W^{1,2}_{n+1}((\tau,4)\times B_2)}\le  N
\end{equation}
with $N=N(\lambda, n, \sigma_\mathbf A^{\textsf x})$.
It then follows from \eqref{E: unaformulaespecial} and \eqref{eq4.25} that
\begin{equation}	\label{E: unasuma}
\dot{\gamma}(\theta)=\frac {|C_r^+|^{1/(n+1)}}{V^\theta(C^+_r(Z))}
\left[ \int_{(\tau,T)\times B_2} \tr \left(\left(\mathbf A-\bar{\mathbf A}\right)D^2 u^{\theta}\right)V^\theta +\int_{(\tau,4)\times\partial B_2}\left(\bar{\mathbf A}-\mathbf A\right)\nabla u\cdot\nu_x\right].
\end{equation}
Now, with the purpose to estimate the above two integrals we need the following lemma.
\begin{lemma}\label{lem2.1}
There is $N=N(\lambda, n, \sigma_\mathbf A^{\textsf x})$ such that for all $R\ge 2r$ with $C^+_{R}(Z)\subset C_1^+$, we have
\begin{equation*}%\label{eq7.59}
\|u^\theta\|_{W^{1,2}_{n+1}((\tau,4)\times B_2\setminus C^+_{R}(Z))}\le N\left(\frac rR\right)^{\alpha-\frac{n+2}{n+1}}\frac {V^\theta(C^+_r(Z))}{V^\theta(C^+_{R}(Z))}.
\end{equation*}
\end{lemma}
\begin{proof}
We use a duality argument. For $\mathbf g$, $\vec F$, and $f$ in $C_c^\infty((\tau,4)\times B_2\setminus C^+_{R}(Z))$, let $v^\theta$ be the unique weak adjoint solution in $L^{\frac{n+1}{n}}((\tau,4)\times B_2)$ to
\begin{equation*}%\label{eq11.38}
{{\mathcal P}^\theta}^\ast v^\theta=\nabla^2\mathbf g +\dv \vec F+f\ \text{in}\ (\tau,4)\times B_2,\quad v^\theta= 0\ \text{on}\ \partial_p^+\left((\tau,4)\times B_2 \right),
\end{equation*}
so that
\begin{equation}\label{eq7.26}
\|v\|_{L^{\frac{n+1}{n}}((\tau,4)\times B_2)}\le N\left[ \|\mathbf g\|_{L^{\frac{n+1}{n}}((\tau,4)\times B_2)} + \|\vec F\|_{L^{\frac{n+1}{n}}((\tau,4)\times B_2)}+\|f\|_{L^{\frac{n+1}{n}}((\tau,4)\times B_2)}\right]
\end{equation}
with $N =N(\lambda, n,\sigma_\mathbf A^{\textsf x})$. Then, by duality,
\begin{equation*}%\label{eq7.29}
\int_{(\tau,4)\times B_2\setminus C^+_{R}(Z)} fu^{\theta} -\vec F\cdot D u^\theta + \tr (\mathbf g\,D^2u^\theta)
=\int_{C^+_r(Z)}\left(f-f_{Z,r}\right) v
\end{equation*}
and recalling the cancellation property of $1_{C^+_r(Z)}\left(f-f_{Z,r}\right)$ with respect to the measure $V^\theta\,dx$, we can write the last integral as
\begin{equation}\label{E: formulanara}
\int_{C^+_r(Z)}\left(f-f_{Z,r}\right)(\tilde v^\theta-\tilde v^\theta(Z))V^\theta
\end{equation}
with $\tilde v^\theta =v^\theta/V^\theta$. Because $v^\theta$ satisfies the homogeneous adjoint equation in $C^+_{R}(Z)$, it follows from Remark \ref{R: una nota} that
\begin{equation}\label{E: acota20}
\sup_{C^+_r(Z)}|\tilde v^\theta-\tilde v^\theta(Z)|\le N\left(\frac rR\right)^\alpha \frac{1}{V^\theta(C^+_R(Z))}\int_{C^+_R(Z)}|v^\theta |.
\end{equation}
Then, H\"older's inequality, \eqref{E: Mackenhoupt properety5}, \eqref{eq4.18}, \eqref{eq7.26}, \eqref{E: formulanara}, and \eqref{E: acota20} imply the Lemma.
\end{proof}

Then, proceeding with \eqref {E: unasuma}, by H\"older's inequality, the improved reverse H\"older's inequalities \cite{Stredulinsky1} satisfied by $V^\theta$, i.e., there is $0<\delta <n$, $\delta =\delta(\lambda, n)$ and $N = N(\lambda,n)$ such that
\begin{equation}\label{E: Mackenhoupt properetymejorada}
\left(\fint_{C^+_r(Y)} {V^\theta}^{\frac{n+1-\delta}{n-\delta}}\right)^{\frac{n-\delta}{n+1-\delta}}\le N\fint_{C^+_r(Y)}V^\theta\ \text{for all}\ C^+_{2r}(Y)\subset C^+_1
\end{equation}
and the doubling properties of $V^\theta$ \eqref{E: Mackenhoupt properety5}, we have
\begin{align}\label{eq5.06}
&\Abs{\int_{C^+_{2r}(Z)}\tr \left(\left(\mathbf A-\bar{\mathbf A}\right)D^2 u^{\theta}\right)V^\theta} \nonumber\\
&\qquad \qquad \le N\|D^2u^\theta\|_{L^{n+1}(C^+_{2r}(Z))}\|\mathbf A-\bar{\mathbf A}\|_{L^{\frac{(n+1)(n+1-\delta)}{\delta}}(C^+_{2r}(Z))}\|V^\theta\|_{L^{\frac{n+1-\delta}{n-\delta}}(C^+_{2r}(Z))}\nonumber\\
&\qquad \qquad \le Nr^{-\frac{n+2}{n+1}}\omega_{\mathbf A}^{\textsf x}(2r)^{^{\frac{\delta}{(n+1)(n+1-\delta)}}}V^\theta(C^+_{r}(Z)),
\end{align}
where we also used \eqref{eq4.22} in the last inequality. Next, for any integer $k\ge 1$ such that $C^+_{2^{k+1} r}(Z)\subset C_1^+$, we have
\begin{multline*}%\label{eq5.06c}
\Abs{\int_{C^+_{2^{k+1}r}(Z)\setminus C^+_{2^{k}r}(Z)} \tr \left(\left(\mathbf A-\bar{\mathbf A}\right)D^2 u^{\theta}\right)V^\theta}\\
\le N\|D^2u^\theta\|_{L^{n+1}(C^+_{2^{k+1}r}(Z)\setminus C^+_{2^{k}r}(Z))}\|\mathbf A-\bar{\mathbf A}\|_{L^{\frac{(n+1)(n+1-\delta)}{\delta}}(C^+_{2^{k+1}r}(Z))}\|V^\theta\|_{L^{\frac{n+1-\delta}{n-\delta}}(C^+_{2^{k+1}r}(Z))},
\end{multline*}
which together with Lemma \ref{lem2.1} and \eqref{E: Mackenhoupt properetymejorada} imply
\begin{multline}\label{eq7.53}
\Abs{\int_{C^+_{2^{k+1}r}(Z)\setminus C^+_{2^{k}r}(Z)}\tr \left(\left(\mathbf A-\bar{\mathbf A}\right)D^2 u^{\theta}\right)V^\theta}\\
\le N2^{-\alpha k}r^{-\frac{n+2}{n+1}}\,((k+2)\sigma_{\mathbf A}^{\textsf x}(2^{k+1}r))^{\frac{\delta}{(n+1)(n+1-\delta)}}V^\theta(C^+_{r}(Z)),
\end{multline}
where we used the fact that $\mathbf A$ and $\bar{\mathbf A}$ and bounded, and
\[
\fint_{C^+_{2^{k+1}r}(Z)}|\mathbf A-\bar{\mathbf A}|
\le N\sum_{j=0}^{k+1}\omega_{\mathbf A}^{\textsf x}(2^{j}r)\le N(k+2)\sigma_{\mathbf A}^{\textsf x}(2^{k+1}r).
\]
Finally, by Lemma \ref{lem2.1} with $R\ge 2r$, \eqref{E: jotramas}, \eqref{E: Mackenhoupt properety5}, \eqref{E: 2 condidicndeunicidad} and \eqref{E: aproximacion}, we have
\begin{equation}\label{eq5.06d}
\Abs{\int_{(\tau, 4)\times B_2\setminus C^+_{R}(Z)} \tr \left(\left(\mathbf A-\bar{\mathbf A}\right)D^2 u^{\theta}\right)V^\theta}
\le N\left(\frac rR\right)^{\alpha-\frac{n+2}{n+1}}\frac{V^\theta(C^+_r(Z))}{V^\theta(C^+_R(Z))}\, .
\end{equation}
Also, by the trace theorem, Lemma \ref{lem2.1} with $R=1/2$, \eqref{E: Mackenhoupt properety5}, \eqref{E: 2 condidicndeunicidad} and \eqref{E: aproximacion}
\begin{equation*}
\Abs{\int_{(\tau,4)\times\partial B_2}\left(\bar{\mathbf A}-\mathbf A\right) D u \cdot\nu_x} \le Nr^{\alpha-\frac{n+2}{n+1}}V^\theta(C^+_r(Z)).
\end{equation*}
Write then
\begin{multline}\label{E: 222}
\int_{(\tau,T)\times B_2} \tr \left(\left(\mathbf A-\bar{\mathbf A}\right)D^2 u^{\theta}\right)V^\theta\\
= \left(\int_{C^+_{2r}(Z)}+\sum_{k=1}^{k_0}
\int_{C^+_{2^{k+1}r}(Z)\setminus C^+_{2^{k}r}(Z)}+\int_{(\tau, 4)\times B_2\setminus C^+_{2^{k_0+1}r}(Z)}\right)\
\tr \left(\left(\mathbf A-\bar{\mathbf A}\right)D^2 u^{\theta}\right)V^\theta.
\end{multline}
Recall that the doubling properties of $V^\theta$ and \eqref{E: 2 condidicndeunicidad} imply that $V^\theta (C^+_R(Z))\ge R^\eta$, $0 < R\le 1$, for some $\eta = \eta(\lambda, n)$. Then, choose $\epsilon >0$ with $\frac\alpha 2 + \left(\frac{n+2}{n+1}-\alpha-\eta\right)\epsilon\ge  0$ and $k_0\ge 1$ with  $r^\epsilon/2\le 2^{k_0+1}r\le r^\epsilon$, to find from \eqref{E: unasuma}, \eqref{E: 222}, \eqref{eq5.06}, \eqref{eq7.53} and \eqref{eq5.06d} that
\begin{equation*}
|\dot{\gamma}(\theta)|\le N\left[r^{\alpha/2}+\sigma_\mathbf A^{\textsf x}(r^\epsilon)^{^{\frac{\delta}{(n+1)(n+1-\delta)}}}\right].
\end{equation*}
Note that by duality
\begin{equation*}
\left(\fint_{C^+_r(Z)} V^{\frac{n+1}{n}} \right)^{n/(n+1)}\left(\fint_{C^+_r(Z)} V\right)^{-1}=
\sup_{\|f\|_{L^{n+1}(C^+_r(Z))}=1}{\frac {|C_r^+|^{1/(n+1)}}{V^\theta(C^+_r(Z))}\int_{C^+_r(Z)} fV\, .}
\end{equation*}
Also, because $V^0$ is identically equal to $1$, we have $\gamma(0)\le 1$. Then, Lemma \ref{L: unlemilla} follows from the fundamental theorem of Calculus and the fact that $V^1=V$.
\end{proof}
\begin{remark}\label{R:remrk10}
Lemma \ref{lem8.06} combined with \eqref{E:4 condicionea adicionales} and \cite[Theorem 1.3]{Es00} show that the following Gaussian bounds hold for the fundamental solution $G(X,Y)$ of $\mathcal P$, when $\mathbf A$ is in $\mathsf{DMO_x}$ over $\bR^{n+1}$: there are
$N_1= N_1(\lambda, n, \omega_{\mathbf A}^{\textsf x})$ and $N_2=N_2(\lambda,n)$ such for all $X, Y$ in $\bR^{n+1}$ with $t > s$, the following holds
\begin{equation*}
N_1^{-1}\left(t-s\right)^{-n/2}e^{-N_2\left(t-s\right)^{1/2}-N_2|x-y|^2/(t-s)}\le G(X,Y)\le N_1\left(t-s\right)^{-n/2}e^{N_2\left(t-s\right)^{1/2}-|x-y|^2/N_2(t-s)}.
\end{equation*}

It is interesting to compare the above two sided Gaussian bounds  with the results proved in \cite{FS1984} and \cite[Theorem 1.2]{Es00}, the main result in \cite{FabesKenig1} and \cite[Theorem 1]{FJK1985}, along with their somehow hidden analysis of the mutual absolute continuity of the measures $G(t,x,\tau,dy)$ and $dy$ over $\{\tau\}\times\bR^n$. Thanks to \cite[Theorems 1.2 and 1.3]{Es00}), the analysis is equivalent to the study of the mutual absolute continuity of Lebesgue measure and the measure $W(\tau,dy)$ over $\bR^n$, where
the measure $W(\tau,dy)$ is the restriction of $W(Y)$ to the hyperplane $s=\tau$. See (b) and \eqref{E: condidebil} on Page~\pageref{E: condidebil} for its existence and relation with $W(Y)$.

Finally, the last claim can be justified via standard methods, when $\mathbf A$ is in $\mathsf{VMO_x}$ over $\bR^{n+1}$, though in fact, the same result holds when $\mathbf A$ is only measurable over $\bR^{n+1}$ and satisfies \eqref{parabolicity}. Regarding the last statement, one must be careful because of issues related with the possible lack of uniqueness of the global adjoint solution $W$ and of the ``parabolic''  measures $G(t,x,\tau,dy)$, which can arise in those cases (See \cite[Theorem 1.4 and Remark 4.1]{{Es00}} and \cite{Na97, Na97b, Safonov99, Kr92}).
\end{remark}
%-----------------------------------------------------------------------------}-%

\end{document}